\documentclass{gtpart} 
\usepackage{subfig}

\usepackage[bbgreekl]{mathbbol}

\usepackage{graphicx}

\usepackage{amssymb}
\usepackage{epstopdf}
\usepackage{tikz}
\usepackage[all,cmtip]{xy}
\usepackage{color}

\newcommand{\iinfty}{{\mathchoice
{\begin{minipage}{.15in}\includegraphics[width=.15in]{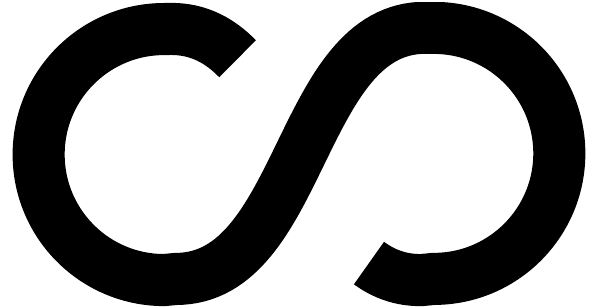}\end{minipage}}
{\begin{minipage}{.10in}\includegraphics[width=.10in]{infty2.pdf}\end{minipage}}
{\begin{minipage}{.08in}\includegraphics[width=.08in]{infty2.pdf}\end{minipage}}
{\begin{minipage}{.08in}\includegraphics[width=.08in]{infty2.pdf}\end{minipage}}
}}
\newcommand{\tree}[3]{\text{\Large {$ {\text{\normalsize $ #1$}}-\!\!\!<^{#2}_{#3}$}}}
\newcommand{\dash}{\!-\!\!\!\!\!-\,} 

\renewcommand{\int}{\operatorname{int}} 
 
\newcommand{\Ker}{\operatorname{Ker}} 
\newcommand{\Cok}{\operatorname{Cok}} 
\newcommand{\id}{\operatorname{id}}

\newcommand{\Arf}{\mathrm{Arf}}

\newcommand{\im}{\operatorname{Im}}

\newcommand\W{{\sf W}} 
 
\newcommand\sL{{\sf L}}
\newcommand\sD{{\sf D}}
\newcommand\sK{{\sf K}}
\newcommand{\z}{\mathbb Z_2}

\newcommand{\N}{\mathbb{N}} 
\newcommand{\bG}{\mathbb{G}} 
\newcommand{\bW}{\mathbb{W}} 
\newcommand{\bL}{\mathbb{L}}

\newcommand{\cT}{\mathcal{T}} 
\newcommand{\cV}{\mathcal{V}} 
\newcommand{\cW}{\mathcal{W}} 

\newcommand{\SL}{\operatorname{SL}}

\newcommand{\sra}{\twoheadrightarrow}

\newcommand{\sW}{\W}

\def\theenumi{\roman{enumi}}

\title{Whitney tower concordance of classical links} 

\author{James Conant}
\givenname{James}
\surname{Conant}
\address{Dept. of Mathematics, University of Tennessee, Knoxville, TN, 37996}
\email{jconant@math.utk.edu}
\urladdr{}

\author{Rob Schneiderman}
\givenname{Rob}
\surname{Schneiderman}
\address{Dept. of Mathematics and Computer Science, Lehman College, City University of New York, Bronx, NY 10468}
\email{robert.schneiderman@lehman.cuny.edu}
\urladdr{}

\author{Peter Teichner}
\givenname{Peter}
\surname{Teichner}
\address{Dept. of Mathematics, University of California, Berkeley, CA and} 
\address{Max-Planck Institut f\"ur Mathematik, Bonn, Germany} 
\email{teichner@mac.com}
\urladdr{}

\keyword{Whitney towers, gropes, link 
concordance, trees, higher-order Arf invariants, higher-order Sato-Levine invariants, twisted Whitney disks}
\subject{primary}{msc2000}{57M25}
\subject{secondary}{msc2000}{57N10}

\arxivreference{}
\arxivpassword{}

\volumenumber{}
\issuenumber{}
\publicationyear{}
\papernumber{}
\startpage{}
\endpage{}
\doi{}
\MR{}
\Zbl{}
\received{}
\revised{}
\accepted{}
\published{}
\publishedonline{}
\proposed{}
\seconded{}
\corresponding{}
\editor{}
\version{}

\newtheorem{thm}{Theorem}[section]    
\newtheorem{lem}[thm]{Lemma}          
\newtheorem{prop}[thm]{Proposition}
\newtheorem{cor}[thm]{Corollary}
\newtheorem{conj}[thm]{Conjecture}
\theoremstyle{definition}
\newtheorem{defn}[thm]{Definition}    
\newtheorem{rem}{Remark} [section]            
\begin{document}

\begin{abstract} 
This paper computes \emph{Whitney tower} filtrations of classical links. Whitney towers consist of iterated stages of Whitney disks and allow a tree-valued intersection theory, showing that the associated graded quotients of the filtration are finitely generated abelian groups. 
\emph{Twisted Whitney towers} are studied and a new quadratic refinement of the intersection theory is introduced, measuring Whitney disk framing obstructions.  
It is shown that the filtrations are completely classified by Milnor invariants together with new \emph{higher-order Sato-Levine} and \emph{higher-order Arf invariants}, which are obstructions to framing a twisted Whitney tower in the $4$--ball bounded by a link in the $3$--sphere. 
Applications include computation of the \emph{grope filtration}, and new geometric characterizations of Milnor's link invariants.

\end{abstract}

\keywords{Whitney towers, gropes, link 
concordance, trees, higher-order Arf invariants, higher-order Sato-Levine invariants, twisted Whitney disks} 

\maketitle

\section{Introduction}\label{sec:intro}
The general failure of the {\em Whitney move} is one reason why $4$-dimensional manifolds are notoriously difficult to understand. A successful Whitney move is shown in Figure~\ref{fig:canceling-pair-and-whitney-disk-and-whitney-move}:
\begin{figure}[h]
        \centerline{\includegraphics[scale=.35]{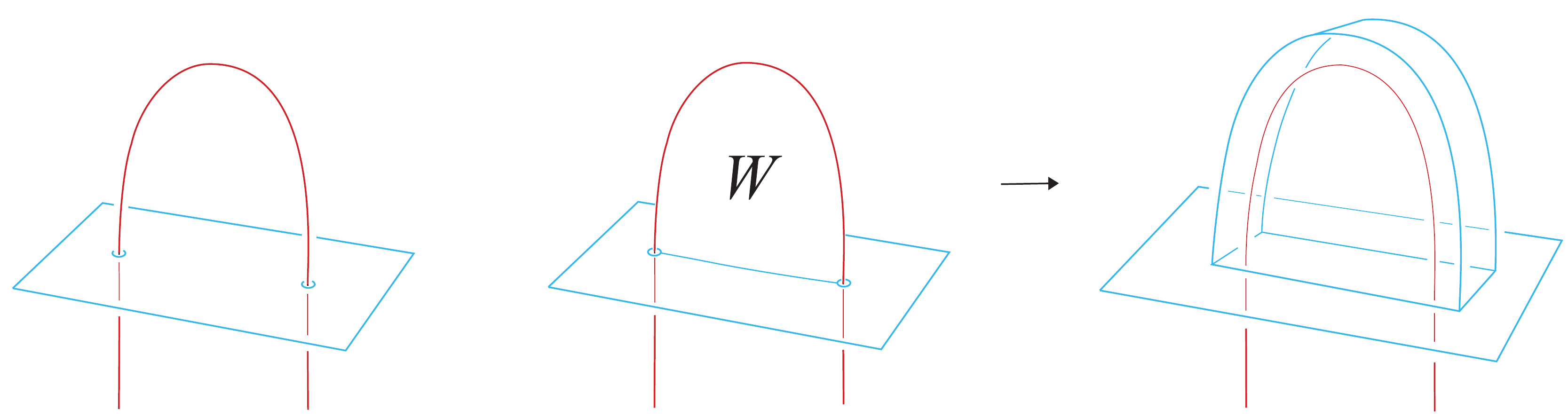}}
        \caption{Left:~A canceling pair of transverse intersections between two local sheets of surfaces in a $3$-dimensional slice of $4$--space. The translucent horizontal blue sheet appears entirely in this 3-dimensional `present', and the red sheet appears as an arc which is assumed to extend into `past' and `future'.
         Middle:~A Whitney disk $W$ pairing the intersections. Right:~A Whitney move guided by $W$ eliminates the intersection pair, without creating any new intersections.}
        \label{fig:canceling-pair-and-whitney-disk-and-whitney-move}
\end{figure}

In higher dimensions, this move is the key to Whitney's strong embedding theorem \cite{Wh} as well as the s-cobordism theorem and the surgery exact sequence. In each case, a pair of intersection points between two submanifolds is removed by a homotopy along an embedded Whitney disk $W$ as in Figure~\ref{fig:canceling-pair-and-whitney-disk-and-whitney-move}. Whitney disks can be found by controlling the intersections between the submanifolds algebraically over the fundamental group of the ambient manifold, and in dimensions greater than four can be assumed by general position to be embedded and disjoint from the relevant submanifolds. 

In four dimensions, generic intersections between Whitney disks and surface sheets can obstruct a successful Whitney move: 
Figure~\ref{fig:unsuccessful-W-move-higher-order-ints}(a) shows how 
such an intersection point leads to an \emph{un}successful Whitney move. 
\begin{figure}[h]
\subfloat[]{\includegraphics[width=3.2in]{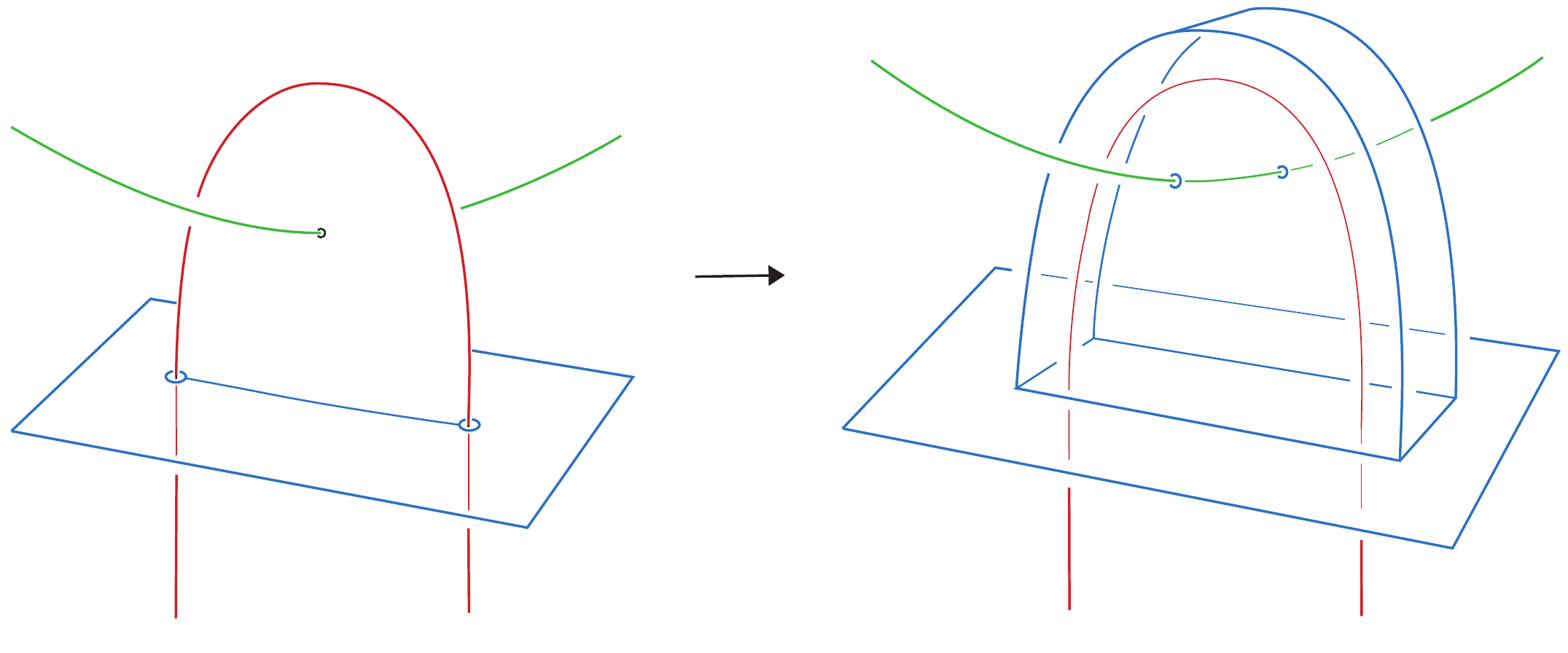}}\hspace{3em}
\subfloat[]{\includegraphics[width=1.7in]{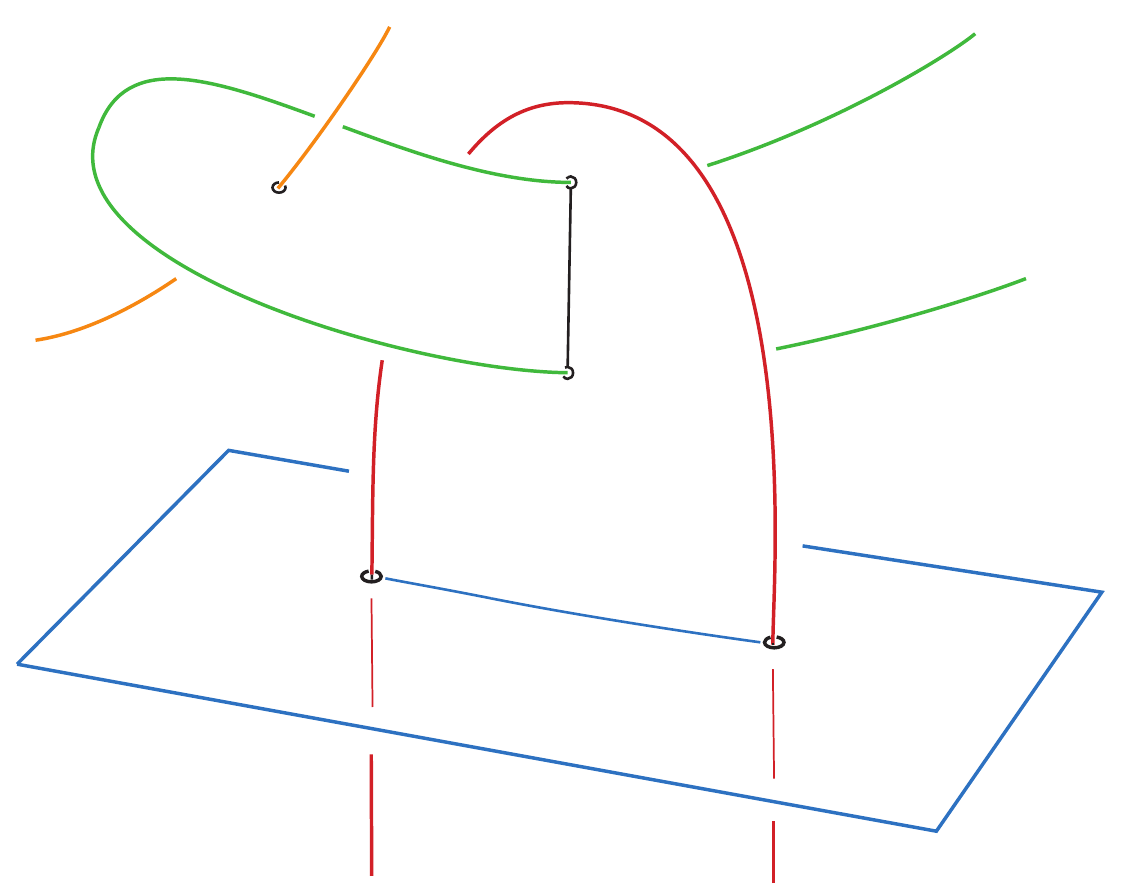}}
\caption{(a) This Whitney move eliminates the previous pair of intersections but creates a new pair of intersections between the 
translucent blue sheet and the sheet described by the green arc. (b) Higher-order intersections and Whitney disks.
All arcs are assumed to extend into past and future, describing local sheets of surfaces in a $4$--ball.}\label{fig:unsuccessful-W-move-higher-order-ints}
\end{figure}

In a $4$--manifold, a Whitney disk $W$ can always be made to be embedded (and \emph{framed}, see section~\ref{subsec:twisted-w-disks}) at the cost of creating
intersections with the surface sheets paired by $W$, but the converse is not always possible (as explained below in Section~\ref{sec:w-towers}). It is natural to try to eliminate such {\em higher-order} intersections using (higher-order) Whitney moves, and this attempt leads to the notion of a \emph{Whitney tower} (Definition~\ref{def:framed-tower}), constructed on immersed surfaces in a $4$--manifold by pairing up as many intersection points as possible with iterated Whitney disks (see Figure~\ref{fig:W-tower-and-trees-and-split}(a)). 

\begin{figure}[h]
\subfloat[]{\includegraphics[width=2in]{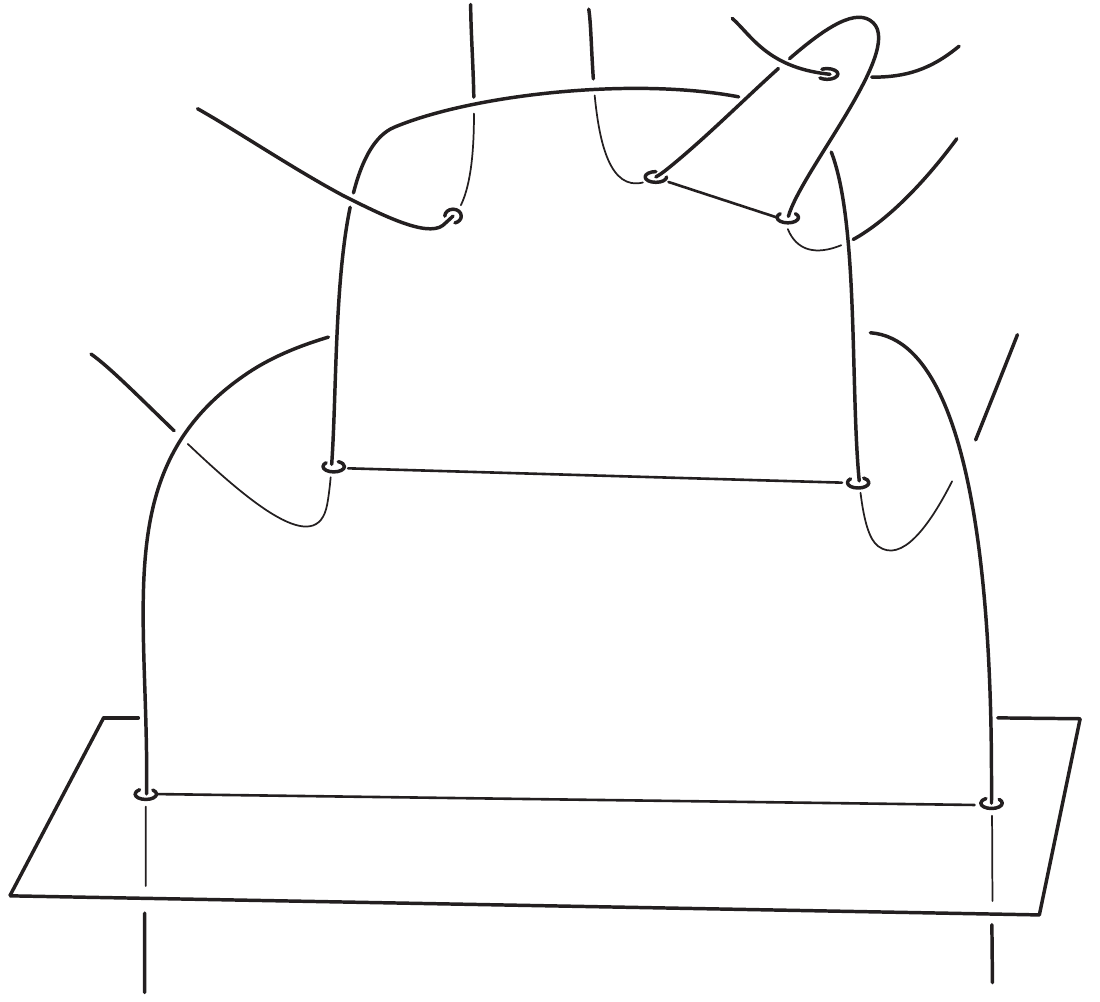}}\hspace{6em}
\subfloat[]{\includegraphics[width=2in]{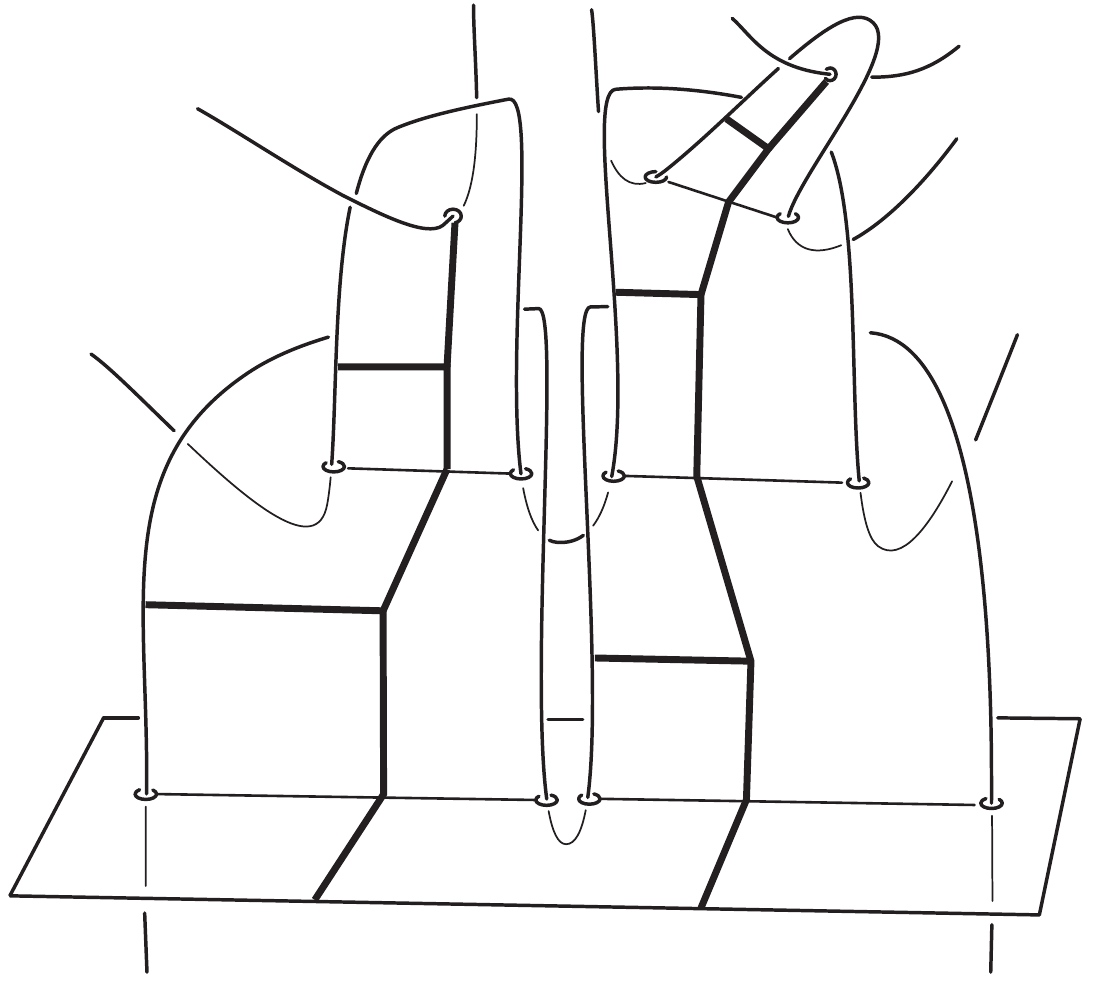}}
\caption{(a) Part of a Whitney tower $\cW$. (b) The unpaired intersections determine trivalent trees, 
and $\cW$ can be `split' so that all singularities are contained in neighborhoods of embeddings of these trees.}\label{fig:W-tower-and-trees-and-split}
\end{figure}

A Whitney tower has a fundamental complexity called {\em order}, which roughly corresponds to the
number of layers of Whitney disks,
and in \cite{CST,ST2} we introduced an accompanying obstruction theory that will be used in this paper to study link concordance. The main idea is that the chaos of multiple Whitney disks and intersection points
can be organized by associating certain unitrivalent trees to the unpaired intersections 
(Figure~\ref{fig:W-tower-and-trees-and-split}), and summing these trees defines an intersection invariant whose vanishing is sufficient to raise the order of the Whitney tower by finding another layer of Whitney disks. By taking the values of this invariant in an appropriate target group, we will also be able to measure
the failure of the Whitney move by determining when the order of a Whitney tower can \emph{not} be raised.

One reason that classical links are especially relevant to Whitney towers comes from the observation that the boundaries of the embedded disk-sheets (blue, red and green) in the $4$--ball described by Figure~\ref{fig:unsuccessful-W-move-higher-order-ints}(a) form the Borromean rings in the boundary the $3$--sphere (see e.g.~Figure~\ref{fig:Borromean-Bing-Hopf} in Section~\ref{sec:realization-maps}). Figure~\ref{fig:W-move-preserves-tree} shows how the trivalent tree associated to the unpaired intersection in the Whitney disk is preserved under the Whitney move, hinting at the fact that this tree represents an obstruction to the existence of disjointly embedded disks bounded by the Borromean rings.
\begin{figure}[h]
\centerline{\includegraphics[scale=.45]{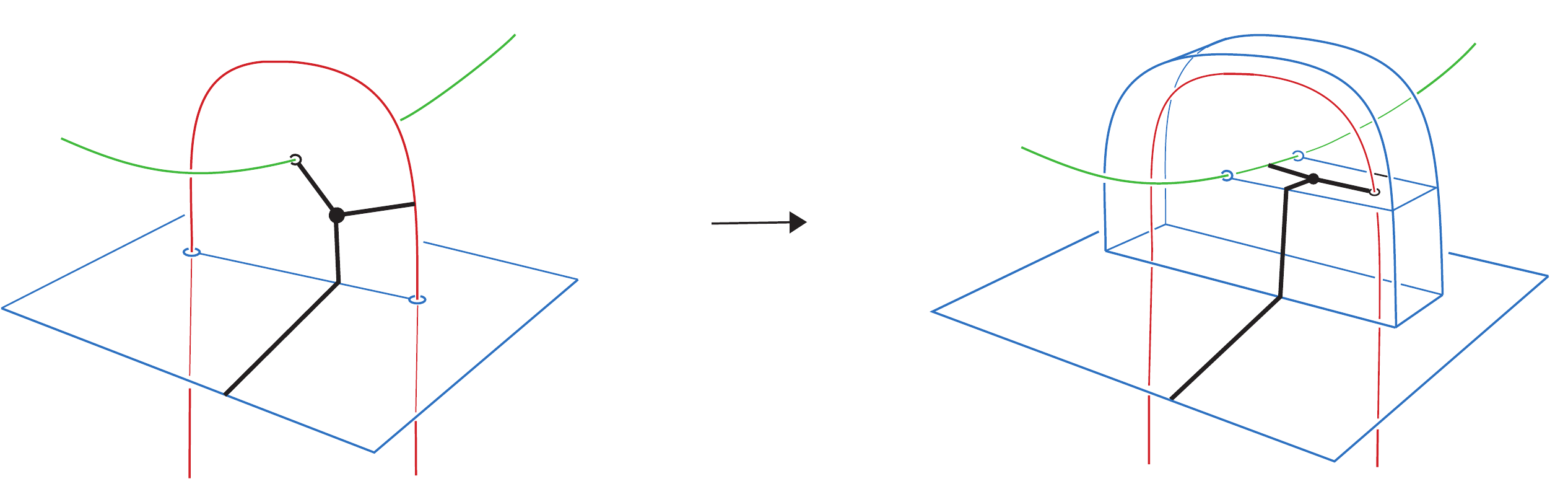}}
         \caption{A Whitney move preserves the associated tree.}
         \label{fig:W-move-preserves-tree}
\end{figure}

In fact, 
any Whitney tower can be `split' as in Figure~\ref{fig:W-tower-and-trees-and-split}(b) so that all singularities are contained in $4$--ball neighborhoods of the unitrivalent trees,
each containing embedded disks bounded by links which are iterated Bing-doubles of unknots. 
While the local ``tree-preserving'' property of Figure~\ref{fig:W-move-preserves-tree} holds for a Whitney move on any Whitney disk containing an unpaired intersection, it turns out that if a Whitney disk $W$ contains a boundary arc of a higher-order Whitney disk then a Whitney move on $W$ locally converts the tree to
two new trees, expressing a Jacobi identity (also called an IHX relation) \cite[Lem.7.2]{S3}. Thus, interpreting these trees as well-defined obstructions to raising the order of a Whitney tower requires some care, see Definition~\ref{def:Tau}.

On the other hand, if a link bounds immersed disks supporting a Whitney tower with \emph{no} unpaired intersections, then doing Whitney moves on all the Whitney disks
leads to disjointly embedded \emph{slice disks} bounded by the link. Motivated by the notion that a Whitney tower of larger order is in some sense a ``better approximation'' of slice disks,
the main goal of this paper is to provide an answer to the following question for any given $n$: ``Which links in the $3$--sphere bound an order $n$ Whitney tower in the $4$--ball?''

The rest of this introduction will mostly be concerned with outlining the answer to this question which is roughly summarized by the following theorem:
\begin {thm}\label{thm:intro-classification}
A link bounds a Whitney tower of order $n$ if and only if its Milnor invariants, higher-order Sato-Levine invariants and higher-order Arf invariants vanish up to order $n$. Compare Corollary~\ref{cor:mu-sl-arf-classify}.
\end {thm}

We work in the smooth oriented category, with orientations frequently suppressed from notation. Throughout this paper identical statements hold in the locally flat topological category, as explained in Remark~\ref{rem:locally-flat-and-smooth}. 


\subsection{The Whitney tower filtration of classical links}\label{subsec:intro-W-tower-filtration}

Referring to Section~\ref{sec:w-towers} for a precise definition of Whitney tower, including
detailed discussion of \emph{framing} requirements,
we consider the (framed) \emph{Whitney tower filtration}
$$
\dots \subseteq \bW_{3} \subseteq \bW_{2} \subseteq \bW_{1} \subseteq  \bW_{0} \subseteq  \bL
$$
on the set $\bL=\bL(m)$ of $m$-component framed links in $S^3$. Here $\bW_n=\bW_n(m)$ is the subset of those framed links that bound 
order $n$ (framed) Whitney towers in $B^4$. In section~\ref{subsec:intro-compare-other-towers} we compare this filtration to other known iterated disk constructions. 

Throughout this paper the number $m$ of link components
will frequently be suppressed from notation because it is fixed in most constructions.

The intersection of all $\bW_n$ contains all slice links because a $2$--disk is a Whitney tower of arbitrarily large order. In fact, this filtration factors through link concordance; and we shall use this fact implicitly in various places. 



Whitney towers built on immersed annuli connecting link components in $S^3\times I$ induce equivalence relations of {\em Whitney tower concordance} on links. The quotient $\W_n$ of 
$\bW_n$ modulo the equivalence relation of Whitney tower concordance of order $n+1$ is the {\em associated graded} of our filtration in the sense that $L\in \bW_{n+1}$ if and only if $L\in\bW_n$ and $[L]=0\in\W_n$. (We will show in Section~\ref{sec:realization-maps} that connected sum leads to a well defined group structure on $\W_n$, so $0$ corresponds to the unlink.)



As a first step towards our goal of describing a classification of this filtration and the associated $\W_n$, we recall (e.g.~from \cite{CST,CST3,L3,ST2}) a combinatorially defined group which is a natural target for the 
intersection invariant associated to the obstruction theory for Whitney towers.

\begin{defn}\label{def:Tau}
 In this paper, a {\em tree} will always refer to a finite oriented unitrivalent tree, where the {\em orientation} of a tree is given by cyclic orderings of the adjacent edges around each trivalent vertex. The \emph{order} of a tree is the number of trivalent vertices.
Univalent vertices will usually be labeled from the set $\{1,2,3,\ldots,m\}$ indexing the link components, and we consider trees up to isomorphisms preserving these labelings.
Define $\cT=\cT(m)$ to be the free abelian group on such trees, modulo the antisymmetry (AS) and Jacobi (IHX) relations shown in Figure~\ref{fig:ASandIHXtree-relations}. 
\begin{figure}[h]
\centerline{\includegraphics[scale=.65]{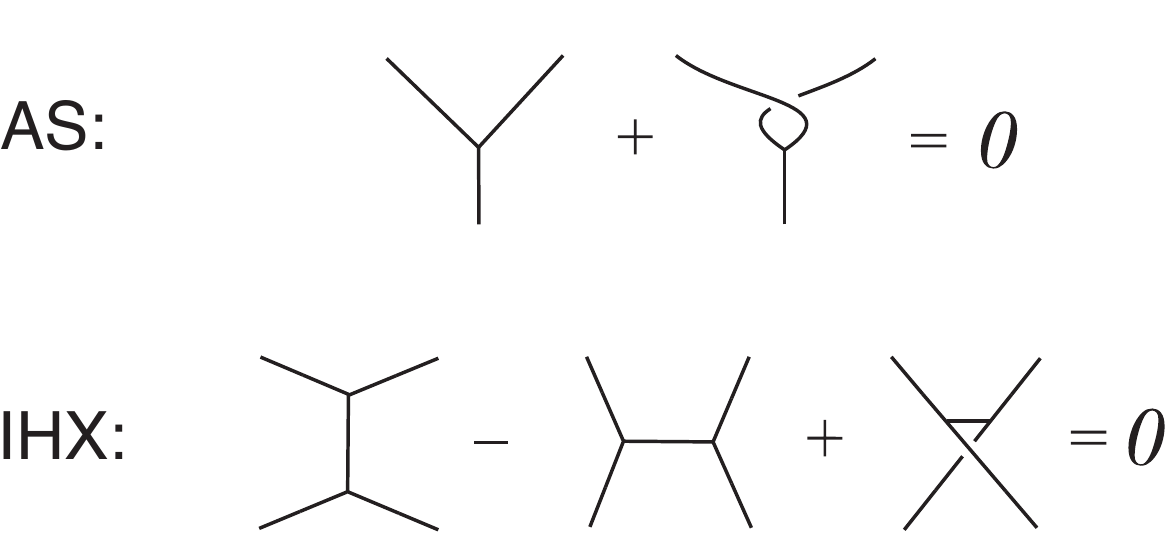}}
         \caption{Local pictures of the \emph{antisymmetry} (AS) and \emph{Jacobi} (IHX) relations
         in $\cT$. Here all trivalent orientations are induced from a fixed orientation of the plane, and univalent vertices possibly extend to subtrees which are fixed in each equation.}
         \label{fig:ASandIHXtree-relations}
\end{figure}

Since the AS and IHX relations are homogeneous with respect to order, $\cT$ inherits a grading $\cT=\oplus_n\cT_n$, where $\cT_n=\cT_n(m)$ is the free abelian group on order $n$ trees, modulo AS and IHX relations.
\end{defn}

As recalled below in Section~\ref{sec:w-towers}, the Whitney tower obstruction theory \cite{CST,ST2} assigns to each order~$n$ Whitney tower $\cW$ an 
order $n$ intersection invariant 
$\tau_n(\cW)\in\cT_n$, which is defined by summing the trees pictured in Figure~\ref{fig:W-tower-and-trees-and-split}(b)
(see Figure~\ref{Cochran-R2-map-fig-Color3} for an explicit example in the $4$--ball).  
The tree orientations are induced by Whitney disk orientations via a convention that corresponds to the AS relations (section~\ref{subsec:w-tower-orientations}), and  
the IHX relations can be realized geometrically by controlled maneuvers on Whitney towers as described in \cite{CST,S2}. 
It follows from the obstruction theory that a link bounds a Whitney tower $\cW$ of order~$n$ with $\tau_n(\cW)=0$ 
if and only if it bounds a Whitney tower of order $n+1$. This fact is the essential ingredient in the proof of the following theorem
(see Section~\ref{sec:realization-maps}):
\begin{thm} \label{thm:R-onto-W} 
The sets $\W_n$ are finitely generated abelian groups under the (well-defined) operation of connected sum $\#$ and there are epimorphisms 
$R_n : \cT_n \sra\W_n$.
\end{thm} 
These {\em realization maps} $R_n$ are defined similarly to Cochran's iterated Bing-doubling construction for realizing Milnor invariants \cite{C}, and are equivalent to `simple clasper surgery along trees' in the sense of Goussarov and Habiro (see section~\ref{subsec:realization-maps}, and Figure~\ref{Cochran-R2-map-fig-Color3} for an example).
\begin{figure}[h]
\centerline{\includegraphics[scale=.32]{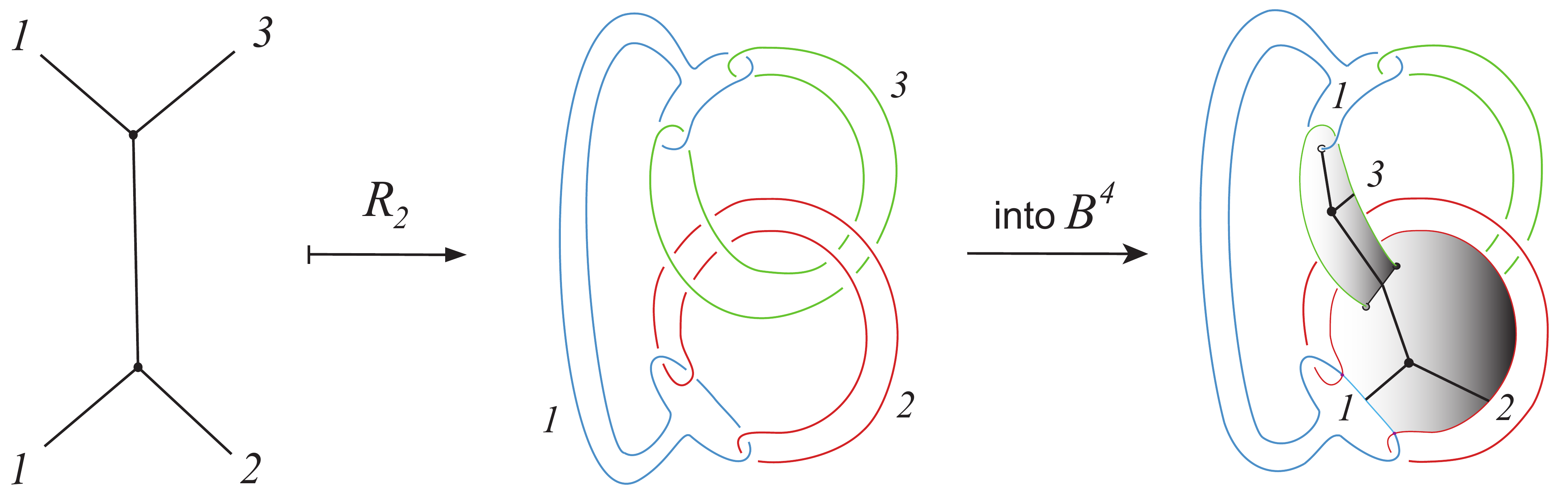}}
\caption{The realization map $R_2$ sends the tree $t$ on the left to the link $L\subset S^3$ shown in the middle. 
The trace of a null-homotopy of $L$ described by a pair of crossing-changes between the blue component 1 and the red component 2 supports an order $2$ Whitney tower $\cW\subset B^4$ bounded by $L$, with $\tau_2(\cW)=t$, as shown on the right. (Pushing further into $B^4$ would show a 3-component unlink bounding disjointly embedded disks).}
         \label{Cochran-R2-map-fig-Color3}
\end{figure}

Theorem~4 of \cite{ST2} implies that $R_n$ is \emph{rationally} an isomorphism for all $n$, and from this a formula for the free ranks of the groups $\W_n$ can be given. Until now we didn't know much about torsion phenomena, and the following fact comes as a huge surprise: 

\begin{thm} \label{thm:intro-even-order-R}
In all even orders, the realization maps $R_{2k}: \cT_{2k}\to \W_{2k}$ are isomorphisms and $\W_{2k}$ are free abelian groups of known rank, detected by Milnor invariants.
\end{thm}
As explained in section~\ref{sec:twisted-to-framed-classification} (Theorem~\ref{thm:even-R-and-mu}), this result follows from the relationship between the Whitney tower intersection invariant $\tau$ and Milnor invariants \cite{CST2, ST2}, together with our proof in \cite{CST3} of a combinatorial conjecture of J. Levine \cite{L2}, which also implies that $\cT_{2k}$ is a free abelian group of known rank \cite[Thm.1.5]{CST3}.

The affirmation of Levine's conjecture also implies that the torsion in $\cT_{2k-1}$ is generated by symmetric trees of the form $\tree{i}{J}{J}$, where $J$ is a subtree of order~$k-1$, and $i$ is a univalent vertex label (see \cite[Cor.1.2]{CST3}). These trees are actually 2-torsion by the antisymmetry relation and hence all torsion in $\cT$ is 2-torsion. The next result shows that a large part of this 2-torsion actually maps trivially to $\W_{2k-1}$. 

\begin{thm} \label{thm:reduced R}
The realization maps $R_{2k-1}$ factor through a quotient $\widetilde\cT_{2k-1}$ of $\cT_{2k-1}$.
\end{thm}
Theorem~\ref{thm:reduced R} is proved in Section~\ref{sec:reduced-groups-obstruction-theory}, where it is also shown that our Whitney tower obstruction theory descends to these \emph{reduced groups} $\widetilde\cT_{2k-1}$:
\begin{defn}\label{def:reduced-tree-group}
Let $\widetilde{\mathcal T}_{2k-1}:={\mathcal T}_{2k-1}/\im\Delta_{2k-1}$, where $\Delta_{2k-1}\colon \mathcal{T}_{k-1}\to\mathcal{T}_{2k-1}$ is defined on generators $t$ of order $k-1$ as follows. For any univalent vertex $v$ of $t$, denote by $\ell(v)$ the label of $v$, and write $t = \ell(v)-\!\!\!- T_v(t)$. Then we get a 2-torsion element of $\cT_{2k-1}$ 
defined by
\[
\Delta_{2k-1}(t):=\sum_v\,\tree{\ell(v)}{T_v(t)}{T_v(t)}
\]
where the sum is over all univalent vertices $v$ of $t$.
\end{defn}
\begin{figure}[h]
\centerline{\includegraphics[width=125mm]{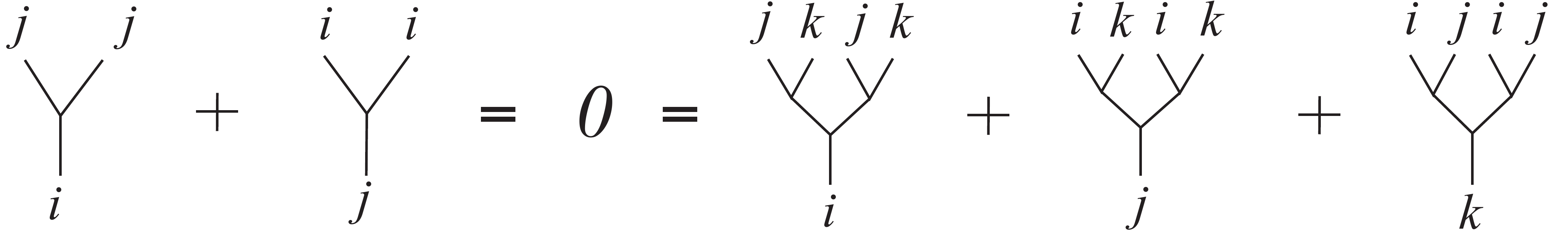}}
         \caption{The framing relations in orders $1$ and $3$.}
         \label{fig:framing-relations}
\end{figure}

We refer to the relations $\Delta_{2k-1}(t)=0$ as {\em framing relations} because they correspond to the image of \emph{twisted} IHX relations in a \emph{twisted Whitney tower} via a conversion to a framed Whitney tower, as explained in Section~\ref{sec:reduced-groups-obstruction-theory}.


Conjecturally, all odd order {\em reduced realization maps} $\widetilde R_{2k-1}: \widetilde\cT_{2k-1}\,{\to}\, \W_{2k-1}$ are isomorphisms, and the following theorem confirms this in half of the cases:

\begin{thm} \label{thm:intro-odd-R-isos}
The reduced realization maps $\widetilde R_{4k-1}$ are isomorphisms and the torsion of $\W_{4k-1}$ is a $\z$-vector space of known dimension, detected by higher order Sato-Levine invariants. 
\end{thm}

The \emph{higher-order Sato-Levine invariants} are certain projections
of Milnor invariants, shifted down one order. They represent obstructions to framing a twisted Whitney tower, as explained in 
section~\ref{sec:twisted-to-framed-classification}. 
To outline the proof of Theorem~\ref{thm:intro-odd-R-isos}, we next introduce the \emph{twisted Whitney tower filtration}, and explain how {\em higher-order Arf invariants} play a role in completing the classifications of both the twisted and framed filtrations.
Theorem~\ref{thm:intro-odd-R-isos} will follow from Theorem~\ref{thm:intro-framed-3-quarter-classification} in section~\ref{sec:twisted-to-framed-classification}.


\subsection{The twisted Whitney tower filtration of classical links}\label{subsec:intro-twisted-filtration}
As illustrated in Figure~\ref{fig:canceling-pair-and-whitney-disk-and-whitney-move} and detailed below in section~\ref{subsec:twisted-w-disks}, a successful Whitney move requires the existence of disjoint parallel copies of the Whitney disk (which extend a canonical normal section over the boundary). Such a Whitney disk is said to be \emph{framed}, 
and this condition is required for all Whitney disks in a Whitney tower. An order $n$ \emph{twisted Whitney tower} is the same as an order $n$ (framed) Whitney tower, except that the framing condition is not required for Whitney disks of order greater than or equal to $n/2$. 

Denote by $\bW^\iinfty_n=\bW^\iinfty_n(m)$ the set of framed $m$-component links that bound order $n$ {\em twisted} Whitney towers.
This gives the
{\em twisted Whitney tower filtration}
$$
\dots \subseteq \bW^\iinfty_{3} \subseteq \bW^\iinfty_{2} \subseteq \bW^\iinfty_{1} \subseteq  \bW^\iinfty_{0}=\bL 
$$
We refer to Section~\ref{sec:realization-maps} for a precise definition, also of the associated graded $\W^\iinfty_n=\W^\iinfty_n(m)$; and to Section~\ref{sec:w-towers} for details on twisted Whitney towers, including the associated \emph{twisted intersection invariant} $\tau_n^\iinfty(\cW)\in\cT^\iinfty_n$ in section~\ref{subsec:intro-w-tower-int-invariants}. 
These notions appear in this paper for the first time.

Briefly, the odd order groups  $\cT^\iinfty_{2k-1}$ are defined as quotients of $\cT_{2k-1}$ by the torsion subgroups, 
generated by trees of the form 
$ i\,-\!\!\!\!\!-\!\!\!<^{\,J}_{\,J}$; 
where $J$ is a subtree of order~$k-1$, and $i$ is a univalent vertex label. These \emph{boundary-twist relations} correspond to the intersections created by performing a boundary-twist on an order $k$ Whitney disk (Figure~\ref{boundary-twist-and-section-fig}). 

In addition to the generating trees 
for $\cT_{2k}$, the groups $\cT^\iinfty_{2k}$ include trees of the form 
\[ 
  \iinfty\,-\!\!\!\!\!-\!\!\!-\!\!\!-\,\,J
\] 
with $k$ trivalent vertices and one univalent vertex labeled by the twist symbol $\iinfty$ (whereas all other univalent vertices are still labeled by elements in $\{1,\dots,m\}$). These additional generators are called \emph{$\iinfty$-trees} or {\em twisted trees} because the vertex labeled by $\iinfty$ will lie in a twisted Whitney disk. 
They represent framing obstructions on order $k$ Whitney disks and are involved in the new \emph{symmetry}, \emph{twisted IHX}, and \emph{interior twist} relations, all of which have geometric interpretations (Definition~\ref{def:Tau-infty-even}). 


Note that here and in the following the symbol $\iinfty$ represents a {\em twist}, and in particular does {\em not} stand for ``infinity''. 

\begin{thm} \label{thm:twisted} 
The sets $\W^\iinfty_n$ are finitely generated abelian groups under the (well-defined) operation of connected sum $\#$ and there are epimorphisms 
$R^\iinfty_n: \cT^\iinfty_n \sra\W^\iinfty_n$.
\end{thm} 
These {\em twisted realization maps} are constructed in Section~\ref{sec:realization-maps}. As in the framed setting, the key to the proof is a criterion for {\em raising the order of a twisted Whitney tower}:
\begin{thm} \label{thm:twisted-raising-intro} 
A link bounds a twisted Whitney tower $\cW$ of order~$n$ with $\tau^{\iinfty}_{n}(\cW)=0$ if and only if it bounds a twisted Whitney tower of order $n+1$. 
\end{thm} 
Theorem~\ref{thm:twisted-raising-intro} follows from the more general Theorem~\ref{thm:twisted-order-raising-on-A}, 
which is proved in Section~\ref{sec:proof-twisted-thm} using {\em twisted Whitney moves}
(Lemma~\ref{lem:twistedIHX}) as well as boundary-twists and a construction for geometrically canceling twisted Whitney disks. 


Next we bring Milnor's $\mu$-invariants into the picture. Let $\sL_n=\sL_n(m)$ be the degree $n$ part of the free Lie $\Z$-algebra $\sL=\oplus \sL_n$ generated by degree $1$ generators $\{X_1,\ldots, X_m\}$. Via the usual correspondence between rooted oriented labeled trees and non-associative bracketings, $\sL_n$ can be identified with the abelian group on \emph{rooted} trees of order $(n-1)$, where the
root is a chosen unlabeled univalent vertex and the other univalent vertices are labeled as before from the index set $\{1,2,3,\ldots,m\}$, modulo IHX relations (Jacobi identities) and \emph{self-annhilation relations} which kill all generators having an order two symmetry. The self-annhilation relations, which are expressed in terms of brackets as $[X,X]=0$, imply the antisymmetry relations of Figure~\ref{fig:ASandIHXtree-relations}, but not vice versa.

\begin{defn}\label{def:eta-infty}
Define $\eta_n\colon \cT^\iinfty_n\to \sL_1\otimes\sL_{n+1}$ on trees $t$ by $\eta_n(t):=\sum_{v\in t} X_{\ell(v)}\otimes T_v(t)$, where the sum is over all univalent vertices $v$ of $t$, with $T_v(t)$ denoting the rooted tree gotten by replacing $v$ with a root,
and $\ell(v)$ the original label of $v$. On $\iinfty$-trees $\iinfty-\!\!\!-\,J$, define 
$\eta_n(\iinfty-\!\!\!-\,J):=\frac{1}{2}\eta_n(J-\!\!\!-\,J)$, where dividing by $2$ makes sense since the terms in $\eta_n(J-\!\!\!-\,J)$
all have even coefficients. It is easy to check that $\eta_n$ is a well-defined homomorphism for all $n$ (see \cite{CST2,CST4}).
\end{defn}

Define the group $\sD_n$ 
to be the kernel of the bracket map $\sL_1\otimes\sL_{n+1}\to \sL_{n+2}$ which sends $X_i\otimes Y\mapsto [X_i,Y]$. 
This group $\sD_n$ is the natural target for the first non-vanishing Milnor invariants of a link \cite{HM, O}, and 
it turns out that $\eta_n$ maps $\cT^\iinfty_n$ onto $\sD_n$ \cite{CST2}.
In \cite{CST2} we explain the precise relationship between twisted Whitney towers and the (first non-vanishing) Milnor invariants via a $4$-dimensional incarnation of the $\eta_n$-map using \emph{grope duality} \cite{KT}. In particular, we have the following result: 
\begin{thm}[\cite{CST2}]\label{thm:twisted-Milnor}
The Milnor invariants of length $\leq n+1$ vanish for links $L\in  \bW^\iinfty_{n}$, and the length $n+2$ Milnor invariants of $L$ gives rise to a homomorphism
$\mu_n\colon \W^\iinfty_n\to \sD_n$ such that $\mu_n\circ R^\iinfty_n=\eta_n$.
\end{thm}  
We will refer to the maps $\mu_n\colon \W^\iinfty_n\to \sD_n$ as \emph{order $n$} Milnor invariants. For $n=0$ they give linking numbers as well as framings (as coefficients of $X_i\otimes X_i$). 

\begin{cor} \label{cor:twisted-triangle} There is a commutative diagram of epimorphisms
\[
\xymatrix{
\cT^\iinfty_n \ar@{->>}[r]^{R_n^\iinfty} \ar@{->>}[rd]_{\eta_n} & \W^\iinfty_n \ar@{->>}[d]^{\mu_n}\\
& \sD_{n}
}\]
\end{cor}

Moreover, we show below in Theorem~\ref{thm:isomorphisms} that
$\eta_n:\cT_n^\iinfty\to\sD_n$ is an \emph{isomorphism} except when $n\equiv 2\mod 4$.  Since $\sD_n$ is a free abelian group of known rank for all $n$ \cite{O}, this completes the computation of $\W^\iinfty_n$ in three quarters of the cases:

\begin{thm}\label{thm:intro-twisted-3-quarter-classification}
If $n\not\equiv 2\mod 4$, the maps $R^\iinfty_n$ and $\mu_n$ give rise to isomorphisms
$$
\cT^\iinfty_n\cong \W^\iinfty_n\cong \sD_n
$$
\end{thm}

The last quarter of the cases is more complicated as can already be seen for $n=2$: In the case $m=1$ of knots, we show in 
\cite{CST2} that the classical Arf invariant induces an isomorphism $\W^\iinfty_2(1) \cong\Z_2$, whereas all Milnor invariants vanish for knots. 

In Corollary~\ref{cor:kernel-eta 4k-2} of section~\ref{subsec:eta-isos} we derive the following result which computes the kernel of $\eta_n$ for all $n\equiv 2\mod 4$:
\begin{prop}\label{prop:kerEta4k-2}
The map sending $1\otimes J $ to $ \iinfty-\!\!\!\!\!\dash\!\!\!\!<^{\,J}_{\,J}\,\,\in\cT^\iinfty_{4k-2}$ for rooted trees $J$ of order $k-1$ defines an isomorphism $\Z_2 \otimes \sL_k \cong\Ker(\eta_{4k-2}:\cT_{4k-2}^\iinfty\to\sD_{4k-2})$.
\end{prop}
It follows from Corollary~\ref{cor:twisted-triangle} that $\Z_2 \otimes \sL_k$ is also an upper bound on
the kernels of the epimorphisms $R^\iinfty_{4k-2}$ and $\mu_{4k-2}$, and the calculation of $\W^\iinfty_{4k-2}$ will be completed by invariants defined on the kernel of 
$\mu_{4k-2}$ which we believe are new concordance 
invariants generalizing the classical Arf invariant, as we describe next.


\subsubsection*{Higher-order Arf invariants}\label{subsec:intro-higher-order-arf}

Let $\sK^\iinfty_{4k-2}$ denote the kernel of $\mu_{4k-2}: \W^\iinfty_{4k-2} \sra \sD_{4k-2}$. It follows from Corollary~\ref{cor:twisted-triangle} and Proposition~\ref{prop:kerEta4k-2} above that mapping $1\otimes J$ to 
$R^\iinfty_{4k-2}( \iinfty-\!\!\!\!\!-\!\!\!<^{\,J}_{\,J}\,\,)$ induces a surjection $\alpha^\iinfty_k: \Z_2 \otimes \sL_k
\sra \sK^\iinfty_{4k-2}$, for all $k\geq 1$.
Denote by $\overline{\alpha^\iinfty_{k}}$ the induced isomorphism on $(\mathbb Z_2\otimes {\sf L}_{k})/\Ker \alpha^\iinfty_{k}$.

\begin{defn}\label{def:Arf-k}  
The \emph{higher-order Arf invariants} are defined by
$$
\Arf_{k}:=(\overline{\alpha^\iinfty_{k}})^{-1}:\sK^\iinfty_{4k-2}\to(\mathbb Z_2\otimes {\sf L}_{k})/\Ker \alpha^\iinfty_{k}
$$
\end{defn}
From Corollary~\ref{cor:twisted-triangle}, Theorem~\ref{thm:intro-twisted-3-quarter-classification}, Proposition~\ref{prop:kerEta4k-2} and Definition~\ref{def:Arf-k} we see that the groups $\W^\iinfty_n$ are computed by the Milnor and higher-order Arf invariants:
\begin{cor}\label{cor:intro-mu-arf-classify-twisted} 
The groups ${\sf W}^\iinfty_{n}$ are classified by Milnor invariants $\mu_n$ and, in addition, higher-order Arf invariants $\Arf_k$ for $n=4k-2$.
\end{cor}
In particular, it follows that a link bounds an order $n$ twisted Whitney tower if and only if its Milnor invariants and higher-order Arf invariants vanish up to order $n$.

We conjecture that the $\alpha^\iinfty_k$ are isomorphisms, which would mean that the 
$\Arf_k$ are very interesting new concordance invariants:
\begin{conj}\label{conj:Arf-k}
 $\Arf_{k}:\sK^\iinfty_{4k-2}\to\mathbb Z_2\otimes {\sf L}_{k}$ are isomorphisms for all $k$.
\end{conj}

Conjecture~\ref{conj:Arf-k} would imply that  $\W_{4k-2}^\iinfty\cong \cT^\iinfty_{4k-2} \cong (\Z_2 \otimes \sL_k) \oplus \sD_{4k-2}$ where the second isomorphism (is non-canonical and)
already follows from Proposition~\ref{prop:kerEta4k-2}. Conjecture~\ref{conj:Arf-k} is true for $k=1$, with $\Arf_1$ given by the classical Arf invariants of the link components \cite{CST2}. It remains an open problem whether $\Arf_k$ is non-trivial for any $k>1$. 
The links $R^\iinfty_{4k-2}( \iinfty-\!\!\!\!\!-\!\!\!<^{\,J}_{\,J}\,\,)$ realizing the image of $\Arf_{k}$ can all be constructed as internal band sums of iterated Bing doubles of knots having non-trivial classical Arf invariant \cite{CST2}. Such links are known not to be slice
by work of J.C. Cha \cite{Cha}, providing evidence in support of Conjecture~\ref{conj:Arf-k}.

In combination with Theorem~\ref{thm:intro-twisted-3-quarter-classification}, Conjecture~\ref{conj:Arf-k} can be
succinctly expressed in terms of the twisted Whitney tower filtration classification as the statement:
``the twisted realization maps $R^\iinfty_n:\cT^\iinfty_n\to\W_n^\iinfty$ are isomorphisms for all $n$.''

A table of the groups $\sW^\iinfty_n(m)$ for low values of $n,m$ is given in Figure~\ref{fig:twisted-groups}, where the higher-order Arf invariant $\Arf_2$ appears in order $6$. The currently unknown ranks of $\Arf_2$ are represented by the ranges of possible ranks of the $2$-torsion subgroups of the groups $\sW^\iinfty_6(m)$.

For $n=0$, the groups are freely generated by the image under $R_0^\iinfty$ of trees $i -\!\!\! -\, j$, with $i\neq j$, and twisted trees $\iinfty -\!\!\! - \,j$. The resulting links are detected by linking numbers and framings, respectively. For order $n=1$, the generators come (via $R_1^\iinfty$) from trees 
$\tree{i}{j}{k}$ 
where all indices are distinct (otherwise the tree is zero in $\cT^\iinfty_1$ by the boundary-twist relations). 
They are detected by Milnor's triple invariants $\mu(ijk)$.  

In order $n=2$, generators include ($R_2^\iinfty$ of) $\iinfty$-trees $\tree{\iinfty}{i}{j}$ (recall that these indeed lie in $\cT^\iinfty_2$ even though the tree has only one trivalent vertex). If $i\neq j$ these are of infinite order, detected by Milnor's $\mu(ijij)$, but for $i=j$ they have order 2 and are detected by the classical Arf invariant of the $i$th component. This shows how our groups $\cT_{4k-2}^\iinfty$ combine Milnor and Arf invariants in one new formalism.

\begin{figure}
$$
\begin{array}{c|ccccc}
&1&2&3&4&5\\
\hline
0& \Z & \Z^3 & \Z^6 & \Z^{10} & \Z^{15}\\
1& 0 &  0    & \Z & \Z^4      &\Z^{10}\\
2& \z& \Z\oplus\z^2&\Z^6\oplus\z^3&\Z^{20}\oplus \z^4&\Z^{50}\oplus\z^5\\
3&0&0&\Z^6&\Z^{36}&\Z^{126}\\
4&0&\Z^3&\Z^{28}&\Z^{146}&\Z^{540}\\
5&0&0&\Z^{36}&\Z^{340}&\Z^{1740}\\
6&0&\Z^6\oplus \z^{e_{2}}&\Z^{126}\oplus\z^{e_{3}}&\Z^{1200}\oplus\z^{e_{4}}&\Z^{7050}\oplus\z^{e_{5}}\\
\end{array}
$$
\caption{A table of the groups $\W_n^\infty(m)$, 
where $m$ runs horizontally and $n$ runs vertically. The possible ranges of the torsion exponents in order $6$ depend on the currently unknown ranks of $\Arf_2$:
$\quad 0\leq e_{2}\leq1$, $\quad 0\leq e_{3}\leq 3$, $\quad 0\leq e_{4}\leq 6$, $\quad 0\leq e_{5}\leq 10$.}\label{fig:twisted-groups}
\end{figure}

\subsection{Framing twisted Whitney towers}\label{subsec:intro-untwisting-the-twisted-filtration}
Translation of the classification of the twisted Whitney tower filtration back into the framed setting will be be accomplished
in Section~\ref{sec:twisted-to-framed-classification}
using a new interpretation of certain first non-vanishing Milnor invariants as obstructions to framing a twisted Whitney tower.
These are the higher-order Sato-Levine invariants which are defined in all odd orders of the framed Whitney tower filtration (section~\ref{subsec:higher-order-SL}).
The higher-order Arf invariants will also appear as framing obstructions (section~\ref{subsec:higher-order-Arf-in-framed-filtration}), however they will be shifted down one order, due to the fact that a twisted Whitney tower of order $2k$ can always be converted into a framed Whitney tower of order $2k-1$ by twisting and IHX constructions (section~\ref{subsec:boundary-twisted-IHX-lemma}). These geometric constructions will explain the origin of the \emph{framing relations} introduced above in Definition~\ref{def:reduced-tree-group}.

Setting $\widetilde{\cT}_{2k}:=\cT_{2k}$ in even orders, Theorem~\ref{thm:intro-framed-3-quarter-classification} will show that the reduced realization maps $\widetilde{R}_n:\widetilde{\cT}_n\to\W_n$ are isomorphisms in three quarters of the cases, in close analogy with Theorem~\ref{thm:intro-twisted-3-quarter-classification} above. 
Then the higher-order Arf invariants will again appear in the other quarter of cases, and Conjecture~\ref{conj:Arf-k} will have an analogous expression in terms of the framed Whitney tower filtration classification 
as the statement: ``the realization maps $\widetilde{R}_n:\widetilde{\cT}_n\to\W_n$ are isomorphisms for all $n$'' (section~\ref{subsec:higher-order-arf-reduced-realization-conj}).
 
However, the analogy with Theorem~\ref{thm:intro-twisted-3-quarter-classification}  does \emph{not} hold for the Milnor invariants $\mu_n$ in the framed filtration, leading to the appearance of the higher-order Sato-Levine invariants 
in the classification of the framed filtration described in Corollary~\ref{cor:mu-sl-arf-classify}. This subtle interaction between Milnor invariants and framing obstructions is the reason why the framed classification is trickier to describe. 

A table of the framed filtration groups $\W_n(m)$ for low values of $n,m$ is given in Figure~\ref{fig:untwisted-groups}, where the higher-order Arf invariant $\Arf_2$ appears in order $5$. The 
higher-order Sato-Levine invariants correspond to $2$-torsion in all odd orders (for $m>1$), and the ranges of possible ranks of the $2$-torsion subgroups of the groups $\sW_5(m)$ correspond to the possible ranks of $\Arf_2$ (as in Figure~\ref{fig:twisted-groups}). 
\begin{figure}[h]
$$
\begin{array}{c|ccccc}
&1&2&3&4&5\\
\hline
0&\Z&\Z^3&\Z^6&\Z^{10}&\Z^{15}\\
1&\z&\z^3 &\Z\oplus\z^6&\Z^4\oplus\z^{10}&\Z^{10}\oplus\z^{15}\\
2&0&\Z&\Z^6&\Z^{20}&\Z^{50}\\
3&0&\z^2&\Z^6\oplus\z^8&\Z^{36}\oplus\z^{20}&\Z^{126}\oplus\z^{40}\\
4&0&\Z^3&\Z^{28}&\Z^{146}&\Z^{540}\\
5&0&\z^{e_{2}}&\Z^{36}\oplus\z^{e_{3}}&\Z^{340}\oplus\z^{e_{4}}&\Z^{1740}\oplus\z^{e_5}\\
6&0&\Z^6&\Z^{126}&\Z^{1200}&\Z^{7050}\\
\end{array}
$$
\caption{A table of the groups $\sW_n(m)$, where $m$ runs horizontally and $n$ runs vertically. The possible ranges of the torsion exponents in order $5$ depend on the currently unknown ranks of $\Arf_2$:
$\quad 3\leq e_{2}\leq4$, $\quad 18\leq e_{3}\leq 21$, $\quad 60\leq e_{4}\leq 66$, $\quad 150\leq e_{5}\leq 160$. 
}\label{fig:untwisted-groups}
\end{figure}

For $n=0$, the groups come from trees $i -\!\!\! -\, j$, and are detected by linking numbers for $i\neq j$ and framings for $i=j$. For order $n=1$, the generators come (via $R_1$) from trees 
$\tree{i}{j}{k}$.
If all indices are distinct then they are detected by Milnor's triple invariants $\mu(ijk)$. However, in $\widetilde{\cT}_1$ repeating indices also give nontrivial elements of order 2. If $i=j=k$, these are detected by the classical Arf invariant of the $i$th component. In the case where exactly two indices are equal, one needs the classical Sato-Levine invariant (but has to note the framing relations from Figure~\ref{fig:framing-relations}).

The main tool for deriving the framed classification from the twisted one is a commutative diagram of exact sequences
(Theorem~\ref{thm:exact-sequence}) in which the various realization maps connect the tree-groups to the associated graded groups of the filtrations.
As a consequence of our resolution in \cite{CST3} of the Levine Conjecture (Theorem~\ref{thm:LC}), all the relevant tree-groups are completely computed.
So together with some additional geometric and algebraic arguments, the graded groups associated to the framed filtration can be computed in terms of those of the twisted filtration.

In Section~\ref{sec:master-diagrams-and-algebra}, the diagram of Theorem~\ref{thm:exact-sequence} relating the $\cT$- and $\W$-groups is extended by the relevant $\eta$- and $\mu$-maps to include exact sequences of $\sD$-groups, giving a bird's eye view of the classifications. 
The resulting pair of {\em master diagrams} gives a succinct summary of the overall algebraic structure connecting the $\cT$-, $\W$-, and $\sD$-groups.


\subsection{Applications to gropes and $k$-slice theorems}\label{subsec:intro-grope-and-geo-k-slice}

Recall (e.g.~from \cite{T2}) that a \emph{grope} of \emph{class $k$} is defined recursively as follows:
 A grope of class $1$ is a circle and a grope of class $2$ is an orientable surface $\Sigma$ with one boundary component. A grope of class $k$ is formed by attaching to every dual pair of curves in a symplectic basis for  $\Sigma$ a pair of gropes whose classes add to $k$. For details on gropes, including framing conditions, see e.g.~\cite{CT1,CT2,CST,FQ,FT2,Kr,KT,S1,T1}.
  
\subsubsection*{The grope filtration by class}\label{subsubsec:intro-grope-filt}
 
The \emph{grope filtration} (by class) 
on the set $\bL=\bL(m)$ of framed links in $S^3$ with $m$ components is defined by the sets
$\bG_n\subset \bL$ of framed links whose components bound class~$(n+1)$ disjointly embedded framed gropes in $B^4$. The index shift is explained by the next theorem, and for the same reason we {\em define} $\bG_0$ to be the set of evenly framed links. 
The main result from \cite{S1}  implies that this grope filtration equals our framed Whitney tower filtration and hence is being computed in this paper:
\begin{thm}[\cite{S1}]\label{thm:w-tower-equals-grope}
$\bG_{n}=\bW_n$ for all $n$.
\end{thm}

The following result is a sample geometric application of our computations, characterizing links with certain vanishing Milnor invariants.
Details and related applications are described in \cite{CST2}.

\begin{thm}[\cite{CST2}]\label{thm:mega-k-slice}
A link has vanishing Milnor invariants of all orders $\leq 2k-2$
(lengths $\leq 2k$)
if and only if its components bound
disjointly embedded surfaces $\Sigma_i$ in the $4$--ball, with each surface a
connected sum of two surfaces $\Sigma'_i$ and $\Sigma''_i$ such that
\begin{enumerate}
\item
 a symplectic basis of curves on $\Sigma'_i$
 bound disjointly embedded framed gropes $G_{i,j}$ of class $k$ in the complement of 
 $\Sigma := \cup_i\Sigma_i$, and 
\item
 a symplectic basis of curves on $\Sigma''_i$ bound immersed disks in the
 complement of 
 $\Sigma\cup G$, where $G$ is the union of all $G_{i,j}$.
\end{enumerate}
\end{thm}
Theorem~\ref{thm:mega-k-slice} is a considerable strengthening of the Igusa-Orr \emph{$k$-slice Theorem} \cite{IO}: Since the geometric conditions 
in both theorems are equivalent to the vanishing of Milnor's invariants
through order $2k-2$ (length $2k$), one can read our result as saying that the {\em immersed gropes} of class $k$ found by Igusa and Orr can be cleaned up to immersed {\em disks} (these are immersed gropes of arbitrarily high class) or {\em embedded} gropes of class $k$.


\subsection{Comparisons with other iterated disk constructions}\label{subsec:intro-compare-other-towers}

Andrew Casson was the first who tried to recover the Whitney move in dimension four by an iterated disk construction. He started with a simply connected $4$--manifold $M$ with a knot $K$ in its boundary. He looked for conditions so that $K$ would bound an embedded disk in $M$. His starting point was an {\em algebraically transverse sphere} for a (singular) disk in $M$ bounding $K$, an assumption that is satisfied in the setting of the s-cobordism theorem or the surgery exact sequence (but not for $M=B^4$). He then showed that $K$ bounds a {\em Casson tower} of arbitrary height in $M$. In such a tower, one attaches an immersed disk to the accessory circles of every intersection point in a previous stage (and requires that the new disk does not intersect previous stages).

Mike Freedman \cite{F1} realized that one can actually re-embed one Casson tower into another and that one can obtain enough geometric control to prove his breakthrough result: Any Casson tower of height > 3 contains in its neighborhood a topologically-flat embedded disk with boundary $K$. This implies Freedman's classification result for simply connected closed $4$--manifolds and leads to many stunning applications.

However, there can be no obstruction theory for finding Casson towers of larger and larger height, not even in $M=B^4$ (where a transverse sphere cannot exist): Any knot $K$ bounds a Casson tower of height 1 (which is just a singular disk) and if $K$ bounds a Casson tower of height 4 then it is topologically slice (and hence bounds a Casson tower of arbitrary height).
 
This motivated Cochran, Orr and the third author \cite{COT} to study another type of tower, now best called a {\em symmetric} Whitney tower of \emph{height} $n$. Here one inductively attaches Whitney disks to previous stages but only allows these new Whitney disks to intersect each other (and not the previously constructed stages). It follows that a symmetric Whitney tower of height $h$ is a (particularly nice)  Whitney tower of order $2^h$ as studied in this paper, see \cite{S1}.

Such symmetric Whitney towers have an extremely rich theory, even in the case of knots (see \cite{CT} for the fact that the filtration is nontrivial for all heights). All the iterated graded groups are in fact infinitely generated \cite{CHL}, one reason being the existence of higher-order von Neumann signatures that take values in the reals $\R$ (infinitely generated as abelian group).  
There are currently no known algebraic criteria for raising the height of a symmetric Whitney tower, and hence not too much hope for a complete classification of the symmetric Whitney tower filtration of links, or even knots.

That's why the current authors set out to study a simpler version of this filtration, and succeeded in
giving the first instance of a complete computation of a filtration defined via an iterated disk construction, as described in this paper. These Whitney tower filtrations have analogues for immersed $2$--spheres in $4$--manifolds, including a formulation of the proposed higher-order Arf invariants. The order $1$ theory goes back to \cite{FK} (see also \cite{Ma, ST1,St, Ya}, and 10.8A and 10.8B of \cite{FQ} where the relation to the Kirby-Siebenmann invariant is explained), but the higher-order theory is not generally understood for closed $4$--manifolds. 

The relationship between Milnor invariants and trees goes back to Cochran's method of constructing links realizing given (integer) Milnor invariants by ``Bing-doubling along a tree'' \cite{C,C1}. It is intriguing to note that the sequences of circles of intersection between Seifert surfaces that arise in Cochran's construction are strongly suggestive of the Whitney disk boundaries that appear in our construction in section~\ref{subsec:realization-maps}.

The classifications described here also have implications for $3$--manifolds, in particular filtrations on homology cylinders, as described in \cite{CST5}, which also explores the Whitney tower filtrations on string links and their relation to finite type invariants. 

In a future paper we will show that the relation of order $n$ Whitney tower concordance on links is the same as the equivalence relation generated by concordance and simple clasper surgeries with $n$ nodes. This is the same as an equivalence relation on string links considered by Meilhan and Yasuhara \cite{MY} called $C_{n+1}$-concordance. They give a list of classifying invariants up to order $4$ which consists of classical Arf invariants, Milnor invariants and various mod $2$ reductions of Milnor invariants (what we here call higher-order Sato-Levine invariants). They stop just short of order $5$ which is where the first higher-order Arf invariant $\Arf_2$ lives (Figure~\ref{fig:untwisted-groups}).

{\bf Acknowledgments:} This paper was partially written while the first two authors were visiting the third author at the Max-Planck-Institut f\"ur Mathematik in Bonn. They all thank MPIM for its stimulating research environment and generous support. The exposition of this paper was significantly improved by a careful and insightful anonymous referee. The first author was also supported by NSF grant DMS-0604351 and the last author was also supported by NSF grants DMS-0806052 and DMS-0757312. The second author was partially supported by PSC-CUNY research grant PSCREG-41-386. 



\tableofcontents

\section{Whitney towers}\label{sec:w-towers}
We sketch here the relevant theory of Whitney towers as developed in \cite{CST,S1,ST2}, giving details for the new notion of \emph{twisted} Whitney towers. We work in the \emph{smooth oriented} category (with orientations usually suppressed from notation), even though all our results hold in the locally flat topological category by the basic results on topological immersions in Freedman--Quinn \cite{FQ}. In fact, it can be shown that the filtrations   $\mathbb W_n$, $\mathbb W^\iinfty_n$ and $\mathbb G_n$ are identical in the smooth and locally flat settings. This is because a topologically flat surface can be promoted to a smooth surface at the cost of only creating unpaired intersections of arbitrarily high order (see Remark~\ref{rem:locally-flat-and-smooth}).

\subsubsection*{Operations on trees}\label{subsec:trees}
To describe Whitney towers it is convenient to use the bijective correspondence
between formal non-associative bracketings of elements from
the index set $\{1,2,3,\ldots,m\}$ and
rooted trees, trivalent and oriented, with each univalent vertex labeled by an element from the index set, except
for the \emph{root} univalent vertex which is left unlabeled. Recall from Definition~\ref{def:Tau} that an orientation of a tree is determined by cyclic orderings of the adjacent edges around each trivalent vertex.

\begin{defn}\label{def:Trees}
Let $I$ and $J$ be two rooted trees.
\begin{enumerate} 
\item The \emph{rooted product} $(I,J)$ is the rooted tree gotten
by identifying the root vertices of $I$ and $J$ to a single vertex $v$ and sprouting a new rooted edge at $v$.
This operation corresponds to the formal bracket (Figure~\ref{inner-product-trees-fig} upper right). The orientation of $(I,J)$ is inherited from those of $I$ and $J$ as well as the order in which they are glued.

\item The \emph{inner product}  $\langle I,J \rangle $ is the
unrooted tree gotten by identifying the roots of $I$ and $J$ to a single non-vertex point.
Note that $\langle I,J \rangle $ inherits an orientation from $I$ and $J$, and that
all the univalent vertices of $\langle I,J \rangle $ are labeled.
(Figure~\ref{inner-product-trees-fig} lower right.)

\item The \emph{order} of a tree, rooted or unrooted, is defined to be the number of trivalent vertices.
\end{enumerate}
\end{defn}
The notation of this paper will not distinguish between a bracketing and its corresponding rooted tree
(as opposed to the notation $I$ and $t(I)$ used in \cite{S1,ST2}).
In \cite{S1,ST2} the inner product is written as a dot-product, and the rooted product
is denoted by $*$.

\begin{figure}[ht!]
        \centerline{\includegraphics[scale=.40]{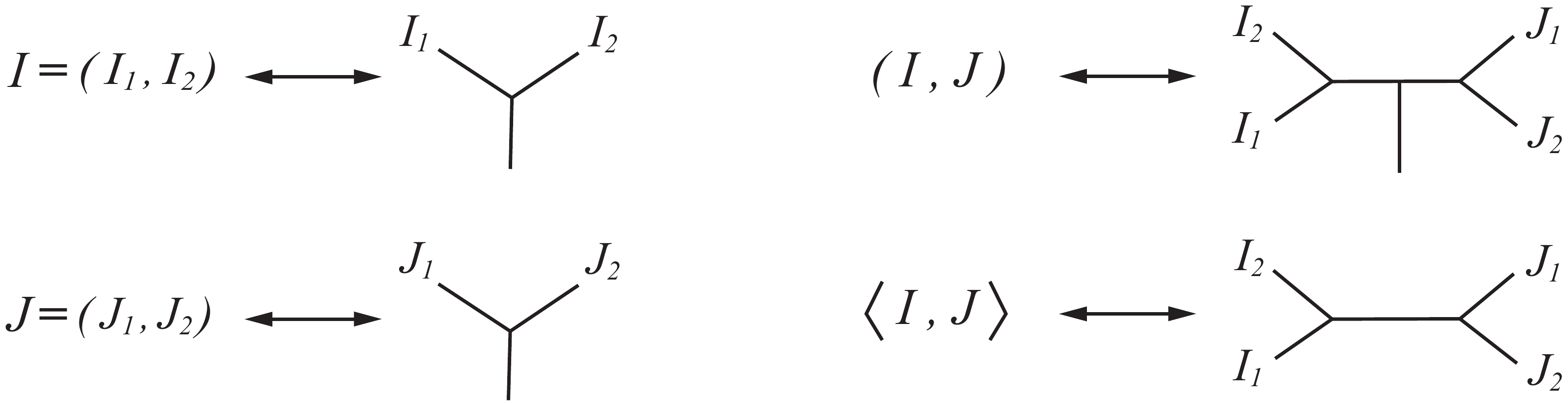}}
        \caption{The \emph{rooted product} $(I,J)$ and \emph{inner product} $\langle I,J \rangle$ of $I=(I_1,I_2)$ and $J=(J_1,J_2)$. All trivalent orientations correspond to a clockwise orientation of the plane.}
        \label{inner-product-trees-fig}
\end{figure}

\subsection{Whitney disks and higher-order intersections}\label{subsec:order-zero-w-towers-and-ints}
A collection $A_1,\ldots,A_m\looparrowright (M,\partial M)$ of 
connected surfaces in a $4$--manifold $M$ is a \emph{Whitney tower of order zero} if the $A_i$ are \emph{properly immersed} in the sense that the boundary is embedded in $\partial M$ and the interior is generically immersed in $M \smallsetminus \partial M$. 




To each \emph{order zero surface} $A_i$ is associated
the order zero rooted tree consisting of an edge with one vertex labeled by $i$, and
to each transverse intersection $p\in A_i\cap A_j$ is associated the order zero
tree $t_p:=\langle \,i\,,\,j\, \rangle$ consisting of an edge with vertices labeled by $i$ and $j$. Note that
for singleton brackets (rooted edges) we drop the bracket from notation, writing $i$ for $(\,i\,)$.

The order 1 rooted Y-tree $(i,j)$, with a single trivalent vertex and two univalent labels $i$ and $j$,
is associated to any Whitney disk $W_{(i,j)}$ pairing intersections between $A_i$ and $A_j$. This rooted tree
can be thought of as being embedded in $M$, with its trivalent vertex and rooted
edge sitting in $W_{(i,j)}$, and its two other edges descending into $A_i$ and $A_j$ as sheet-changing paths. (The cyclic orientation at the trivalent vertex of the bracket  $(i,j)$  corresponds to an orientation of $W_{(i,j)}$ via a convention described below in \ref{subsec:w-tower-orientations}.)

Recursively, the rooted tree $(I,J)$ is associated to any Whitney disk $W_{(I,J)}$ pairing intersections
between $W_I$ and $W_J$ (see left-hand side of Figure~\ref{WdiskIJandIJKint-fig}); with the understanding that if, say, $I$ is just a singleton $i$, then $W_I$ denotes the order zero surface $A_i$. Note that a Whitney disk $W_{(I,J)}$ can be created by a finger move pushing $W_J$ through $W_I$.

To any transverse intersection $p\in W_{(I,J)}\cap W_K$ between $W_{(I,J)}$ and any
$W_K$ is associated the un-rooted tree $t_p:=\langle (I,J),K \rangle$  (see right-hand side of Figure~\ref{WdiskIJandIJKint-fig}).

\begin{figure}[ht!]
        \centerline{\includegraphics[width=120mm]{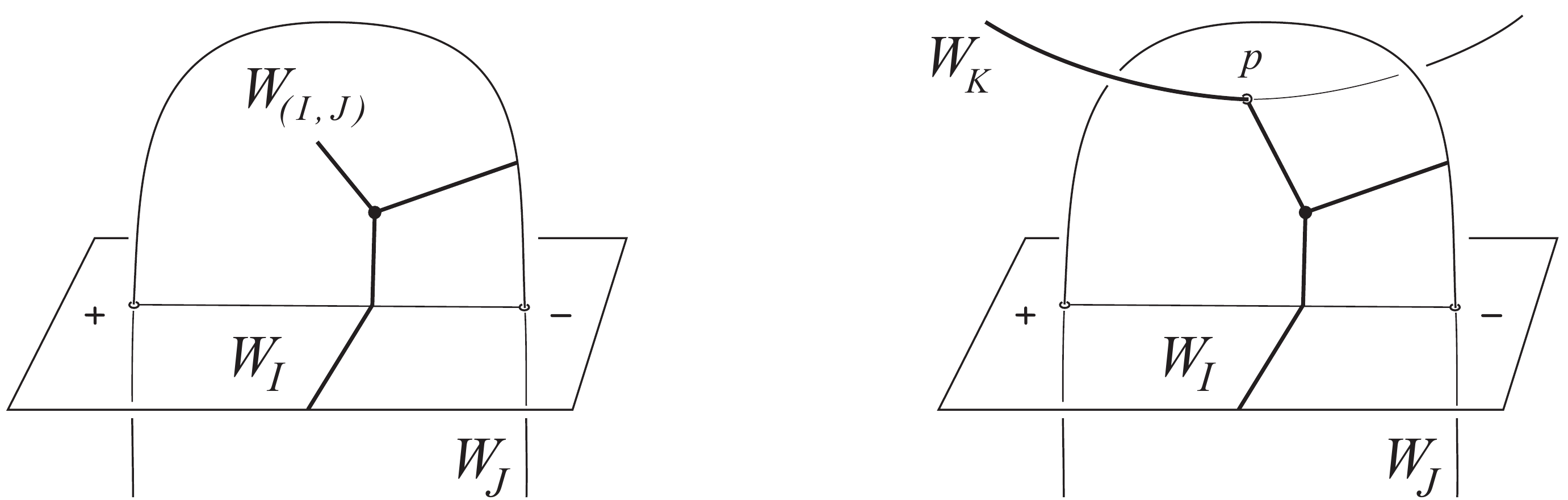}}
        \caption{On the left, (part of) the rooted tree $(I,J)$ associated to a Whitney disk $W_{(I,J)}$. On the right, (part of) the unrooted tree $t_p=\langle (I,J),K \rangle$ associated to an intersection $p\in W_{(I,J)}\cap W_K$. Note that $p$ corresponds to where the roots of $(I,J)$ and $K$ are identified to a (non-vertex) point in $\langle (I,J),K \rangle$.}
        \label{WdiskIJandIJKint-fig}
\end{figure}

\begin{defn}\label{def:int-and-Wdisk-order}
The \emph{order of a Whitney disk} $W_I$ is defined to be the order of the rooted tree $I$, and the \emph{order of a transverse intersection} $p$ is defined to be the order of the tree $t_p$.
\end{defn}

\begin{defn}\label{def:framed-tower}
A collection $\cW$ of properly immersed surfaces together with higher-order
Whitney disks is an \emph{order $n$ Whitney tower}
if $\cW$ contains no unpaired intersections of order less than $n$.  
\end{defn}
The Whitney disks in $\cW$ must have disjointly embedded boundaries, and generically immersed interiors.  All Whitney disks and order zero surfaces must also be \emph{framed} (as discussed in the next subsection).


\begin{rem}\label{rem:locally-flat-and-smooth}
We sketch here a brief explanation of why the smooth and locally flat filtrations are equal. A locally flat surface can be made smooth by a small perturbation, which after introducing cusps as necessary can be assumed to be a regular (locally flat) homotopy. By a general position argument, this regular homotopy can be assumed to be a finite number of finger moves, which are guided by arcs and lead to canceling self-intersection pairs which admit small disjointly embedded Whitney disks (which are `inverses' to the finger moves). These Whitney disks are only locally flat, but can be perturbed to be smooth,
again only at the cost of creating paired self-intersections, and iteration of this process leads to an arbitrarily high-order
smooth sub-Whitney tower pairing all intersections created by the original surface perturbation. 
\end{rem}


\subsection{Twisted and framed Whitney disks}\label{subsec:twisted-w-disks}
The normal disk-bundle of a Whitney disk $W$ in $M$ is isomorphic to $D^2\times D^2$,
and comes equipped with a canonical nowhere-vanishing \emph{Whitney section} over the boundary given by pushing $\partial W$  tangentially along one sheet and normally along the other, avoiding the tangential direction of $W$
(see Figure~\ref{Framing-of-Wdisk-fig}, and e.g.~1.7 of \cite{Sc}).
Pulling back the orientation of $M$ with the requirement that the normal disks
have $+1$ intersection with $W$ means the Whitney section determines
a well-defined (independent of the orientation of $W$)
relative Euler number $\omega(W)\in\Z$ which represents the obstruction to extending
the Whitney section across $W$. Following traditional terminology, when $\omega(W)$ vanishes $W$ is said to be \emph{framed}. (Since $D^2\times D^2$ has a unique trivialization up to homotopy, this terminology is only mildly abusive.)
In general when $\omega(W)=k$, we say that $W$ is
$k$-\emph{twisted}, or just \emph{twisted} if the value of $\omega(W)$ is not specified.
So a $0$-twisted Whitney disk is a framed Whitney disk.

\begin{figure}[ht!]
        \centerline{\includegraphics[width=120mm]{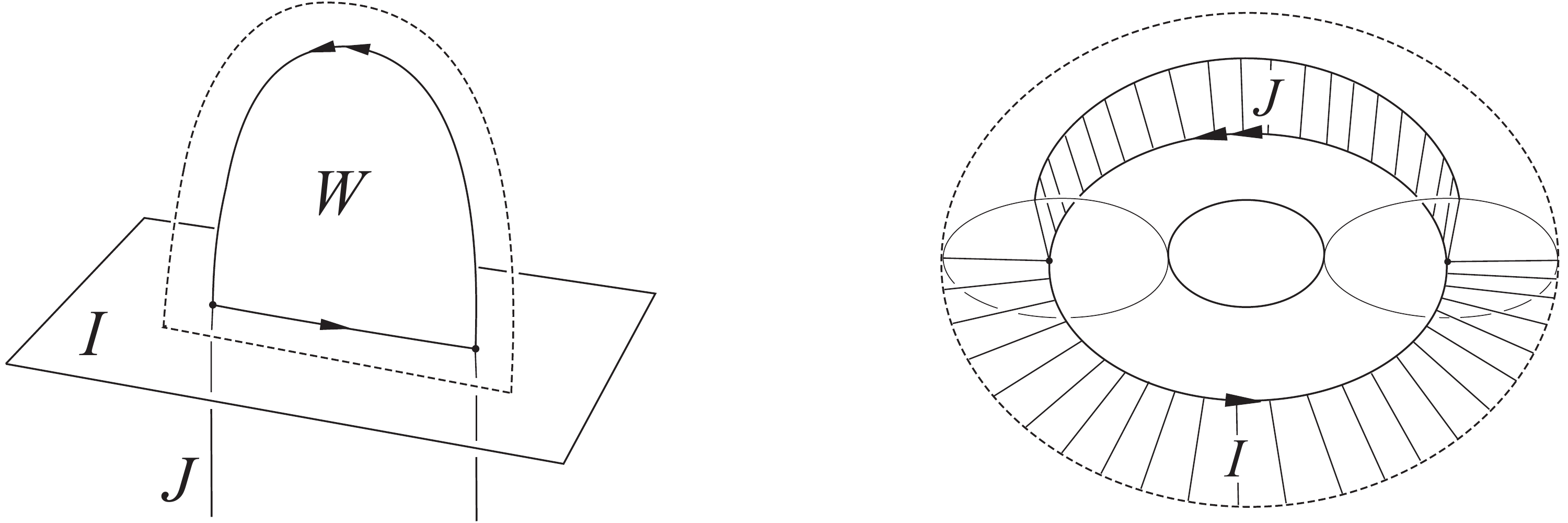}}
        \caption{The Whitney section over the boundary of a framed Whitney disk is
         indicated by the dotted loop shown on the left for an embedded Whitney disk $W$ in
         a 3-dimensional slice of $4$--space. On the right is shown an embedding into $3$--space of the normal
         disk-bundle over $\partial W$, indicating how the Whitney section determines a well-defined nowhere 		
         vanishing section which lies in the $I$-sheet and is normal to the $J$-sheet. }
        \label{Framing-of-Wdisk-fig}
\end{figure}

Note that a {\em framing} of $\partial A_i$ (respectively $A_i$) is by definition a trivialization of the normal bundle of the immersion. If the ambient $4$-manifold is oriented, this is equivalent to an orientation and a nonvanishing normal vector field on $\partial A_i$ (respectively $A_i$).
The twisting $\omega(A_i)\in\Z$ of an order zero surface is also defined when a framing of $\partial A_i$ is given, and $A_i$ is said to be \emph{framed} 
when $\omega(A_i)=0$.


\subsection{Twisted Whitney towers and their orientations}\label{subsec:intro-twisted-w-towers}
In the definition of an order $n$ Whitney tower given just above (following \cite{CST,S1,S2,ST2})
all Whitney disks and order zero surfaces are required to be framed. It turns out that the natural generalization to twisted Whitney towers involves allowing twisted Whitney disks only in at least ``half the order'' as follows:

\begin{defn}\label{def:Twisted-W-towers}
A \emph{twisted Whitney tower of order $0$} is a collection of properly immersed surfaces
in a $4$--manifold
(without any framing requirement).

For $n>0$, a \emph{twisted Whitney tower of order $(2n-1)$} is just a (framed) Whitney
tower of order $(2n-1)$ as in Definition~\ref{def:framed-tower} above.

For $n>0$, a \emph{twisted Whitney tower of order $2n$} is a Whitney
tower having all intersections of order less than $2n$ paired by
Whitney disks, with all Whitney disks of order less than $n$ required to be framed, but Whitney disks of order at least $n$ allowed to be twisted.
\end{defn}

\begin{rem}\label{rem:framed-is-twisted}
Note that, for any $n$, an order $n$ (framed) Whitney tower
is also an order $n$ twisted Whitney tower. We may
sometimes refer to a Whitney tower as a \emph{framed} Whitney tower to emphasize
the distinction, and will always use the adjective ``twisted'' in the setting of
Definition~\ref{def:Twisted-W-towers}.
\end{rem}

\begin{rem}\label{rem:motivate-twisted-towers}
The convention of allowing only order $\geq n$ twisted Whitney disks in order $2n$ twisted Whitney towers is explained both algebraically and geometrically in \cite{CST2} via relationships with Milnor invariants and grope duality. In any event, an order $2n$ twisted Whitney tower can always be modified
so that all its Whitney disks of order $>n$ are framed (see \textbf{notation and conventions} in section~\ref{subsec:twisted-order-raising-thm-proof}), so the twisted Whitney disks of order equal to $n$ are the important ones (i.e.~they may represent obstructions to ``raising the order'').  
\end{rem}


 \subsubsection*{Whitney tower orientations}\label{subsec:w-tower-orientations}

Orientations on order zero surfaces in a Whitney tower $\cW$ are fixed, and required to induce the orientations
on their boundaries.
After choosing and fixing orientations on all the Whitney disks in
$\cW$, the associated trees 
are embedded in $\cW$ so that the vertex orientations are induced from
the Whitney disk orientations, with the descending edges of each
trivalent vertex enclosing the \emph{negative intersection point} of the corresponding Whitney disk, as in Figure~\ref{WdiskIJandIJKint-fig}.
(In fact, if a tree $t$ has more than one trivalent vertex which corresponds to the same Whitney disk, then
$t$ will only be immersed in $\cW$, but this immersion can be taken to be a local embedding around each trivalent vertex of $t$
as in Figure~\ref{WdiskIJandIJKint-fig}.)

This ``negative corner'' convention, which differs from the
positive corner convention in the older papers \cite{CST,ST2}, is
compatible with standard orientation and commutator conventions for group elements used in \cite{CST2}.
With these conventions, different choices of orientations on Whitney disks in $\cW$ correspond to antisymmetry relations (as explained in \cite{ST2}).

  
\subsection{Intersection invariants for (twisted) Whitney towers}\label{subsec:intro-w-tower-int-invariants}



The obstruction theory of \cite{ST2} in the current simply connected setting works as follows.
\begin{defn}
The \emph{order $n$ intersection invariant} $\tau_n(\cW)$ of an order
$n$ Whitney tower $\cW$ is defined to be
$$
\tau_n(\cW):=\sum \epsilon_p\cdot t_p \in\cT_n
$$ 
where the sum is over all order $n$ intersections $p$,
with
$\epsilon_p=\pm 1$ the usual sign of a transverse intersection
point.
\end{defn}

As mentoned in the introduction, 
if $L$ bounds $\cW\subset B^4$ with $\tau_n(\cW)=0\in \cT_n$, then $L$ bounds a Whitney tower
of order $n+1$. This is a special case of the simply connected version of the more general Theorem~2 of \cite{ST2}. 
We will use the following version of Theorem~2 of \cite{ST2}
where the order zero surfaces are either properly immersed disks in 
$B^4$ or properly immersed annuli in $S^3\times I$:  

\begin{thm}[\cite{ST2}]\label{thm:framed-order-raising-on-A}
If a collection $A$ of properly immersed surfaces in a simply connected $4$--manifold supports an order $n$ Whitney tower $\cW$ with $\tau_n(\cW)=0\in\cT_n$, then $A$ is regularly homotopic (rel $\partial$) to 
$A'$ which supports an order $n+1$ Whitney tower.
\end{thm}

In Section~\ref{sec:reduced-groups-obstruction-theory} this theorem will be strengthened by showing that
the same conclusion holds if
$\tau_n(\cW)$ vanishes in the reduced group $\widetilde{\cT}_n$ (Definition~\ref{def:reduced-tree-group}).

The intersection invariants for Whitney towers are extended to
twisted Whitney towers as follows:

\begin{defn}\label{def:Tau-infty-odd}
The abelian group $\cT^{\iinfty}_{2n-1}$ is the quotient of $\cT_{2n-1}$ by the \emph{boundary-twist relations}: 
\[
\langle  (i,J),J \rangle \,=\, i\,-\!\!\!\!\!-\!\!\!<^{\,J}_{\,J}\,\,=\,0
\] 
Here $J$ ranges over all order $n-1$ rooted trees.
\end{defn}

The boundary-twist relations
correspond geometrically to the fact that 
performing a boundary twist (Figure~\ref{boundary-twist-and-section-fig}) on an order $n$ Whitney disk $W_{(i,J)}$ creates an order $2n-1$ intersection point
$p\in W_{(i,J)}\cap W_J$ with associated tree $t_p=\langle  (i,J),J \rangle $ (which is 2-torsion
by the AS relations) and changes $\omega (W_{(i,J)})$ by $\pm1$. Since order $n$ twisted Whitney disks are allowed in an order $2n$ Whitney tower such trees do not represent obstructions to the existence of the next order twisted tower.

For any rooted tree $J$ we define the corresponding {\em $\iinfty$-tree}, denoted by $J^\iinfty$, by labeling the root univalent vertex with the symbol ``$\iinfty$'':
$$
J^\iinfty := \iinfty\!-\!\!\!- J 
$$ 


\begin{defn}\label{def:Tau-infty-even}
The abelian group $\cT^{\iinfty}_{2n}$ is the free abelian group on order $2n$
trees and order $n$ $\iinfty$-trees, modulo the
following relations:
\begin{enumerate}
     \item AS and IHX relations on order $2n$ trees (Figure~\ref{fig:ASandIHXtree-relations})
   \item \emph{symmetry} relations: $(-J)^\iinfty = J^\iinfty$
  \item \emph{twisted IHX} relations: $I^\iinfty=H^\iinfty+X^\iinfty- \langle H,X\rangle $
   \item {\em interior twist} relations: $2\cdot J^\iinfty=\langle J,J\rangle $
\end{enumerate}
\end{defn}

Here the AS and IHX relations are as usual, but they only apply to non-$\iinfty$ trees. 
The \emph{symmetry relation} corresponds to the fact that the relative 
Euler number $\omega(W)$ is independent of the orientation of the Whitney disk $W$, with the minus sign denoting that the cyclic orderings at the trivalent vertices of $-J$ differ from those of $J$ at an odd number of vertices.
The \emph{twisted IHX relation} corresponds to the effect of performing a Whitney move in the presence of a twisted Whitney disk, as described below in
Lemma~\ref{lem:twistedIHX}. The \emph{interior-twist relation} corresponds to the fact that creating a 
$\pm1$ self-intersection
in a $W_J$ changes the twisting by $\mp 2$ (Figure~\ref{InteriorTwistPositiveEqualsNegative-fig}).
\vspace{.25 in}

\begin{rem}\label{rem:quadratic-form}
The symmetry, twisted IHX, and interior twist relations in $\mathcal T^\iinfty_{2n}$ have a surprisingly natural algebraic interpretation that we explain in \cite{UQF}. The idea is to extend the map $J\mapsto J^\iinfty$ to a {\em symmetric quadratic refinement} $q$ of the bilinear form $\langle \cdot,\cdot\rangle$ on the free quasi-Lie algebra of rooted trees (the intersection form on Whitney disks) by defining $q(J)=J^\iinfty$ and extending to linear combinations by the formula
\[
q(J+K):=J^\iinfty+K^\iinfty+\langle J,K\rangle
\]
Expanding $q(I-H+X)=0$ leads to the 6-term IHX relation
\[
I^\iinfty+H^\iinfty+X^\iinfty=\langle I,H \rangle-\langle I,X \rangle+\langle H,X \rangle
\]
which is equivalent to the twisted IHX relation in the presence of the interior-twist relations. Those in turn follow by setting $K:=-J$ from the symmetry relation.
In \cite{UQF} we show that 
$\cT$ is the universal home for invariant symmetric bilinear forms on free quasi-Lie algebras, and that $\cT^\iinfty_{2k}$ is the universal (symmetric quadratic) refinement of this form in order $k$.
\end{rem}

\begin{rem}\label{rem:finite-type-IHX}

We discovered in \cite{CST} that the (framed) IHX relation can be realized in three dimensions as well as four, and
it is interesting to note that many of the relations that we obtain for twisted Whitney towers in four dimensions can also be realized by rooted clasper surgeries (grope cobordisms) in three dimensions. Here the twisted Whitney disk corresponds to a $\pm1$ framed leaf of  a clasper.   For example the relation $I^\iinfty=H^\iinfty+X^\iinfty-\langle H,X\rangle$ has the following clasper explanation. $I^\iinfty$ represents a clasper with one isolated twisted leaf. By the topological IHX relation, one can replace $I^\iinfty$ by two claspers of the form $H^\iinfty$ and $(-X)^\iinfty=X^\iinfty$ embedded in a regular neighborhood of the original clasper with leaves parallel to the leaves of the original. The twisted leaves are now linked together, so applying Habiro's zip construction (which complicates the picture considerably) one gets three tree claspers, of the form 
$H^\iinfty$, $X^\iinfty$ and $\langle H,-X\rangle$ respectively. 

Similarly, the relation $2\cdot J^\iinfty=\langle J,J\rangle$ has an interpretation where one takes a clasper which represents $J^\iinfty$ and splits off a geometrically canceling parallel copy, representing the tree $J^\iinfty$. Again, because the twisted leaves link, we also get the term $\langle J,-J\rangle.$ 

These observations will be enlarged upon in \cite{CST5} to analyze filtrations on homology cylinders and string links.
\end{rem}

Recall from Definition~\ref{def:Twisted-W-towers} (and Remark~\ref{rem:motivate-twisted-towers}) that twisted Whitney disks
only occur in even order twisted Whitney towers, and only those of half-order are
relevant to the obstruction theory. 
\begin{defn}\label{def:Tau-infty}
The \emph{order $n$ intersection intersection invariant}
$\tau_{n}^{\iinfty}(\cW)$ of an order
$n$ twisted Whitney tower $\cW$ is defined to be
$$
\tau_{n}^{\iinfty}(\cW):=\sum \epsilon_p\cdot t_p + \sum \omega(W_J)\cdot J^\iinfty\in\cT^{\iinfty}_{n}
$$
where the first sum is over all order $n$ intersections $p$ and the second sum is over all order $n/2$
Whitney disks $W_J$ with twisting $\omega(W_J)\in\Z$. For $n=0$, recall from \ref{subsec:order-zero-w-towers-and-ints} above our notational convention that $W_j$ denotes $A_j$, and that $\omega(A_j)\in\Z$ is the relative Euler number of the normal bundle of $A_j$ with respect to the given framing of $\partial A_j$ as in \ref{subsec:twisted-w-disks} .

\end{defn}

By splitting the twisted Whitney disks, as explained in subsection~\ref{subsec:split-w-towers} below, 
for $n>0$ we may actually assume that all non-zero $\omega(W_J)\in\{\pm 1\}$, just like the signs $\epsilon_p$. 


As in the framed case, the vanishing of $\tau_{n}^{\iinfty}$ is sufficient for the existence of a 
twisted Whitney tower of order $(n+1)$, and the proof in Section~\ref{sec:proof-twisted-thm} of Theorem~\ref{thm:twisted} (describing the twisted realization maps $R^\iinfty_n:\cT_n^\iinfty\to\W^\iinfty_n$) will be based on the following analogue of the framed order-raising Theorem~\ref{thm:framed-order-raising-on-A} to the twisted setting: 
\begin{thm}\label{thm:twisted-order-raising-on-A}
If a collection $A$ of properly immersed surfaces in a simply connected $4$--manifold supports an order $n$ twisted Whitney tower $\cW$ with $\tau_n^\iinfty(\cW)=0\in\cT^\iinfty_n$, then $A$ is regularly homotopic (rel $\partial$) to 
$A'$ which supports an order $n+1$ twisted Whitney tower.
\end{thm}
The proof of Theorem~\ref{thm:twisted-order-raising-on-A} is given in Section~\ref{sec:proof-twisted-thm} below.

Proofs of the ``order-raising'' Theorems \ref{thm:twisted-order-raising-on-A} and \ref{thm:framed-order-raising-on-A} (and its strengthening Theorem~\ref{thm:framed-order-raising-mod-Delta} below) depend on
realizing the relations in the target groups by controlled manipulations of Whitney towers. The next two subsections introduce combinatorial
notions useful for describing the algebraic effect of such geometric constructions. 

\emph{For the rest of this section we assume our Whitney towers are of positive order for convenience of notation.}

\subsubsection*{Intersection forests}\label{subsec:int-forests}
Recall that the trees associated to intersections and Whitney disks in a Whitney tower can be considered to be immersed in the Whitney tower, with vertex orientations induced by the Whitney tower orientation, as in Figure~\ref{WdiskIJandIJKint-fig}. 

\begin{defn}\label{def:intersection forests}
The \emph{intersection forest} $t(\cW)$ of a framed Whitney tower $\cW$ is the disjoint union of signed trees associated to all unpaired intersections $p$ in $\cW$: 
 \[
 t(\cW)=\amalg_p\ \epsilon_p \cdot  t_p
 \]
with $\epsilon_p$ the sign of the intersection point $p$.
For $\cW$ of order~$n$, we can think of the signed order $n$ trees in $t(\cW)$ as an ``abelian word'' in the generators $\pm t_p$ which represents 
$\tau_n(\cW)\in\cT_n$.  
More precisely, $t(\cW)$ is an element of the free abelian monoid, with unit $\emptyset$, generated by (isomorphism classes of) signed trees, trivalent, labeled and vertex-oriented as usual. We emphasize that there are no cancellations or other relations here. 

\begin{rem}
In the older papers \cite{CST,S1,ST2} we referred to $t(\cW)$ as the ``geometric intersection tree'' (and to the group element
$\tau_n(\cW)$ as the order $n$ intersection ``tree'', rather than ``invariant''), but the term ``forest'' better describes
the disjoint union of (signed) trees $t(\cW)$.
\end{rem}

Similarly to the framed case, the \emph{intersection forest} $t(\cW)$ of a {\em twisted} Whitney tower $\cW$ is the disjoint union of signed trees associated to all unpaired intersections $p$ in $\cW$ and integer-coefficient $\iinfty$-trees associated to all non-trivially twisted Whitney disks $W_J$ in $\cW$:
\[
t(\cW)=\amalg_p \ \epsilon_p \cdot  t_p \,\, + \amalg_J \ \omega(W_J)\cdot  J^\iinfty
\]
with $\omega(W_J)\in\Z$ the twisting of $W_J$. Again, there are no cancellations or relations (and the informal ``$+$'' sign in the expression is purely cosmetic).   
\end{defn}
We will see in the next subsection that all the trees can be made to be disjoint in $\cW$, with all non-zero $\omega(W_J)=\pm 1$, so that $t(\cW)$ is also a topological disjoint union which corresponds to an element in the free abelian monoid generated by (isomorphism classes of) signed trees, 
and signed $\iinfty$-trees.

\subsection{Splitting twisted Whitney towers}\label{subsec:split-w-towers}
A framed Whitney tower is \emph{split} if the set of singularities in the interior
of any Whitney disk consists of either a single point, or a single boundary arc of a Whitney disk, or is empty.
This can always be arranged, as observed in Lemma~13 of \cite{ST2} (Lemma~3.5 of \cite{S1}), by performing finger moves along Whitney disks guided by arcs
connecting the Whitney disk boundary arcs (see Figure~\ref{fig:W-tower-and-trees-and-split}). Implicit in this construction is that the finger moves preserve the Whitney disk framings
(by not twisting relative to the Whitney disk that is being split -- see Figure~\ref{twist-split-Wdisk-fig}).
A Whitney disk $W$ is \emph{clean} if the interior of $W$ is embedded and disjoint from the rest of the Whitney tower.
In the setting of twisted Whitney towers, it
will simplify the combinatorics to use ``twisted'' finger moves to similarly split-off twisted Whitney disks 
into $\pm 1$-twisted
clean Whitney disks.

We call a twisted Whitney tower \emph{split} if all of its non-trivially twisted Whitney disks are clean and have twisting 
$\pm 1$, and all of its framed Whitney disks are split in the usual sense (as for framed Whitney towers).
\begin{lem}\label{lem:split-w-tower}
If $A$ supports an order $n$ twisted Whitney tower $\cW$, then $A$ is homotopic (rel $\partial$) to 
$A'$ which supports a split order $n$ twisted
Whitney tower $\cW'$, such that:
\begin{enumerate}
\item The disjoint union of non-$\iinfty$ trees $\amalg_p \ \epsilon_p \cdot  t_p \subset t(\cW)$ is isomorphic to 
the disjoint union of non-$\iinfty$ trees $\amalg_{p'} \ \epsilon_{p'} \cdot  t_{p'} \subset t(\cW')$.
\item Each $\omega(W_J)\cdot  J^\iinfty$ in $t(\cW)$ gives rise to the disjoint union of exactly $|\omega(W_J) |$-many $\pm 1\cdot J^\iinfty$ in $\cW'$, 
where the sign $\pm$ corresponds to the sign of $\omega(W_J)$.
\end{enumerate}
\end{lem}

\begin{proof}
\begin{figure}
\centerline{\includegraphics[width=120mm]{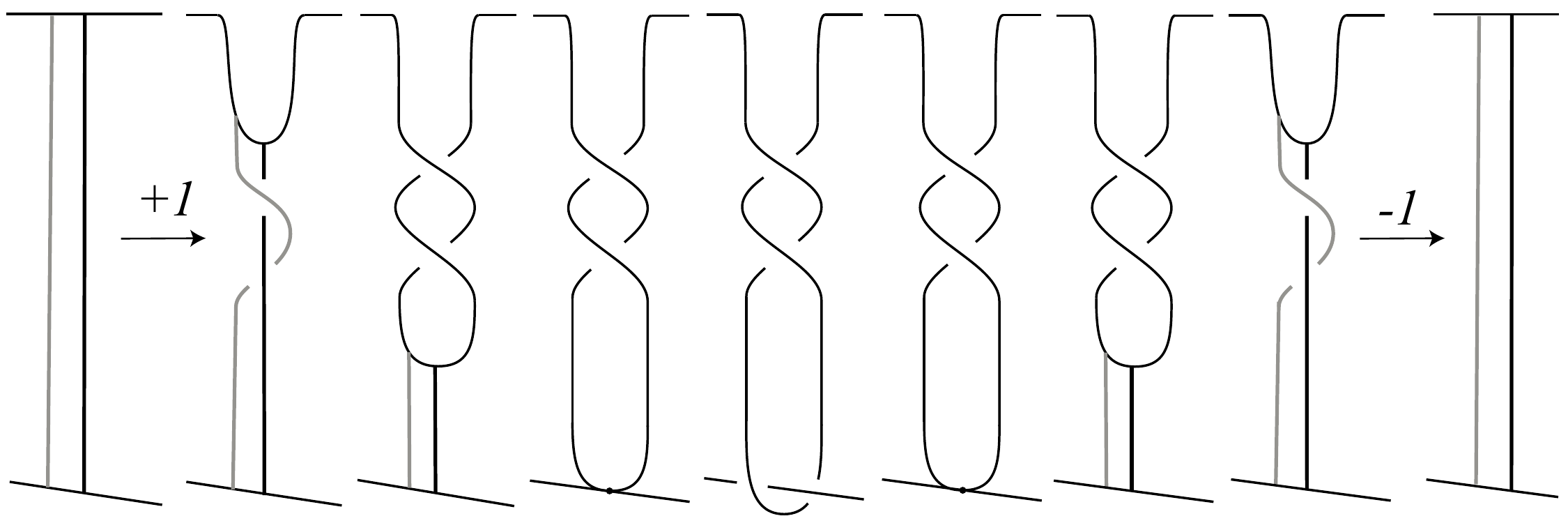}}
         \caption{A neighborhood of a twisted finger move which splits a Whitney disk into two Whitney disks. 
         The vertical black arcs are slices of
         the new Whitney disks, and the grey arcs are slices of extensions of the Whitney sections.
         The finger-move is supported in a neighborhood of an arc in the original Whitney disk running 
         from a point in the Whitney disk boundary on the ``upper'' surface sheet to a point 
         in the Whitney disk boundary on the ``lower'' surface sheet. (Before the finger-move this guiding arc would 
         have been visible in the middle picture as a vertical black arc-slice of the original
         Whitney disk.)}
         \label{twist-split-Wdisk-fig}
\end{figure}
Illustrated in
Figure~\ref{twist-split-Wdisk-fig} is a local picture of a twisted finger move, which splits one Whitney disk into two, while also changing twistings.
If the original Whitney disk in Figure~\ref{twist-split-Wdisk-fig} was framed, then the two new Whitney disks will have twistings $+1$ and $-1$, respectively. In general, if the arc guiding the finger move splits the twisting of the original Whitney disk into $\omega_1$ and $\omega_2$ zeros of the extended Whitney section, then the two new Whitney disks will have twistings $\omega_1+1$ and
$\omega_2-1$, respectively. Thus, by repeatedly splitting off framed corners into $\pm 1$-twisted Whitney disks, any 
$\omega$-twisted Whitney disk ($\omega \in\Z$) can be split into $|\omega |$-many $+1$-twisted or $-1$-twisted clean Whitney disks, together with split framed Whitney disks containing any interior intersections in the original twisted Whitney disk. Combining this with the untwisted splitting of the framed Whitney disks as in Lemma~13 of \cite{ST2} gives the result.
\end{proof}


\section{The realization maps}\label{sec:realization-maps}
This section contains clarifications and proofs of Theorems \ref{thm:R-onto-W} and \ref{thm:twisted} from the introduction
which state the existence of surjections $R_n\colon\cT_n\to\W_n$ and $R^\iinfty_n\colon\cT^\iinfty_n\to\W^\iinfty_n$ for all $n$, in particular exhibiting the 
sets $\W_n$ and $\W^\iinfty_n$ as finitely generated abelian groups under connected sum. 

All proofs in this section apply in the reduced setting as well, and the constructions
described here also define the surjections $\widetilde{R}_n\colon\widetilde{\cT}_n\to\W_n$ described
in the introduction. 

Recall that our manifolds are assumed oriented, but orientations are suppressed from the discussion as much as possible.
In the following an orientation is fixed once and for all on $S^3$; and a \emph{framed link} has oriented components, each equipped with a nowhere-vanishing 
normal section.
\begin{defn}\label{def:links-bounding-towers}
A framed link $L\subset S^3=\partial B^4$ \emph{bounds} an order $n$ Whitney tower $\cW$ if
$\cW\subset B^4$ is an order $n$ Whitney tower whose order zero surfaces are immersed disks bounded by the components of $L$, as in Definition~\ref{def:framed-tower}. 

Similarly, a framed link $L\subset S^3=\partial B^4$ \emph{bounds} an order $n$ twisted Whitney tower $\cW$ if
$\cW\subset B^4$ is an order $n$ twisted Whitney tower whose order zero surfaces are immersed disks bounded by the components of $L$, as in Definition~\ref{def:Twisted-W-towers}. 
\end{defn}

\begin{defn}\label{def:wtc}
For $n\geq 1$, framed links $L_0$ and $L_1$ in $S^3$ are \emph{Whitney tower concordant of order $n$} if the $i$th components of $L_0\subset S^3\times\{0\}$ and $-L_1\subset S^3\times\{1\}$ cobound an immersed annulus $A_i$ for each $i$ such that the $A_i$ are transverse and support an order $n$ Whitney tower.
If the $A_i$ support a \emph{twisted} order $n$ Whitney tower then 
$L_0$ and $L_1$ are said to be \emph{twisted Whitney tower concordant of order $n$}.
\end{defn}
Note that a (twisted) Whitney tower concordance preserves framings on on $L_0$ and $L_1$ (as links in $S^3$)
since, for all $i$, $\omega(A_i)=0$ because all self-intersections of the $A_i$ come in (geometrically) canceling pairs in any (twisted) Whitney tower of order $n\geq 1$.

Recall from the introduction that the set of $m$-component framed links in $S^3$ which bound order $n$ (twisted) Whitney towers in $B^4$ is denoted by $\bW_n=\bW_n(m)$ (resp. $\bW^\iinfty_n$); and the quotient of $\bW_n$ by the equivalence relation of order $n+1$ (twisted) Whitney tower concordance is 
denoted by $\W_n$ (resp. $\W^\iinfty_n$).

Throughout this section the twisted setting mirrors the framed setting, with discussions and arguments given simultaneously.

We begin by deriving from the ``order-raising'' Theorem~\ref{thm:framed-order-raising-on-A} the following essential criterion for links to represent 
equal elements in the associated graded $\W_n$:

\begin{cor}\label{cor:tau=w-concordance}
Links $L_0$ and $L_1$ represent the same element of $\W_n$
if and only if there exist order $n$ Whitney towers $\cW_i$ in $B^4$ with $\partial\cW_i=L_i$ and $\tau_n(\cW_0)=\tau_n(\cW_1)\in\cT_n$.
\end{cor}

\begin{proof}
If $L_0$ and $L_1$ are equal in $\W_n$ then they cobound $A$ supporting an order $n+1$ Whitney tower $\cV$ in $S^3\times I$, and any order $n$ Whitney tower $\cW_1$ in $B^4$ bounded by $L_1$ can be extended by $\cV$ to 
form an order $n$ Whitney tower $\cW_0$ in $B^4$ bounded by $L_0$, with 
$\tau_n(\cW_0)=\tau_n(\cW_1)\in\cT_n$
since $\tau_n(\cV)$ vanishes. 

Conversely, suppose that $L_0$ and $L_1$ bound order $n$ Whitney towers $\cW_0$ and $\cW_1$ in $4$--balls $B_0^4$ and $B_1^4$, with 
$\tau_n(\cW_0)=\tau_n(\cW_1)$. Then constructing $S^3\times I$
as the connected sum $B_0^4\# B_1^4$ (along balls in the complements of $\cW_0$ and $\cW_1$), and tubing together the corresponding order zero disks of $\cW_0$ and $\cW_1$, and taking the union of the Whitney disks in $\cW_0$ and 
$\cW_1$, yields a collection $A$ of properly immersed annuli connecting $L_0$ and $L_1$ and supporting an order $n$ Whitney tower $\cV$.  Since the orientation of the ambient $4$--manifold has been reversed for one of the original Whitney towers, say $\cW_1$, which results in a global sign change for 
$\tau_n(\cW_1)$, it follows that $\cV$ has vanishing order $n$ intersection invariant:
$$
\tau_n(\cV)=\tau_n(\cW_0)-\tau_n(\cW_1)=\tau_n(\cW_0)-\tau_n(\cW_0)=0\in\cT_n
$$
So by Theorem~\ref{thm:framed-order-raising-on-A}, $A$ is homotopic (rel $\partial$) to $A'$ supporting
an order $n+1$ Whitney tower, and hence $L_0$ and $L_1$ are equal in $\W_n$.
\end{proof} 

\begin{rem}\label{rem:tau=w-concordance}
The analogous statement and proof of Corollary~\ref{cor:tau=w-concordance} holds in the twisted case (with the ``twisted order-raising'' Theorem~\ref{thm:twisted-order-raising-on-A} playing the role of Theorem~\ref{thm:framed-order-raising-on-A}). For this case, we'll spell out the statement carefully but in several instances below we will just state that the twisted case is analogous:
Links $L_0$ and $L_1$ in $\bW_n^\iinfty$ represent the same element of $\W^\iinfty_n$
if and only if there exist order $n$ twisted Whitney towers $\cW_0$ and $\cW_1$ in $B^4$ bounded by $L_0$ and $L_1$ respectively such that
$\tau^\iinfty_n(\cW_0)=\tau^\iinfty_n(\cW_1)\in\cT^\iinfty_n$.
\end{rem}

\begin{rem}\label{rem:reduced-tau=w-concordance}
Remark~\ref{rem:tau=w-concordance} similarly applies to the reduced setting by Theorem~\ref{thm:framed-order-raising-mod-Delta} below, although we will omit further reference to $\widetilde{\cT}$ in this section.
\end{rem}

\subsection{Band sums of links}\label{subsec:band-sum}
The \emph{band sum} $L\#_\beta L'\subset S^3$
of oriented $m$-component links $L$ and $L'$ along bands $\beta$ is defined as follows: Form $S^3$ as the connected sum of $3$--spheres containing $L$ and $L'$ along balls in the link complements.  Let
$\beta$ be a collection of disjointly embedded oriented bands joining like-indexed link components such that the band orientations are compatible with the link orientations. Take the usual connected sum of each pair of components along the corresponding band. Although it is well-known that the concordance class of $L\#_\beta L'$ depends in general on $\beta$, it turns out that the image of 
$L\#_\beta L'$ in $\W_n$ (or in $\W^\iinfty_n$) does not depend on $\beta$:

\begin{lem}\label{lem:link-sum-well-defined}   
For links $L$ and $L'$ representing elements of $\W_n$, any band sum $L\#_\beta L'$ represents an element of 
$\W_n$ which only depends on the equivalence classes of $L$ and $L'$ in $\W_n$.	The same statement holds in $\W^\iinfty_n$.
\end{lem}

\begin{proof} We shall only give the proof in the framed case, the twisted case is analogous.
If $L_0$ and $L_1$ represent the same element of $\W_n$, 
and if $L'_0$ and $L'_1$ represent the same element of $\W_n$, then
by Corollary~\ref{cor:tau=w-concordance} above, for $i=0,1$, there are order $n$ Whitney towers $\cW_i$
and $\cW'_i$ bounding $L_i$ and $L'_i$ such that $\tau_n(\cW_0)=\tau_n(\cW_1)$ and $\tau_n(\cW'_0)=\tau_n(\cW'_1)$.
By Lemma~\ref{lem:exists-tower-sum} just below, $L_i\#_{\beta_i} L'_i$ bounds $\cW_i^\#$ for $i=0,1$, with
$$
\tau_n(\cW_0^\#)=\tau_n(\cW_0)+\tau_n(\cW'_0)=\tau_n(\cW_1)+\tau_n(\cW'_1)=\tau_n(\cW_1^\#)
$$
so again by Corollary~\ref{cor:tau=w-concordance}, $L_0\#_{\beta_0} L'_0$ is order $n+1$ Whitney tower concordant to 
$L_1\#_{\beta_1} L'_1$, hence $L_0\#_{\beta_0} L'_0$ 
and $L_1\#_{\beta_1} L'_1$ represent the same element of $\W_n$.
\end{proof}

\begin{lem}\label{lem:exists-tower-sum}
If $L$ and $L'$ bound order $n$ (twisted) Whitney towers $\cW$ and $\cW'$ in $B^4$, then for any $\beta$ there exists an order $n$ (twisted)
Whitney tower $\cW^\#\subset B^4$ bounded by $L\#_\beta L'$, such that 
$t(\cW^\#)=t(\cW)\amalg t(\cW')$, where $t(\cV)$ denotes the intersection forest of a Whitney tower $\cV$ as above in subsection~\ref{subsec:int-forests}.
\end{lem}
\begin{proof}
Let $B$ and $B'$ be the $3$--balls in the link complements used to form the $S^3$ containing $L\#_\beta L'$.  Then gluing together the two $4$--balls containing $\cW$ and $\cW'$ along $B$ and $B'$ forms $B^4$ containing $L\#_\beta L'$ in its boundary. Take $\cW^\#$ to be the boundary band sum of $\cW$ and $\cW'$ along the order zero disks guided by the bands 
$\beta$, with the interiors of the bands perturbed slightly into the interior of $B^4$. 
It is clear that $t(\cW^\#)$ is just the disjoint union $t(\cW)\amalg t(\cW')$ since no new singularities have been created.
\end{proof}

\subsection{Definition of the realization maps}\label{subsec:realization-maps}
The realization maps $R_n$ are defined as follows: Given any group element $g\in\cT_n$, by Lemma~\ref{lem:realization-of-geometric-trees} just below there exists an $m$-component link $L\subset S^3$ bounding
an order $n$ Whitney tower $\cW\subset B^4$ such that $\tau_n(\cW)=g\in\cT_n$.
Define $R_n(g)$ to be the class determined by $L$ in $\W_n$.
This is well-defined (does not depend on the choice of such $L$) by Corollary~\ref{cor:tau=w-concordance}. 
The twisted realization map $R^\iinfty_n$ is defined via Lemma~\ref{lem:realization-of-geometric-trees} the same way using twisted Whitney towers.

\begin{lem}\label{lem:realization-of-geometric-trees}
For any disjoint union $\amalg_p\ \epsilon_p \cdot  t_p \,\, + \amalg_J\ \omega (W_J) \cdot  J^\iinfty$ there exists an $m$-component link $L$ bounding a twisted  Whitney tower $\cW$
with intersection forest $t(\cW)= \amalg_p\ \epsilon_p \cdot  t_p \,\, + \amalg_J \ \omega (W_J) \cdot  J^\iinfty$. 
If the disjoint union contains no $\iinfty$-trees then all Whitney disks in $\cW$ are framed.
\end{lem}
Note that if in the disjoint union all non-$\iinfty$ trees are order at least $n$ and all $\iinfty$-trees are order at least $n/2$ then  $\cW$ will have order $n$.

\begin{proof}
It suffices to consider the cases where the disjoint union consists of just a single (signed) tree or $\iinfty$-tree  since by Lemma~\ref{lem:exists-tower-sum} any sum of such trees can then be realized by band sums of links.

The following algorithm, in the untwisted case, is the algorithm called "Bing-doubling along a tree" by Cochran and used in Section 7 of \cite{C} and Theorem 3.3 of \cite{C1} to produce links in $S^3$  with prescribed (first non-vanishing) Milnor invariants.

\begin{figure}
\centerline{\includegraphics[scale=.325]{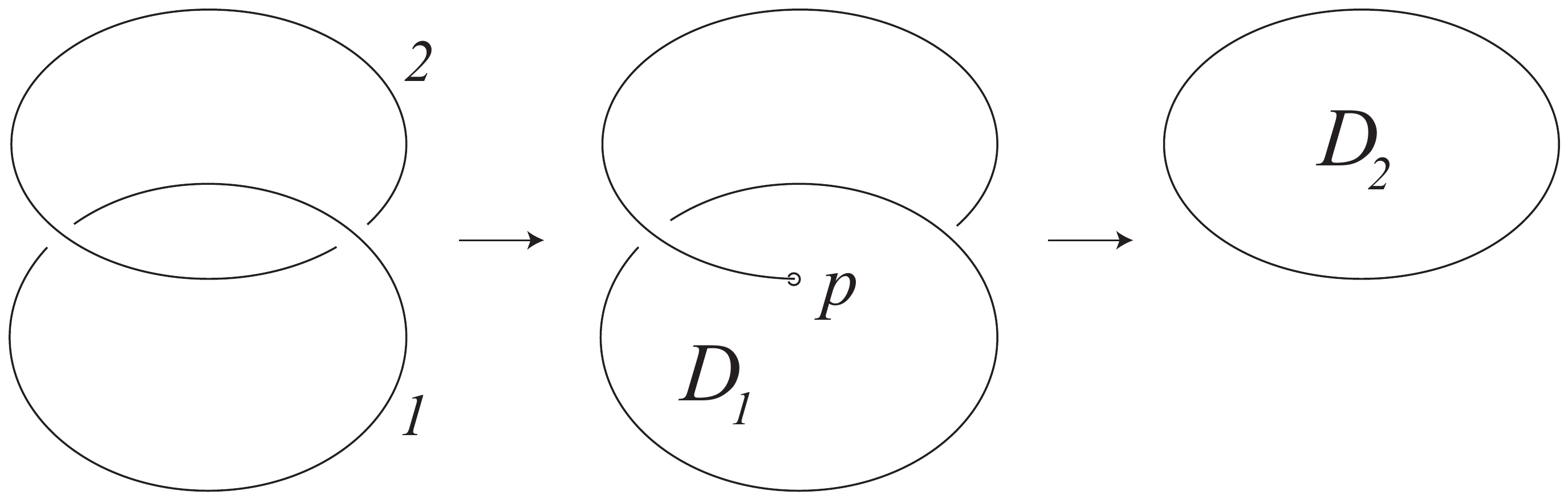}}
         \caption{Pushing into $B^4$ from left to right: A Hopf link in $S^3=\partial B^4$ bounds embedded disks $D_1\cup D_2\subset B^4$ which intersect in a point $p$, with $t_p=\langle 1,2 \rangle$. }
         \label{fig:Hopf-disk}
\end{figure}

\textbf{Realizing order zero trees and $\iinfty$-trees.}
A 0-framed Hopf link bounds an order zero Whitney tower $\cW=D_1\cup D_2\subset B^4$, where the two embedded disks $D_1$ and $D_2$ have a single interior intersection point $p$ with $t_p=\langle 1,2 \rangle = 1 -\!\!\!-\!\!\!- \,2 $ (see Figure~\ref{fig:Hopf-disk}).
Assuming appropriate fixed orientations of $B^4$ and $S^3$, the sign $\epsilon_p$ associated to $p$ is the usual sign of the Hopf link. So taking a 0-framed $(m-2)$-component trivial link together with a Hopf link (as the $i$th and $j$th components) gives an $m$-component link $L$ bounding
$\cW$ with $t(\cW)=\epsilon_p\cdot\langle i,j \rangle = \epsilon_p\cdot i -\!\!\!- \,j $, for any 
$\epsilon_p=\pm 1$, and $i\neq j$.

To realize the tree $\pm\ i -\!\!\!-\!\!\!- \,i $, we can use the unlink with framings 0, except that the component labeled by the index $i$ has framing $\pm 2$. Similarly, if the component has framing $\pm 1$ then the resulting tree is $\pm \ \iinfty -\!\!\!-\!\!\!-\,i $.

\begin{figure}
\centerline{\includegraphics[scale=.4]{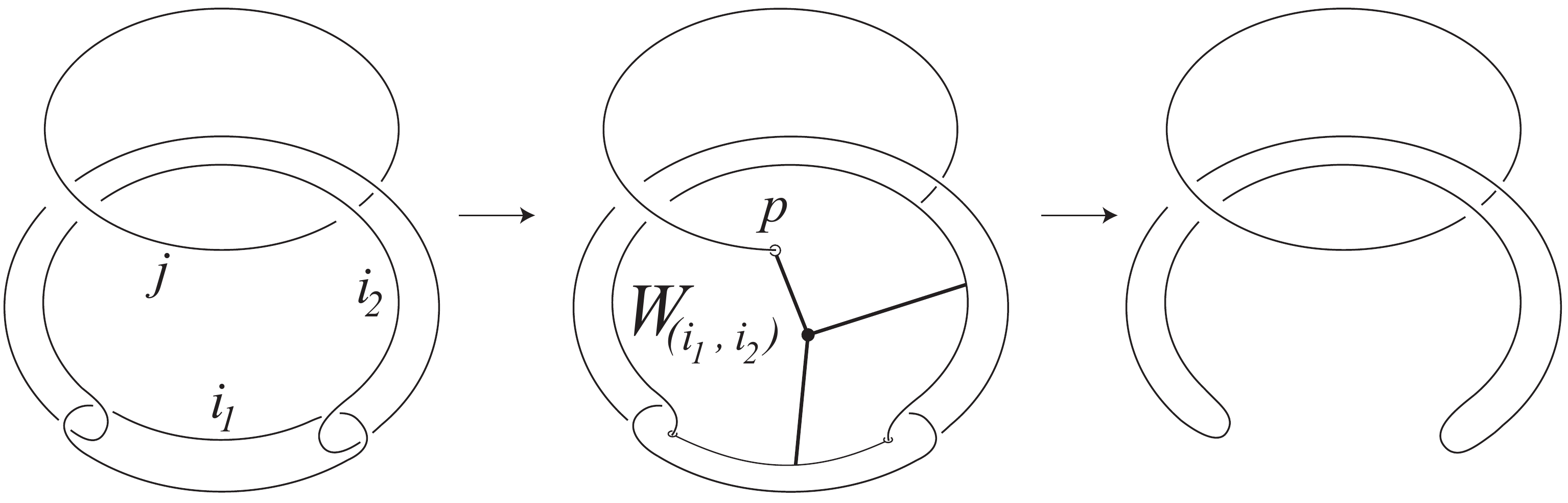}}
         \caption{Pushing into $B^4$ from left to right: The disks $ D_{i_2}$ and $D_j$ extend to the
         right-most picture where they are completed by capping off the unlink. The disk $D_{i_1}$ only extends to the middle picture where the intersections between $D_{i_1}$ and $ D_{i_2}$ are paired by the Whitney disk $W_{(i_1,i_2)}$, that has a single interior intersection $p\in W_{(i_1,i_2)}\cap D_j$ with $t_p=\langle (i_1,i_2),j \rangle$.}
         \label{fig:Borromean-Bing-Hopf}
\end{figure}

\textbf{Realizing order $1$ trees.}
Consider now a link $L$ whose $i$th and $j$th components form a Hopf link $L^i\cup L^j$ bounding disks
$D_i\cup D_j\subset B^4$ with transverse intersection $p=D_i\cap D_j$. Assume that $D_i\cup D_j$ extends
to an order zero Whitney tower $\cW$ bounded by $L$ with $t(\cW)=\epsilon_p\cdot t_p=\epsilon_p\cdot\langle i,j \rangle$.

Replacing $L^i$ by an untwisted Bing-double
$L^{i_1}\cup L^{i_2}$
results in a new sublink of Borromean rings $L^{i_1}\cup L^{i_2}\cup L^j$ bounding disks $D_{i_1}\cup D_{i_2}\cup D_j$
as indicated in Figure~\ref{fig:Borromean-Bing-Hopf}, with
$D_{i_1}$ and $ D_{i_2}$ intersecting in a canceling pair of intersections paired by an order $1$ Whitney disk $W_{(i_1,i_2)}$, which can be formed from $D_i$ with a small collar removed,
so that $W_{(i_1,i_2)}$ has a single intersection with $D_j$ corresponding to the original $p=D_i\cap D_j$.
(One can think of $D_{i_1}$ and $ D_{i_2}$ as being formed by the trace of the obvious 
pulling-apart homotopy that shrinks $L^{i_1}$ and $L^{i_2}$ down in a tubular neighborhood of $L^i$, 
with the canceling pair of intersections between $D_{i_1}$ and $D_{i_2}$ being created as the clasps are pulled apart.)

The effect of this Bing-doubling operation on the intersection forest is that the original order zero
$t_p=\langle i,j \rangle$ has given rise to the order~$1$ tree $\langle (i_1,i_2),j \rangle$. 
Switching the orientation on one of the new components changes the sign of $p$, as can be checked using our orientation conventions.
By relabeling and/or banding together components of this new link any labels on this order~$1$ tree
can be realized. Since the doubling was untwisted, $W_{(i_1,i_2)}$ is framed (see Figures \ref{fig:Bing-unlink-W-disk} and \ref{fig:Bing-unlink-W-disk-twisting}), so the Whitney tower
bounded by the new link
is order $1$. 

\begin{figure}
\centerline{\includegraphics[scale=.35]{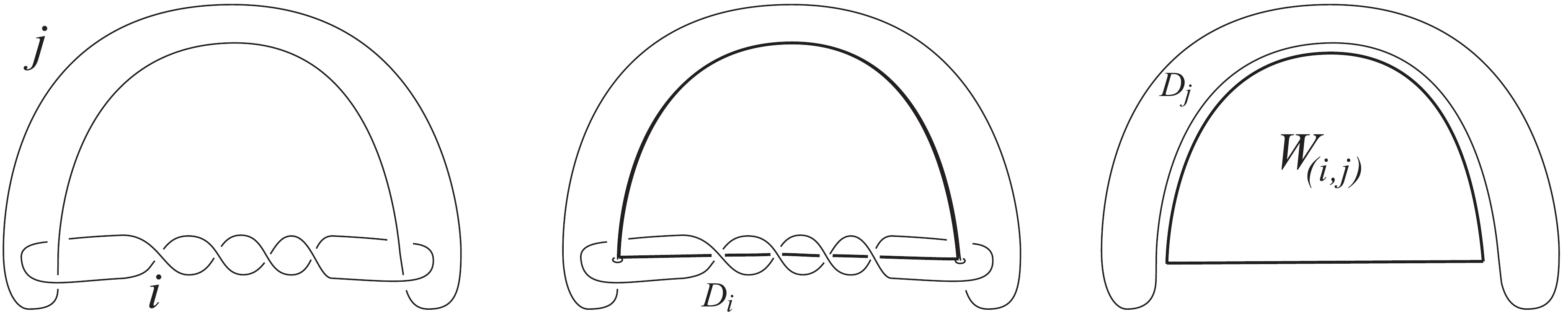}}
         \caption{Pushing into $B^4$ from left to right: An $i$- and $j$-labeled $n$-twisted Bing-double (case $n=2$) of the unknot in 
         $S^3=\partial B^4$
         bounds disks $D_i$ and $D_j$ whose intersections are paired by a Whitney disk $W_{(i,j)}$. 
         $D_j$ extends to the right-hand picture but $D_i$ only extends to the middle picture, where the boundary of $W_{(i,j)}$ is indicated by the dark arcs. The rest of $W_{(i,j)}$ extends into the right-hand picture where disjointly embedded disks bounded by the unlink complete both $W_{(i,j)}$ and $D_j$. The interior of $W_{(i,j)}$ is embedded and disjoint from both $D_i$ and $D_j$. Figure~\ref{fig:Bing-unlink-W-disk-twisting} shows that $W_{(i,j)}$ is twisted, with $\omega(W_{(i,j)})=n$.}
         \label{fig:Bing-unlink-W-disk}
\end{figure}

\textbf{Realizing order $n$ trees.}
Since any order $n$ tree can be gotten from some order $n-1$ tree by attaching two new edges to a univalent vertex as in the previous paragraph, it follows inductively that 
any order $n$ tree is the intersection forest of a Whitney tower bounded by some link. (First create a distinctly-labeled tree of the desired `shape' by doubling, then correct the labels by interior band-summing.)

\textbf{Realizing $\iinfty$-trees of order $1$.}
As illustrated (for the case $n=2$) in Figures~\ref{fig:Bing-unlink-W-disk} and \ref{fig:Bing-unlink-W-disk-twisting}, the $n$-twisted Bing-double
of the unknot (with components labeled $i$ and $j$) bounds an order $2$ twisted Whitney tower $\cW$ with 
$t(\cW)=n\cdot ( i,j )^\iinfty=n\cdot  \iinfty \!-\!\!\!\!\!-\!\!\!\!<^{\,i}_{\,j}$. Banding together the two components
would yield a knot realizing $(i,i)^\iinfty$. 

\begin{figure}
\centerline{\includegraphics[scale=.4]{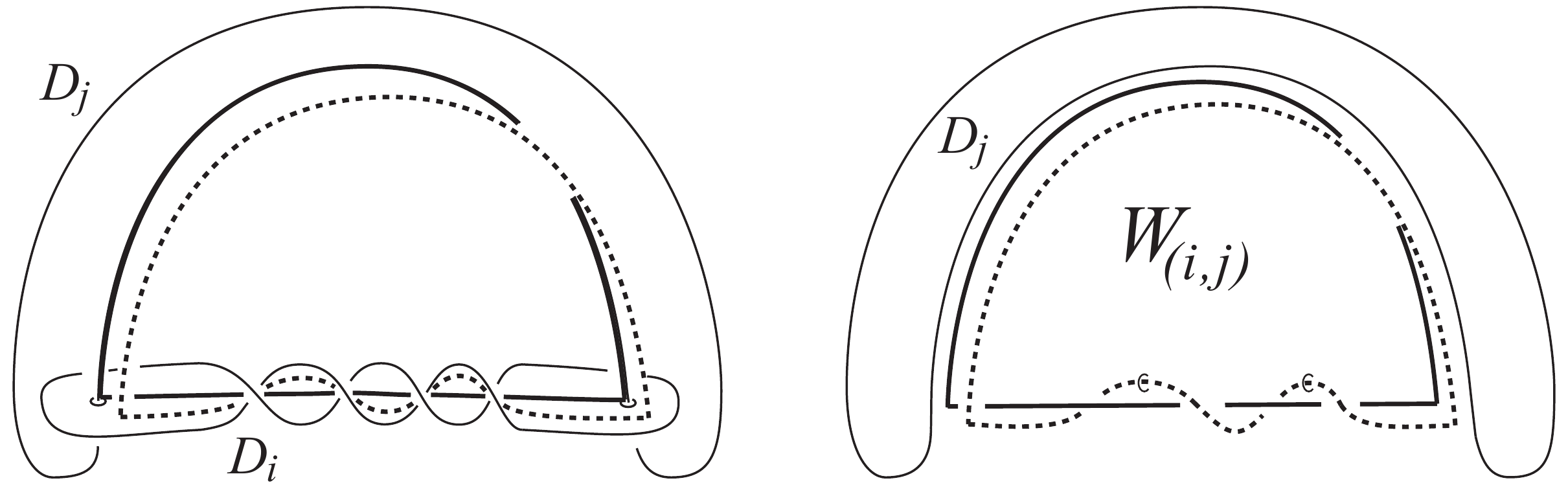}}
         \caption{The Whitney section over $\partial W_{(i,j)}$ (from Figure~\ref{fig:Bing-unlink-W-disk}) is indicated by the dashed arcs on the left. The twisting 
         $\omega (W_{(i,j)})=n$ (the obstruction to extending the Whitney section across the Whitney disk) corresponds to the 
         $n$-twisting of the Bing-doubling operation.}
         \label{fig:Bing-unlink-W-disk-twisting}
\end{figure}

\textbf{Realizing $\iinfty$-trees of order $n$.}
By applying iterated untwisted Bing-doubling operations to the $i$- and $j$-labeled components 
of the order $1$ case, one can construct for any rooted tree $( I,J )$ a link bounding a twisted Whitney tower
$\cW$ with $t(\cW)=n \cdot ( I,J )^\iinfty$. For instance, if in the construction of Figure~\ref{fig:Bing-unlink-W-disk} the $j$-labeled link component is replaced by an untwisted Bing-double, then the disk $D_j$ in that construction would be replaced by a (framed) Whitney disk $W_{(j_1,j_2)}$, and the $n$-twisted $W_{(i,j)}$ would be replaced by an $n$-twisted
$W_{(i,(j_1,j_2))}$.   (As for non-$\iinfty$ trees above, first create a distinctly-labeled tree of the desired `shape' by doubling, then correct the labels by interior band-summing.) 
\end{proof}


\subsection{Surjectivity of the realization maps}\label{subsec:realization-maps}
In this section we will prove Theorems~\ref{thm:R-onto-W} and~\ref{thm:twisted}: The realization maps $R_n$ and $R^\iinfty_n$ are epimorphisms.

We will show, moreover, that $\W_n$ is the set of framed links $L\in\bW_n$ modulo the relation that $[L_1]=[L_2]\in\W_n$ if and only if $L_1 \# -L_2$ lies in $\bW_{n+1}$, 
where $-L$ is the mirror image of $L$ with reversed framing.

\begin{proof} From Lemma~\ref{lem:link-sum-well-defined} the band sum of links gives well-defined operations in $\W_n$ and $\W^\iinfty_n$ which are clearly associative and commutative, with the $m$-component unlink representing an identity element.  The realization maps are homomorphisms by Lemma~\ref{lem:exists-tower-sum} and surjectivity is proven as follows: Given any link $L\in \bW_n$, choose a Whitney tower $\cW$ of order $n$ with boundary $L$ and compute $\tau:=\tau_n(\cW)$. Then take $L':= R_n(\tau)$, a link that's obviously in the image of $R_n$ and for which we know a Whitney tower $\cW'$ with boundary $L'$ and $\tau(\cW') = \tau$.
By Corollary~\ref{cor:tau=w-concordance} it follows that $L$ and $L'$ represent the same element in $\W_n$.

If $L_0$ and $L_1$ represent the same element of $\W_n$ (resp. $\W^\iinfty_n$), then by Corollary~\ref{cor:tau=w-concordance} there exist order $n$ (twisted) Whitney towers $\cW_0$ and $\cW_1$ in $B^4$ bounded by $L_0$ and $L_1$ respectively such that
$\tau_n(\cW_0)=\tau_n(\cW_1)\in\cT_n$ (resp. $\tau^\iinfty_n(\cW_0)=\tau^\iinfty_n(\cW_1)\in\cT^\iinfty_n$). We want to show that $L_0\#-L_1$ bounds an order~$n+1$ (twisted) Whitney tower, which will follow from Lemma~\ref{lem:exists-tower-sum} and the ``order-raising'' Theorem~\ref{thm:framed-order-raising-on-A} (respectively Theorem~\ref{thm:twisted-order-raising-on-A}) if $-L_1$ bounds an order $n$ (twisted) Whitney tower $\overline{\cW_1}$ such that $\tau_n(\overline{\cW_1})=-\tau_n(\cW_1)\in\cT_n$ 
(resp. $\tau^\iinfty_n(\overline{\cW_1})=-\tau^\iinfty_n(\cW_1)\in\cT^\iinfty_n$). If $r$ denotes the reflection on $S^3$ which sends $L_1$ to 
$-L_1$, then the product $r\times\id$ of $r$ with the identity is an involution on $S^3\times I$, and the image $r\times\id(\cW_1)$
of $\cW_1$ is such a $\overline{\cW_1}$. To see this, observe that $r\times\id$ switches the signs of all transverse intersection points,
and is an isomorphism on the oriented trees in $\cW_1$; and hence switches the signs of all Whitney disk framing obstructions (which can be computed as intersection numbers between Whitney disks and their push-offs) -- note that $r\times\id$ is only being applied to 
$\cW_1$, while $S^3\times I$ is fixed.   

Assume now that $L_0\#-L_1\subset S^3$ bounds an order~$n+1$ (twisted) Whitney tower $\cW\subset B^4$. By the definition of connected sum, $S^3$ decomposes as the union of two disjoint $3$--balls $B_0$ and $B_1$ containing $L_0$ and $-L_1$, joined together by the $S^2\times I$ through which passes the bands guiding the connected sum. 
Taking another $4$--ball with the same decomposition of its boundary $3$--sphere, and gluing the $4$--balls together by identifying the 
boundary $2$--spheres of the $3$--balls, and identifying the $S^2\times I$ subsets by the identity map, forms $S^3\times I$ containing
an order $n+1$ (twisted) Whitney tower concordance between $L_0$ and $-L_1$ which consists of $\cW$ together with the parts of the connected-sum bands that are contained in $S^2\times I$.  
\end{proof}


\section{Implications of the twisted IHX construction}\label{sec:proof-twisted-thm}
This section is mostly dedicated to proving the ``twisted order-raising'' Theorem~\ref{thm:twisted-order-raising-on-A} of Section~\ref{sec:w-towers}, which was used in Section~\ref{sec:realization-maps} to construct the twisted realization maps.
A key step in the proof given in section~\ref{subsec:twisted-order-raising-thm-proof} involves a geometric realization of the 
twisted IHX relation as described in Lemma~\ref{lem:twistedIHX} below. 

In section~\ref{subsec:boundary-twisted-IHX-lemma}, Lemma~\ref{lem:twistedIHX} is also used to show how any order $2n$ twisted Whitney tower can be converted into an order $2n-1$ framed Whitney tower. This result (Lemma~\ref{lem:boundary-twisted-IHX}) will be used later in the proof of Theorem~\ref{thm:exact-sequence} in Section~\ref{subsec:proof-thm-exact-sequences}.

Then in section~\ref{sec:reduced-groups-obstruction-theory}, Lemma~\ref{lem:twistedIHX} is used again to prove the order-raising
Theorem~\ref{thm:framed-order-raising-mod-Delta} in the reduced setting.

Recall the statement of Theorem~\ref{thm:twisted-order-raising-on-A}: 
If a collection $A$ of properly immersed surfaces in a simply connected $4$--manifold supports an order $n$ twisted Whitney tower $\cW$
with $\tau^\iinfty_n(\cW)=0\in\cT^\iinfty_n$, then $A$ is regularly homotopic (rel $\partial$) to $A'$ supporting an order $n+1$ twisted Whitney tower.

Recall also from Definition~\ref{def:intersection forests} that the intersection forest $t(\cW)$ of an order 
$n$ twisted Whitney tower $\cW$
is a disjoint union of signed trees which can be considered to be immersed in $\cW$. The order $n$ trees in $t(\cW)$ (together with the order $n/2$ 
$\iinfty$-trees if $n$ is even) represent $\tau_n^\iinfty(\cW)\in\cT^\iinfty_n$, and the proof of Theorem~\ref{thm:twisted-order-raising-on-A} involves controlled manipulations of $\cW$ which first convert $t(\cW)$ into ``algebraically canceling'' pairs of isomorphic trees with opposite signs, and then exchange these for ``geometrically canceling'' intersection points which are paired by a new layer of Whitney disks. We pause here to clarify these
notions:


\textbf{Algebraic versus geometric cancellation:}
Note that $t(\cW)$ is a combinatorial object which by Lemma~\ref{lem:split-w-tower} above can be considered geometrically as the image of an embedding in $\cW$. 
If Whitney disks $W_I$ and $W_J$ in $\cW$ intersect transversely in a pair of points $p$ and $p'$, then $t_p$ and $t_{p'}$ are isomorphic (as labeled, oriented trees). If $p$ and $p'$ have opposite signs, and if the ambient $4$--manifold is simply connected, then there exists a Whitney disk $W_{(I,J)}$ pairing $p$ and $p'$,
and we say that $\{\,p\,,\,p'\,\}$ is a \emph{geometrically canceling} pair.
In this setting we also refer to $\{\,\epsilon_p\cdot t_p\,,\,\epsilon_{p'}\cdot t_{p'}\,\}$ as a geometrically canceling pair of signed trees in $t(\cW)$ (regarding them as subsets of $\cW$ associated to the geometrically canceling pair of points).

On the other hand, given transverse intersections $p$ and $p'$ in $\cW$ with $ t_p = t_{p'}$ 
(as labeled oriented trees) and $\epsilon_p=-\epsilon_{p'}$, we say that $\{\,p\,,\,p'\,\}$ is an \emph{algebraically canceling} pair of intersections, and similarly call $\{\,\epsilon_p\cdot t_p\,,\,\epsilon_{p'}\cdot t_{p'}\,\}$ an algebraically canceling pair of signed trees in $t(\cW)$. Changing the orientations at a \emph{pair} of trivalent vertices in any tree $t_p$ does not change its value in $\cT$ by the AS relations, and (as discussed in 3.4 of \cite{ST2}) such orientation changes can be realized by changing orientations of Whitney disks in $\cW$ together with our orientation conventions (\ref{subsec:w-tower-orientations}).

Any geometrically canceling pair is also an algebraically canceling pair, but the converse is clearly not true as an algebraically canceling pair can have \emph{corresponding trivalent vertices} lying in \emph{different Whitney disks}.  A process for converting algebraically canceling pairs into geometrically canceling pairs by manipulations of the Whitney tower is described in 4.5 and 4.8 of \cite{ST2}.

Similarly, if a pair of twisted Whitney disks $W_{J_1}$ and $W_{J_2}$ have isomorphic (unoriented) trees $J^\iinfty_1$ and $J^\iinfty_2$ with opposite twistings $\omega(W_{J_1})=-\omega (W_{J_2})$, then the Whitney disks form an 
\emph{algebraically canceling} pair (as do the corresponding signed $\iinfty$-trees in $t(\cW)$). Note that the orientations of the $\iinfty$-trees are not relevant here by the independence of $\omega(W)$ from the orientation of $W$ and the symmetry relations in $\cT^\iinfty$. A geometric construction for eliminating algebraically canceling pairs of twisted Whitney disks from a twisted Whitney tower will be described below.

\subsection{Proof of the twisted order-raising Theorem~\ref{thm:twisted-order-raising-on-A}}\label{subsec:twisted-order-raising-thm-proof}

To motivate the proof we summarize here how the methods of 
\cite{CST,S1,ST2} (as described in Section~4 of \cite{ST2})
apply in the framed setting to prove the analogous order-raising theorem in framed setting (Theorem~\ref{thm:framed-order-raising-on-A} of Section~\ref{sec:w-towers}): The first part of the proof changes the intersection forest $t(\cW)$ so that all trees occur in algebraically canceling pairs by using the $4$--dimensional IHX construction of \cite{CST} to realize IHX relators, and by adjusting Whitney disk orientations as necessary to realize AS relations. The second part of the proof uses the Whitney move IHX construction of \cite{S1}
to ``simplify'' the shape of the algebraically canceling pairs of trees. Then the third part of the proof uses controlled
homotopies to exchange the simple algebraic canceling pairs for geometrically canceling intersection points which are paired by a new layer of Whitney disks as described in 4.5 of \cite{ST2}. All constructions only change the order $0$ surfaces by regular homotopies consisting of finger moves, Whitney moves, and isotopies.

Extending these methods to the present twisted setting will require two variations: realizing the new relators in 
$\cT_n^\iinfty$, and achieving an analogous geometric cancellation for twisted Whitney disks corresponding to algebraically canceling pairs of (simple) $\iinfty$-trees. We will concentrate on these new variations, referring the reader to   
\cite{CST,S1,ST2} for the other parts just mentioned.


\textbf{Notation and conventions:}

By Lemma~\ref{lem:split-w-tower} it may be assumed that $\cW$ is split at each stage of the constructions throughout the proof, so that all trees in $t(\cW)$ are embedded in $\cW$.  
In spite of modifications, $\cW$ will not be renamed during the proof.
Throughout this section we
will notate elements of $t(\cW)$ as formal sums, representing disjoint union by juxtaposition.

Note that if $\cW$ is an order $n$ twisted Whitney tower, then the intersection forest $t(\cW)$ may contain
higher order trees and $\iinfty$-trees in addition to those representing $\tau_n^\iinfty(\cW)\in\cT_n^\iinfty$. These higher-order
elements of $t(\cW)$ can be ignored throughout the proof for the following reasons: On the one hand, in a split $\cW$ all the 
constructions leading to the elimination of unpaired order $n$ intersections (and twisted Whitney disks of order $n/2$)
of $\cW$
can be carried out away from any higher-order
elements of $t(\cW)$. Alternatively, one could first exchange all twisted Whitney disks of order greater than $n/2$ for unpaired intersections of order greater than $n$ by boundary-twisting (Figure~\ref{boundary-twist-and-section-fig}). Then, all intersections of order greater than $n$ can be converted into into many algebraically canceling pairs of order $n$ intersections by repeatedly ``pushing down'' unpaired intersections
until they reach the order zero disks, as illustrated for instance in Figure~12 of \cite{S2} (assuming, as we may, that $\cW$ contains no Whitney disks of order greater than $n$). 

Thus, we can and will assume throughout the proof that $t(\cW)$ represents $\tau_n^\iinfty(\cW)\in\cT_n^\iinfty$.

\textbf{The odd order case:} Given $\cW$ of order $2n-1$ with $\tau^\iinfty_{2n-1}(\cW)=0\in\cT^\iinfty_{2n-1}$, it will suffice to modify
$\cW$ --- while only creating unpaired intersections of order at least $2n$ and twisted Whitney disks of order at 
least $n$ --- so that all order $2n-1$ trees in $t(\cW)$ come in algebraically canceling
pairs of trees (since by
\cite{ST2} the corresponding algebraically canceling pairs of order $2n-1$ intersection points can be
exchanged for geometrically canceling intersections which are
paired by Whitney disks, as mentioned just above).

\begin{figure}
\centerline{\includegraphics[width=125mm]{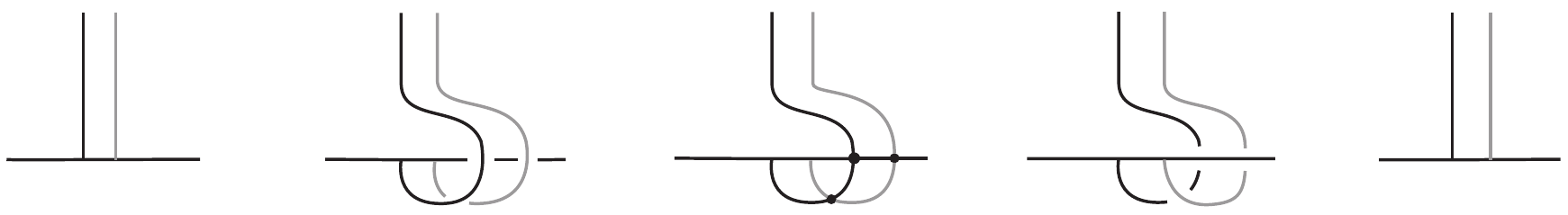}}
         \caption{Boundary-twisting a Whitney disk $W$ changes $\omega (W)$ by $\pm 1$ and creates an intersection point with one of the sheets paired by $W$. The horizontal arcs trace out part of the sheet, the dark non-horizontal arcs
         trace out the newly twisted part of a collar of $W$, and the grey arcs indicate part of the Whitney section over $W$. The bottom-most
         intersection in the middle picture corresponds to the $\pm 1$-twisting created  by the move.}
         \label{boundary-twist-and-section-fig}
\end{figure}

Since $\tau_{2n-1}^\iinfty(\cW)=0\in\cT^\iinfty_{2n-1}$, the intersection forest $t(\cW)$ is in the span of 
IHX and boundary-twist relators, after choosing Whitney disk orientations to realize AS relations as necessary. By locally creating intersection
trees of the form $+I -H  +X$ using  the 4-dimensional geometric IHX theorem of \cite{CST} 
(and by
choosing Whitney disk orientations to realize AS relations as needed), 
$\cW$
can be modified so that all order $2n-1$ trees in $t(\cW)$ either come in algebraically
canceling pairs, or are boundary-relator trees of the form
$\pm\langle (i,J),J \rangle$. 

For each tree of the form 
$t_p=\pm\langle (i,J),J \rangle$ we can create an algebraically canceling
 $t_{p'}=\mp\langle (i,J),J \rangle$ at the cost of only creating order~$n$
$\iinfty$-trees as follows. First use Lemma~14 of \cite{ST2} (Lemma~{3.6} of \cite{S1}) to move
the unpaired intersection point $p$ so that $p\in W_{(i,J)}\cap
W_J$. Now, by boundary-twisting $W_{(i,J)}$ into its supporting
Whitney disk $W'_J$ (Figure~\ref{boundary-twist-and-section-fig}), 
an algebraically canceling intersection $p'\in
W_{(i,J)}\cap W'_J$ can be created at the cost of changing the
twisting $\omega (W_{(i,J)})$ by $\pm 1$. 
Since $\langle (i,J),J \rangle$ has an order $2$ symmetry, the canceling sign can always be realized by a Whitney disk orientation choice. 
This algebraic cancellation of $t_p$ has been achieved at the cost of only adding to $t(\cW)$ the 
order $n$ $\iinfty$-tree $(i,J)^\iinfty$ corresponding to the $\pm 1$-twisted order $n$ Whitney disk 
$W_{(i,J)}$.


Having arranged that all the order $2n-1$ trees in $t(\cW)$ occur in algebraically canceling pairs,
applying the tree-simplification and geometric cancellation described in \cite{ST2} to all these algebraically canceling pairs 
yields an order $2n$ twisted Whitney tower $\cW'$.

\textbf{The even order case:} For $\cW$ of order $2n$ with $\tau^\iinfty_{2n}(\cW)=0\in\cT^\iinfty_{2n}$, we arrange for $t(\cW)$ to
consist of only algebraically canceling pairs of generators by
realizing all relators in $\cT_{2n}^\iinfty$,  then construct an order
$2n+1$ twisted Whitney tower by introducing a new method
for geometrically canceling the pairs of twisted Whitney disks 
(while the algebraically canceling pairs of non-$\iinfty$ trees lead to geometrically canceling intersections as before):

First of all, the order $0$ case corresponding to linking numbers is easily checked, so we will assume $n\geq 1$.
The IHX relators and AS relations for non-$\iinfty$ trees can be realized as usual.
 Note that any signed tree $\epsilon\cdot J^\iinfty\in t(\cW)$ does not depend on the orientation of the tree $J$ because changing the orientation on the corresponding twisted Whitney disk $W_J$ does not change $\omega (W_J)$.

For any rooted tree $J$ the relator $\langle J,J \rangle-2\cdot J^\iinfty$ corresponding to the 
interior-twist relation can be realized as follows. Use finger moves to create a clean
framed Whitney disk $W_J$. Performing a positive interior twist on $W_J$
as in Figure~\ref{InteriorTwistPositiveEqualsNegative-fig}  creates a
self-intersection $p\in W_J\cap W_J$ with $t_p=\langle J,J \rangle$ and
changes the twisting $\omega(W_J)$ of $W_J$ to $-2$. 
The negative of the relator is similarly constructed starting with a negative twist.
\begin{figure}[h]
\centerline{\includegraphics[width=125mm]{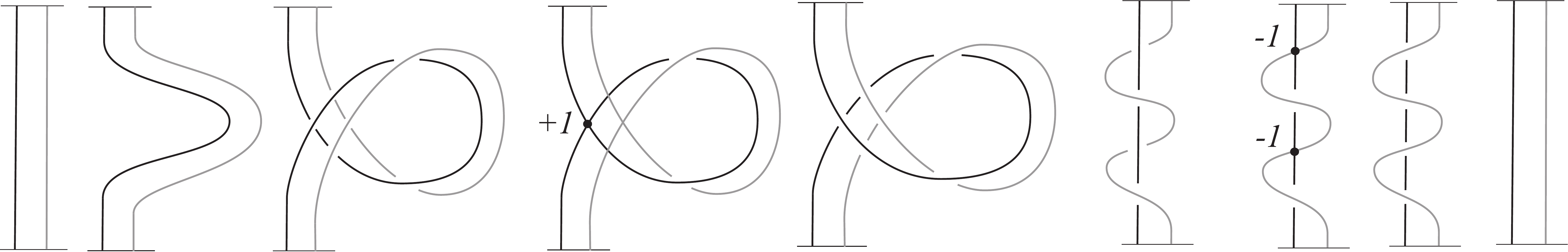}}
         \caption{A $+1$ interior twist on a Whitney disk changes the twisting by $-2$, as is seen in the pair of $-1$ intersections between the black vertical slice of the Whitney disk and the grey slice of a Whitney-parallel copy. Note that the pair of (positive) black-grey intersections near the $+1$ intersection is just an artifact of the immersion of the normal bundle into $4$--space and does not contribute to the relative Euler number.}
         \label{InteriorTwistPositiveEqualsNegative-fig}
\end{figure}

The relator $-I^\iinfty+H^\iinfty+X^\iinfty-\langle H,X \rangle$ corresponding to
the twisted IHX relation is realized as follows. For any
rooted tree $I$, create a clean framed Whitney disk $W_I$ by
finger moves. Then split this framed Whitney disk using the twisted
finger move of Lemma~\ref{lem:split-w-tower} into two clean twisted
Whitney disks with twistings $+1$ and $-1$, and associated signed
$\iinfty$-trees $+I^\iinfty$ and $-I^\iinfty$, respectively. The next step is
to perform a $+1$-twisted version (described in Lemma~\ref{lem:twistedIHX} below)
of the ``Whitney move IHX'' construction of Lemma~{7.2}
in \cite{S1}, which will replace the $+1$-twisted Whitney disk by two
$+1$-twisted Whitney disks having $\iinfty$-trees $+H^\iinfty$ and $+X^\iinfty$,
and containing a single negative intersection point with tree
$-\langle H,X \rangle$, where $H$ and $X$ differ locally from $I$ as in the usual IHX relation.
Thus, any Whitney tower can be modified to create
exactly the relator $-I^\iinfty+H^\iinfty+X^\iinfty-\langle H,X \rangle$, for any rooted tree $I$.
The negative of the relator can be similarly realized by using Lemma~\ref{lem:twistedIHX}
applied to the
$-1$-twisted $I$-shaped Whitney disk.

So since $\tau_{2n}^\iinfty(\cW)$ vanishes, it may be arranged, by realizing relators as above, that
all the trees in $t(\cW)$ occur in algebraically canceling pairs.
Now, by repeated applications of Lemma~\ref{lem:twistedIHX} below, the algebraically canceling pairs of clean $\pm 1$-twisted Whitney disks
can be exchanged for (many) algebraically canceling pairs of clean $\pm 1$-twisted Whitney disks,
all of whose trees are \emph{simple} (right- or left-normed), with the $\iinfty$-label at one end of the tree as illustrated in Figure~\ref{simple-infty-tree-fig} -- this also creates more algebraically canceling pairs of non-$\iinfty$ trees (the ``error term'' trees in
Lemma~\ref{lem:twistedIHX}).

As in the odd case, all algebraically canceling pairs of
intersections with
non-$\iinfty$ trees can be exchanged for geometrically canceling pairs by \cite{ST2}.
To finish building the desired 
order~$2n+1$ twisted Whitney tower, we will describe how to eliminate the remaining algebraically canceling pairs of clean twisted order~$n$ Whitney disks (all having simple trees)
using a construction that bands together Whitney disks and is additive on twistings.
This construction is an iterated elaboration of a construction originally from Chapter~10.8 of \cite{FQ} (which was used to show that
that $\tau_1 \otimes \Z_2$ did not depend on choices of pairing intersections by Whitney disks).

 \begin{figure}[h]
\centerline{\includegraphics[width=85mm]{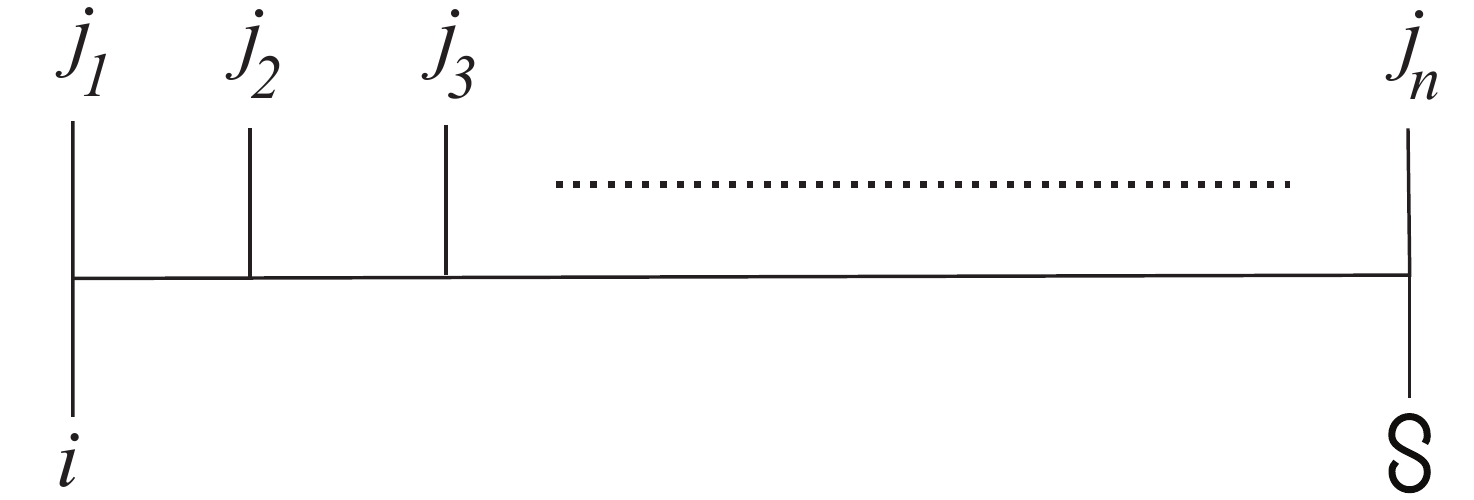}}
         \caption{The simple twisted tree $J_n^\infty$.}
         \label{simple-infty-tree-fig}
\end{figure}

Consider an algebraically canceling pair of clean $\pm 1$-twisted Whitney disks $W_{J_n}$ and $W'_{J_n}$, whose simple $\iinfty$-trees $+ J_n^\iinfty$ and $-J_n^\iinfty$ are as in
Figure~\ref{simple-infty-tree-fig}, using the notation $J_n=(\cdots ((i,j_1),j_2),\cdots , j_n)$. Each trivalent vertex corresponds to a Whitney disk, and we will work from left to right, starting with the order one Whitney disks $W_{(i,j_1)}$ and $W'_{(i,j_1)}$,
banding together Whitney disks of the same order from the two trees, while only creating new unpaired intersections of order greater than $2n$. At the last step, $W_{J_n}$ and $W'_{J_n}$ will be banded together into a single framed clean Whitney disk, providing the desired geometric cancellation. (The reason for working with \emph{simple} trees is that the construction for achieving geometric cancellation requires \emph{connected}
surfaces for certain steps. For instance, the following construction only gets started because the left most trivalent vertices
of an algebraically canceling pair of simple trees correspond to Whitney disks which pair the connected order zero surfaces $D_i$ and $D_{j_1}$.)

\begin{figure}[h]
\centerline{\includegraphics[width=100mm]{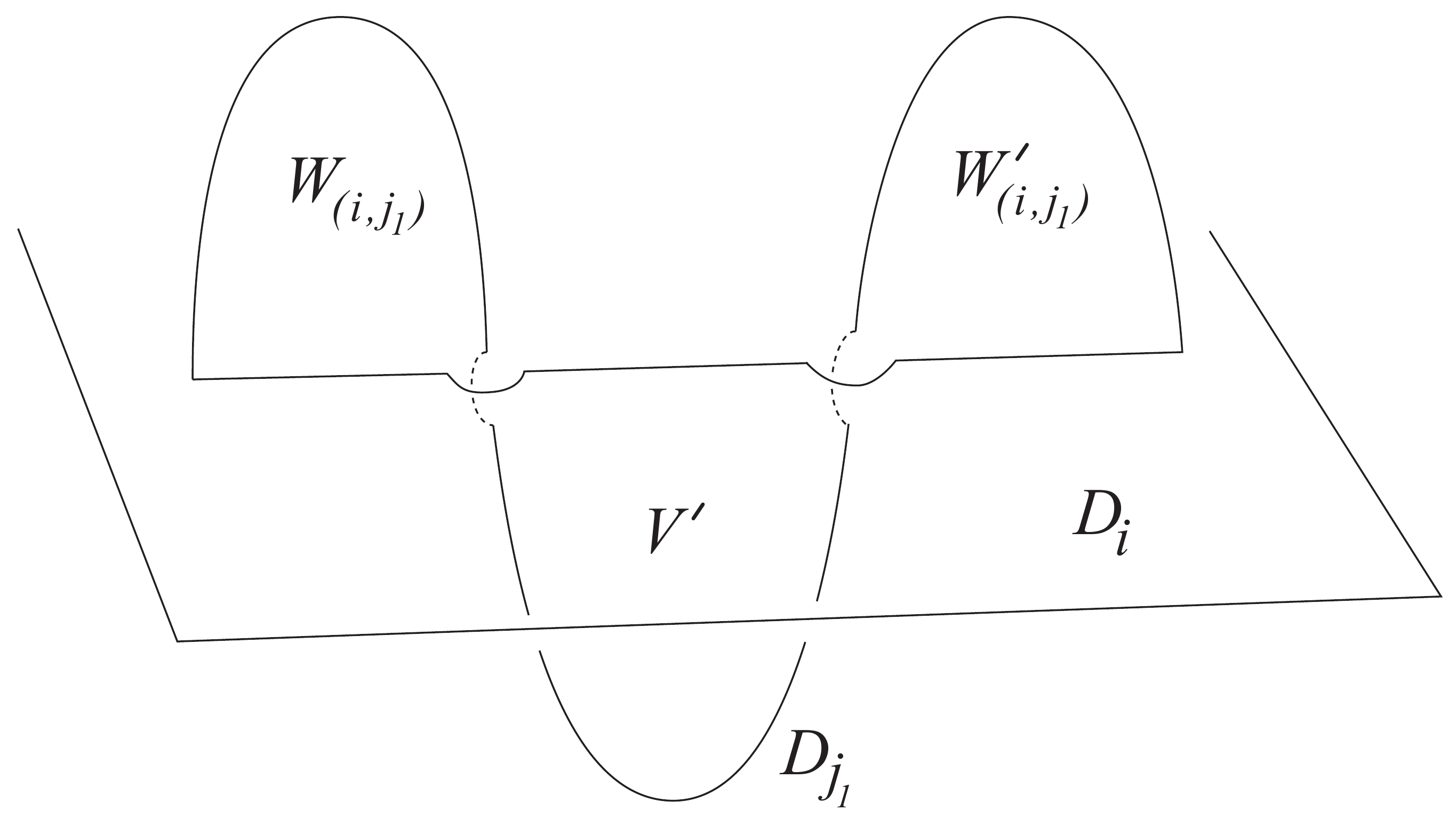}}
         \caption{The Whitney disks $W_{(i,j_1)}$, $W'_{(i,j_1)}$, and $V'$ are 
         banded together to form the Whitney disk
         $W''_{(i,j_1)}$ pairing the outermost pair of intersections between $D_i$ and $D_{j_1}$.
         In the cases $n>1$, the interior of $W''_{(i,j_1)}$ contains two pairs of canceling intersections with $D_{j_2}$ (which are not shown), and supports the sub-towers consisting of the rest of the higher-order Whitney disks (that were supported by $W_{(i,j_1)}$ and $W'_{(i,j_1)}$) corresponding to the trivalent vertices in both trees
$\pm J_n^\infty$.}
         \label{bandedWdisksWithTwistsNoHigherInts}
\end{figure}

To start the construction consider the Whitney disks $W_{(i,j_1)}$ and $W'_{(i,j_1)}$, pairing intersections between the order zero immersed disks $D_i$ and $D_{j_1}$. Let $V$ be another Whitney disk for a canceling pair consisting of one point from each of the points paired by
$W_{(i,j_1)}$ and $W'_{(i,j_1)}$. Figure~\ref{bandedWdisksWithTwistsNoHigherInts} illustrates how a parallel copy $V'$ of $V$ can be banded together
with $W_{(i,j_1)}$ and $W'_{(i,j_1)}$ to form a Whitney disk $W''_{(i,j_1)}$ for the remaining canceling pair. The twisting of $W''_{(i,j_1)}$ is the sum of the twistings on $W_{(i,j_1)}$, $W'_{(i,j_1)}$, and $V$; so $W''_{(i,j_1)}$ is framed if $V$ is framed, since both $W_{(i,j_1)}$ and $W'_{(i,j_1)}$ are framed for $n>1$ (and in the $n=1$ case $W_{(i,j_1)}=W_{J_n}$ and $W'_{(i,j_1)}=W'_{J_n}$ contribute canceling $\pm 1$ twistings). If $V$ is both framed and clean, then the result of
replacing $W_{(i,j_1)}$ and $W'_{(i,j_1)}$ by $V$ and $W''_{(i,j_1)}$ preserves the order of $\cW$ and creates no new intersections.

So if $n=1$, then $W_{J_n}$ and $W'_{J_n}$ have been geometrically canceled, meaning that their corresponding $\iinfty$-trees have been eliminated from $t(\cW)$ without creating any new unpaired order $2n$ intersections or new twisted order $n$ Whitney disks.

The next step shows how $V$ can be arranged to be framed and clean, at the cost of only creating intersections of order greater than $2n$: Any twisting $\omega(V)$ can be killed by boundary twisting $V$ into $D_{j_1}$.  Then, using the construction shown in Figure~\ref{bandedWdisksWithPushDown}, any interior intersection between $V$
and any $K$-sheet (e.g.~an intersection with $D_{j_1}$ from boundary-twisting) can be pushed down into $D_i$ and paired by a thin Whitney disk $W_{(K,i)}$, which in turn has intersections with the $D_{j_1}$-sheet that can be paired by a Whitney disk
$W_{K_1}:=W_{((K,i),j_1)}$ made from a Whitney-parallel copy of $W_{(i,j_1)}$. 
Now, parallel copies of the Whitney disks from the sub-tower supported by $W_{(i,j_1)}$ can be used to build a sub-tower on $W_{K_1}$:
Using the notation $K_{r+1}=(K_r,i)$,
for $r=1,2,3,\ldots n$, the Whitney disk $W_{K_{r+1}}$ is built from a Whitney-parallel copy of $W_{J_r}$, and pairs intersections between $W_{K_r}$ and $j_r$. Note that the order of each $W_{K_r}$ is at least $r$.
The top order $W_{K_{n+1}}$ inherits the $\pm 1$-twisting from $W_{J_n}$, and has a single interior intersection with tree $\langle K_{n+1}, J_n \rangle$ which is of order at least $2n+1$. 

\begin{figure}[h]
\centerline{\includegraphics[width=100mm]{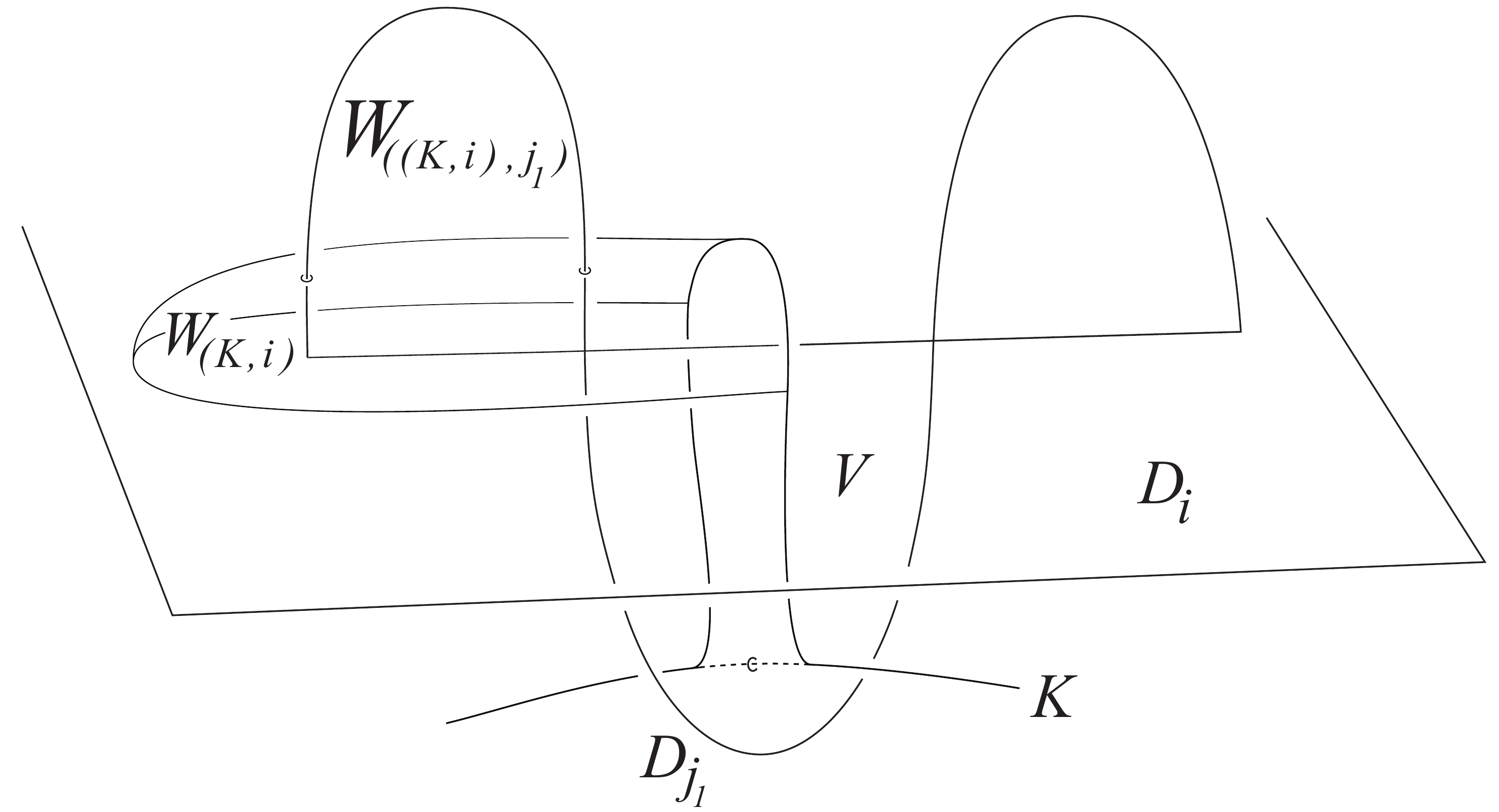}}
         \caption{}
         \label{bandedWdisksWithPushDown}
\end{figure}

For multiple intersections between $V$ and various $K$-sheets this part of the construction can be carried out simultaneously using nested parallel copies of the thin Whitney disk in 
Figure~\ref{bandedWdisksWithPushDown} and more Whitney-parallel copies of the sub-towers described in the previous paragraph.

The result of the construction so far is that the left-most trivalent vertices of the trees $+ J_n^\iinfty$ and $-J_n^\iinfty$ now correspond to the \emph{same} order $1$ Whitney disk $W''_{J_1}=W''_{(i,j_1)}$, at the cost of having created (after splitting-off) a clean twisted Whitney disk of order at least $n+1$, and an unpaired intersection of order at least $2n+1$. In particular, this completes the proof for the case $n=1$.


For the cases $n>1$, observe that since $W''_{J_1}$ is \emph{connected}, this construction can be repeated, with $W''_{J_1}$ playing the role of $D_i$, and $D_{j_2}$ playing the role of $D_{j_1}$, to get a single order $2$ Whitney disk $W''_{J_2}$ which corresponds to the second trivalent vertices
from the left in both trees $+ J_n^\iinfty$ and $-J_n^\iinfty$. By iterating the construction, eventually we band together $W_{J_n}$ and $W'_{J_n}$ into a single framed clean Whitney disk at the last step, having created only clean twisted Whitney disks of order at least $n+1$, and unpaired intersections of order at least $2n+1$.

\subsection{The geometric twisted IHX relation}\label{subsec:twistedIHX}

The proof of Theorem~\ref{thm:twisted-order-raising-on-A} is completed by the following lemma which describes a
twisted IHX construction on the intersection tree $t(\cW)$ of a twisted Whitney tower $\cW$. This geometric move is
based on the framed version given in Lemma~{7.2}
of \cite{S1}.

\begin{lem}\label{lem:twistedIHX}
Any split twisted Whitney tower $\cW$ containing a clean $+1$-twisted Whitney disk with signed $\iinfty$-tree
$+I^\iinfty$ can be modified (in a neighborhood of the Whitney disks and local order zero sheets corresponding to $I^\iinfty$)
to a twisted Whitney tower $\cW'$ such that $t(\cW')$ differs from $t(\cW)$ exactly by replacing
$+I^\iinfty$ with the signed trees $+H^\iinfty$, $+X^\iinfty$, and $-\langle H,X \rangle$, where $+I-H+X$ is a Jacobi relator.

Similarly, a clean $-1$-twisted Whitney disk with $\iinfty$-tree $-I^\iinfty$ in $t(\cW)$
can be replaced by $-H^\iinfty$, $-X^\iinfty$, and $+\langle H,X \rangle$ in $t(\cW')$.
\end{lem}

\begin{proof}
Before describing how to adapt the construction and notation of
\cite{S1} to give a detailed proof of Lemma~\ref{lem:twistedIHX}, we explain why the framed geometric relation
$+I=+H-X$ leads to the twisted relation $+I^\iinfty=+H^\iinfty+X^\iinfty-\langle H,X \rangle$. In the framed case, a Whitney disk
with tree $I$ is replaced by Whitney disks with trees $H$ and $X$, such that the new Whitney disks are parallel
copies of the original using the Whitney framing, and inherit the framing of the original. In order to preserve the trivalent
vertex orientations of the trees, the orientation of the H-Whitney disk is the same as the original 
I-Whitney disk, and the
orientation of the X-Whitney disk is the opposite of the I-Whitney disk. Now, if the original I-Whitney disk was 
$+1$-twisted,
then both the H- and X-Whitney disks will inherit this same $+1$-twisting, because the twisting -- which is a self-intersection number -- is independent of the Whitney disk orientation. 
The H- and X-Whitney disks will also intersect in a single point with sign $-1$,
since they inherited opposite orientations from the I-Whitney disk. Thus, (after splitting) a twisted Whitney tower can be modified so that a $+I^\iinfty$ is replaced by exactly $+H^\iinfty+X^\iinfty-\langle H,X \rangle$ in the intersection forest. Similarly, a $-I^\iinfty$ can be replaced exactly by $-H^\iinfty-X^\iinfty+\langle H,X \rangle$.

The framed IHX Whitney move construction is described in detail in \cite{S1} (over four pages, including six figures).
We describe here how to adapt that construction to the present twisted case, including the relevant modification of notation.
Orientation details are not given in \cite{S1}, but all that needs to be checked is that the X-Whitney disk inherits the opposite
orientation as the H-Whitney disk (given that the tree orientations are preserved, and using our negative-corner orientation convention in \ref{subsec:w-tower-orientations} above). In Lemma~{7.2}
of \cite{S1}, the ``split sub-tower $\cW_p$''  refers to the Whitney disks and order zero sheets containing the tree $t_p$ of an unpaired intersection $p$ in a split Whitney tower $\cW$. In the current setting, a clean $+1$-twisted Whitney disk $W$ plays the role of $p$, and the construction will modify $\cW$ in a neighborhood of the Whitney disks and order zero sheets containing the $\iinfty$-tree associated to $W$.
In the notation of Figure~18 of \cite{S1}, the sub-tree of the I-tree denoted by $L$ contains $p$, so to interpret the entire construction in our case only requires the understanding that this sub-tree contains the $\iinfty$-label sitting in $W$. (Note that in Figure~18 of \cite{S1} the labels $I$, $J$, $K$ and $L$ denote \emph{sub}-trees, and in particular the $I$-labeled sub-tree should not be confused with the ``I-tree'' in the IHX relation.)

In the case where the $L$-labeled sub-tree is order zero, then $L$ is just the $\iinfty$-label, and the upper trivalent vertex of the I-tree in Figure~18 of \cite{S1} corresponds to the clean $+1$-twisted $W$, with $\iinfty$-tree $((I,J),K)^\iinfty$. Then the construction, which starts by performing a Whitney move on the framed Whitney disk
$W_{(I,J)}$ corresponding to the lower trivalent vertex of the I-tree, yields the $+1$-twisted
H- and X-Whitney disks as discussed in the first paragraph of this proof, with $\iinfty$-trees $(I,(J,K))^\iinfty$ and $(J,(I,K))^\iinfty$, and non-$\iinfty$ tree $\langle (I,(J,K)),(J,(I,K))\rangle$ corresponding to the resulting unpaired intersection (created by taking Whitney-parallel copies of the twisted $W$ to form the H- and X-Whitney disks).

In the case where the $L$-labeled sub-tree is order $1$ or greater, then the upper trivalent vertex of the
I-tree in Figure~18 of \cite{S1} corresponds to a framed Whitney disk, and Whitney-parallel copies of this framed Whitney disk and the other Whitney disks corresponding to
the $L$-labeled sub-tree are also used to construct the sub-towers containing the $+1$-twisted Whitney disks with H and X $\iinfty$-trees (which will again will lead to a single unpaired intersection as before).
\end{proof}

\subsection{Twisted even order and framed odd order Whitney towers}\label{subsec:boundary-twisted-IHX-lemma}
The following lemma implies that 
$\bW^\iinfty_{2n}\subset \bW_{2n-1}$, a fact that will be used later in the proof of Theorem~\ref{thm:exact-sequence}:
\begin{lem}\label{lem:boundary-twisted-IHX}
If a collection $A$ of properly immersed surfaces in a simply connected $4$--manifold supports an order $2n$ \emph{twisted} Whitney tower, then $A$ is homotopic (rel $\partial$) to 
$A'$ which supports an order $2n-1$ \emph{framed} Whitney tower.
\end{lem}
\begin{proof}
Let $\cW$ be any order $2n$ twisted Whitney tower $\cW$ supported by $A$. If $\cW$ contains no order $n$ non-trivially twisted Whitney disks, 
then $\cW$ is an order $2n$ framed Whitney tower, hence also is an order $2n-1$ framed Whitney tower. If $\cW$ does contain order $n$ non-trivially twisted Whitney disks, they can be eliminated at the cost of only creating intersections of order at least $2n-1$ as follows:

Consider an order $n$ twisted Whitney disk $W_J\subset\cW$ with twisting $\omega(W_J)=k\in\Z$. If $W_J$ pairs intersections between an order zero surface $A_i$ and an order $n-1$ Whitney disk $W_I$ then $J=(i,I)$, and by performing $|k|$ boundary-twists of $W_J$ into $W_I$, $W_J$ can be made to be framed at the cost of only creating $|k|$ order $2n-1$ intersections, whose corresponding trees are of the form $\langle\,(i,I),I\,\rangle$.

If $W_J$ pairs intersections between two Whitney disks, then by applying the twisted geometric IHX move of
Lemma~\ref{lem:twistedIHX} (as many times as needed), $W_J$ can be replaced by (many) order $n$ twisted Whitney disks
each having a boundary arc on an order zero surface as in the previous case, at the cost of only creating unpaired intersections of order $2n$, each of which is an error term in Lemma~\ref{lem:twistedIHX}.   
\end{proof}

\subsection{Obstruction theory for the reduced tree groups}\label{sec:reduced-groups-obstruction-theory}
Using the twisted IHX Lemma~\ref{lem:twistedIHX}, this section strengthens the obstruction theory for framed Whitney towers described in \cite{ST2} by showing that the vanishing
of $\tau_{2n-1}(\cW)$ in the reduced group $\widetilde{\cT}_{2n-1}:=\cT_{2n-1}/\im(\Delta_{2n-1})$ is sufficient for the promotion of $\cW$ to a Whitney tower of order 
$2n$. This means that $\cT_n$ can be replaced everywhere by 
$\widetilde{\cT}_n$ (with $\widetilde{\cT}_{2n}:=\cT_{2n}$) throughout Section~\ref{sec:realization-maps}, showing that the
realization maps $\widetilde{R}_n:\widetilde{\cT}_n\to\W_n$ are well-defined epimorphisms.


Recall from the introduction the \emph{framing relations} which in odd orders give the reduced group $\widetilde\cT_{2n-1}:={\cT}_{2n-1}/\im(\Delta_{2n-1})$ with $\Delta_{2n-1}$ given as follows.
\begin{defn}\label{def:Delta}
The map $\Delta_{2n-1}:\cT_{n-1}\rightarrow \cT_{2n-1}$ is defined for generators $t\in\cT_{n-1}$ by
$$\Delta (t):=\sum_{v\in t} \langle \ell(v),(T_v(t),T_v(t))\rangle$$
where the sum is over all univalent vertices $v$ of $t$, with $T_v(t)$ denoting the rooted tree gotten by replacing $v$ with a root,
and $\ell(v)$ the original label of $v$. 
\end{defn}
That $\Delta_{2n-1}$ is a well-defined homomorphism is clear since AS and IHX relations go to ``doubled'' relations. 
See Figure~\ref{fig:Delta trees} for explicit illustrations of $\Delta_1$ and $\Delta_3$.
\begin{figure}[h]
\centerline{\includegraphics[width=90mm]{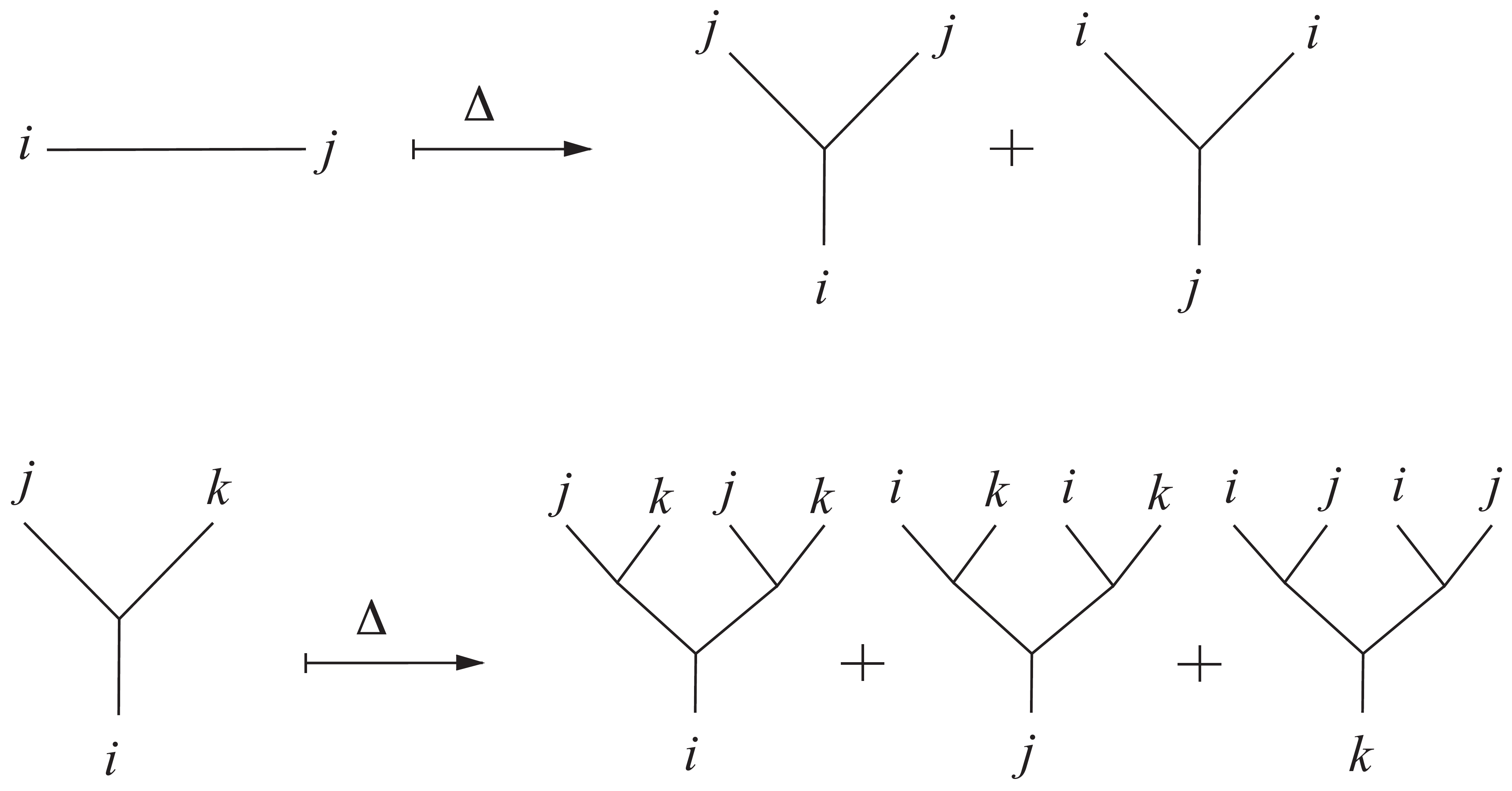}}
         \caption{The map $\Delta_{2n-1}:\cT_{n-1}\rightarrow \cT_{2n-1}$ in the cases $n=1$ and $n=2$.}
         \label{fig:Delta trees}
\end{figure}

The following theorem strengthens Theorem~\ref{thm:framed-order-raising-on-A} in Section~\ref{sec:w-towers}.
\begin{thm}\label{thm:framed-order-raising-mod-Delta}
If a collection $A$ of properly immersed surfaces in a simply connected $4$--manifold supports a framed
 Whitney tower $\cW$ of order $(2n-1)$ with $\tau_{2n-1}(\cW)\in\im(\Delta_{2n-1})$, then $A$ is regularly homotopic (rel $\partial$) to  $A'$ which supports a framed Whitney tower of order $2n$.
\end{thm}

\begin{proof}
As discussed above in the outline the proof of Theorem~\ref{thm:twisted-order-raising-on-A} (section~\ref{subsec:twisted-order-raising-thm-proof}), to prove Theorem~\ref{thm:framed-order-raising-mod-Delta} it will suffice to show that the intersection forest $t(\cW)$ can be changed by trees representing any element in
$\im(\Delta_{2n-1})<\cT_{2n-1}$ at the cost of only introducing trees of order greater than or equal to $2n$, so that the order $2n-1$ trees in $t(\cW)$ all occur in algebraically canceling pairs. Note that $\im(\Delta_{2n-1})$ is $2$-torsion by the AS relations, so orientations and signs are not an issue here. As in Section~\ref{sec:proof-twisted-thm}, elements of $t(\cW)$ will be denoted by formal sums, and $\cW$ will not be renamed as modifications are made.

\textbf{The case $n=1$:} Given any order zero tree $\langle i,j \rangle$, create a clean framed Whitney disk 
$W_{( i,j )}$ by performing a finger move between the order zero surfaces $A_i$ and $A_j$. Then use a twisted finger move 
(Figure~\ref{twist-split-Wdisk-fig}) to split $W_{( i,j )}$ into two twisted Whitney disks with associated trees $(i,j)^\iinfty -(i,j)^\iinfty$. 
Now boundary-twist each Whitney disk into a different sheet to recover the framing and add 
\[
\langle i,(i,j) \rangle +  \langle j,(i,j) \rangle = \Delta_1( \langle i,j \rangle )
\]
 to $t(\cW)$. Alternatively, after creating the framed $W_{( i,j )}$, perform an interior twist on $W_{( i,j )}$ to get 
$\omega(W_{( i,j )})=\pm2$, then kill $\omega(W_{( i,j )})$ by two boundary-twists, one into each sheet, again adding $\langle i,(i,j) \rangle +  \langle j,(i,j) \rangle$ to $t(\cW)$.
Note that $\im\Delta_1$ in $\cT_1$ corresponds to the order $1$ FR framing relation of \cite{S3,ST1}.

\begin{figure}[h]
\centerline{\includegraphics[scale=.35]{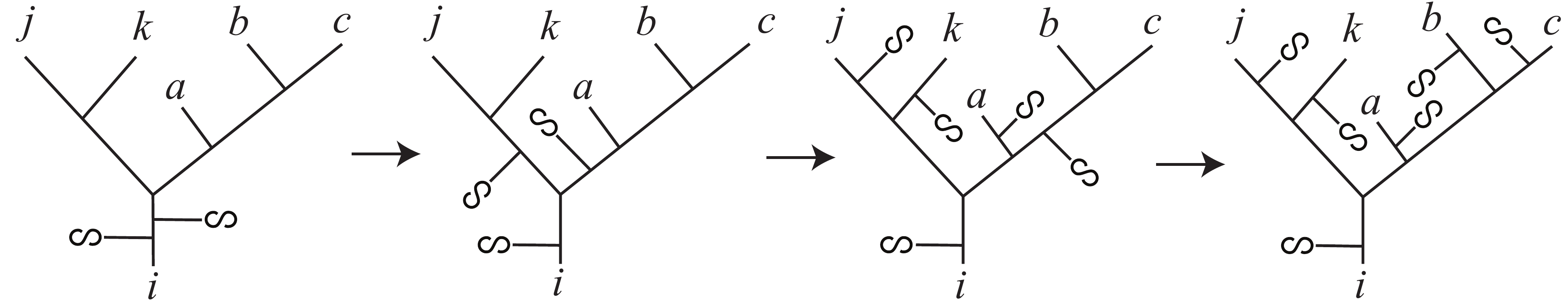}}
         \caption{Multiple $\infty$-roots attached to a tree represent sums (disjoint unions) of trees. On the left: the two trees that result from twist-splitting a clean $W_{( i, (I_1, I_2))}$ in the case $\langle i,(I_1, I_2) \rangle=\langle i, ((j,k),(a,(b,c))) \rangle$. Each arrow indicates an application of a twisted IHX Whitney move, which pushes $\infty$-roots towards the univalent vertices. The right-most sum of trees becomes the image of $\langle i,(I_1, I_2) \rangle$ under $\Delta$ after applying boundary-twists to the associated twisted Whitney disks.}
         \label{fig:Delta-infty-tree-example}
\end{figure}
\textbf{The cases $n>1$:} For any order $n-1$ tree $\langle i,(I_1, I_2) \rangle$, create a clean $W_{( i, (I_1, I_2))}$ by finger moves. (Here we are taking any order $n-1$ tree, choosing an $i$-labeled univalent vertex, and writing it as the inner product 
of the order zero rooted tree $i$ and the remaining order $n-1$ tree.)
Then split $W_{( i, (I_1, I_2))}$ using a twisted finger move to get two twisted Whitney disks each having associated $\iinfty$-tree
$( i, (I_1, I_2))^\iinfty$. 
Leave one of these twisted Whitney disks alone, and to the other apply the twisted geometric IHX Whitney move (Lemma~\ref{lem:twistedIHX} of Section~\ref{sec:proof-twisted-thm}) to replace 
$( i, (I_1, I_2))^\iinfty$ by $( I_1, (I_2, i))^\iinfty+( I_2, (i,I_1))^\iinfty-\langle ( I_1, (I_2, i)),( I_2, (i,I_1))\rangle$ in $t(\cW)$. Note that the 
tree $\langle ( I_1, (I_2, i)),( I_2, (i,I_1))\rangle$ is order $2n$. If $I_1$ and $I_2$ are not both order zero then continue to 
apply the twisted geometric IHX Whitney move (pushing the $\iinfty$-labeled vertices away from the $\iinfty$-labeled vertex that is adjacent to the original $i$-labeled vertex) until the resulting union of trees has all $\iinfty$-labeled vertices adjacent to a univalent vertex (all twisted Whitney disks have a boundary arc on an order zero surface) -- see Figure~\ref{fig:Delta-infty-tree-example} for an example.  Then, boundary-twisting each twisted Whitney disk into the order zero surface recovers the framing on each Whitney disk and the resulting change in $t(\cW)$ is a sum of trees as in the right hand side of the equation in Definition~\ref{def:Delta} representing the image of 
$\langle i,(I_1, I_2) \rangle$ under $\Delta_{2n-1}$, together with trees of order at least $2n$.
\end{proof}

\section{From the twisted to the framed classification}\label{sec:twisted-to-framed-classification}

This section fills in the outline of section~\ref{subsec:intro-untwisting-the-twisted-filtration} by explaining how to derive the classification of the framed Whitney tower filtration from the classification of
the twisted filtration described in section~\ref{subsec:intro-twisted-filtration} of the introduction.
The main tool is Theorem~\ref{thm:exact-sequence}, which relates the
two filtrations and their relevant tree groups in a diagram of exact sequences. The subsequent sections~\ref{subsec:higher-order-SL} and \ref{subsec:higher-order-Arf-in-framed-filtration} show how this result, together with the Levine Conjecture 
and the Milnor invariant-intersection invariant relationship, leads to the framed classification in terms of Milnor invariants, higher-order Sato-Levine invariants and higher-order Arf invariants. 
The proof of Theorem~\ref{thm:exact-sequence} is completed in section~\ref{subsec:proof-thm-exact-sequences}. 

The starting point is the following surprisingly simple relation between the twisted and framed Whitney tower filtrations. Recall that in even orders the reduced groups $\widetilde{\cT}_{2k}$ and realization maps $\widetilde{R}_{2k}$ are by definition equal to $\cT_{2k}$ and $R_{2k}$.

\begin{thm}\label{thm:exact-sequence}
There are commutative diagrams of exact sequences
\[
\xymatrix{ 
0\ar[r] & \widetilde{\cT}_{2k} \ar@{->>}[d]^{\widetilde{R}_{2k}}\ar[r] &  \cT^\iinfty_{2k}\ar@{->>}[d]^{R^\iinfty_{2k}} \ar[r] & \widetilde\cT_{2k-1}\ar@{->>}[d]^{\widetilde{R}_{2k-1}}\ar[r] &  \cT^\iinfty_{2k-1}  \ar@{->>}[d]^{R^\iinfty_{2k-1}}\ar[r] & 0 \\
0\ar[r] & \W_{2k}  \ar[r] &  \W^\iinfty_{2k} \ar[r] & \W_{2k-1}\ar[r] &  \W^\iinfty_{2k-1} \ar[r] & 0
 } \]
where all maps in the bottom row are induced by the identity on the set of links.
Moreover, there are isomorphisms
\[
\Cok ( \cT_{2k} \to  \cT^\iinfty_{2k}) \cong \Z_2 \otimes \sL'_{k+1} \cong \Ker(  \widetilde\cT_{2k-1}  \to  \cT^\iinfty_{2k-1} )
\]
\end {thm}
Here $\sL'=\oplus\sL'_n$ is Levine's \emph{quasi-Lie algebra}, which we define using the usual identification of brackets with rooted trees:
\begin{defn}[\cite{L2}]\label{def:L'} 
The \emph{degree $n$} abelian group $\sL'_n=\sL'_n(m)$ is generated by rooted trees of order~$(n-1)$, each having an unlabeled root univalent vertex, with all other univalent vertices labeled by elements of $\{1,\dots,m\}$, modulo the AS and IHX relations of Figure~\ref{fig:ASandIHXtree-relations}.
\end{defn}
The prefix `quasi' reflects the fact that, although the IHX relation corresponds
to the Jacobi identity, the usual Lie algebra self-annihilation relation $[X,X]=0$ does 
{\em not} hold in $\sL'$. 
It is replaced by the weaker antisymmetry (AS) relation $[Y,X]=-[X,Y]$.
Levine shows that in odd \emph{degrees} (even orders) the natural projection $\sL_{2k-1}'\to\sL_{2k-1}$ is an isomorphism, 
while in even \emph{degrees} (odd orders) we have 
the split exact sequence
\begin{equation} \label{Z2-L'-L ses}
\tag{\sf L} 0\to\mathbb Z_2\otimes {\sf L}_{k}\to\sL'_{2k}\to \sL_{2k}\to 0
\end{equation}
where the left map sends $X$ to $[X,X]$, and the right map is the natural projection \cite{L3}.

To see exactness of the bottom row of Theorem~\ref{thm:exact-sequence} one first observes that there is a natural inclusion $\bW_{n} \subseteq \bW_{n}^\iinfty$, and that $\bW_{2k-1}^\iinfty = \bW_{2k-1}$ by definition.
Then, from the inclusion 
$\bW_{2k}^\iinfty \subseteq \bW_{2k-1}$ shown above in Lemma~\ref{lem:boundary-twisted-IHX}, exactness follows since $\W_n := \bW_n/ \bW_{n+1}$ and $\W^\iinfty_n := \bW^\iinfty_n/  \bW^\iinfty_{n+1}$. 
Proof of Theorem~\ref{thm:exact-sequence} is completed in Section~\ref{subsec:proof-thm-exact-sequences}, using the \emph{universality} of $ \cT^\iinfty_{2k}$ as the target of quadratic refinements of the canonical `inner product' pairing 
\[
\langle \ , \ \rangle: \sL'_{k+1} \times \sL'_{k+1} \to \cT_{2k}
\]
given by gluing the roots of two rooted trees, see Definition~\ref{def:Trees} and \cite{UQF}.

In order to proceed with the analysis of the framed filtration, notice that
the group 
\[
\sK^\mu_{2k-1}:=\Ker(\W_{2k-1} \sra \W^\iinfty_{2k-1}\cong\sD_{2k-1})
\]
is precisely the kernel of the Milnor invariant $\mu_{2k-1}\colon \W_{2k-1}\to \sD_{2k-1}$, where the \emph{order $n$ Milnor invariant} 
$\mu_n:\W_n \to\sD_n$ is defined in the framed setting via the composition $\W_n\to\W^\iinfty_n\to\sD_n$ induced by the inclusion $\bW_n \subset \bW^\iinfty_n$. This breaks the bottom exact sequence in Theorem~\ref{thm:exact-sequence} into two short exact sequences, and calculating $\sK^\mu_{2k-1}$ will thus allow us to compute $\W_{2k}$ and $\W_{2k-1}$ in terms of $\W^\iinfty_{2k}$ and $\W^\iinfty_{2k-1}$.

On the other hand, the group $\sK^\mu_{2k-1}\cong\bW^\iinfty_{2k}/\bW_{2k}$ precisely measures the obstructions to framing a twisted Whitney tower of order $2k$. We will show next how these obstructions are detected by the higher-order Sato-Levine invariants defined as projections of the order $2k$ Milnor invariants, together with a direct analogue of the higher-order Arf invariants defined for the framed Whitney tower filtration. 

In direct analogy with $\sD_n$, the group $\sD'_n$ is defined as the kernel of the quasi-Lie bracketing map $\sL'_1\otimes\sL'_{n+1}\to \sL'_{n+2}$. Analogously to $\eta_n$, there is a map $\eta'_n\colon \cT_n\to \sD_n'$ defined again by summing over choosing a root at each univalent vertex of each generator: 
\begin{defn}[\cite{L2}]\label{def:eta-prime}
On trees $t\in\cT_n$ set $\eta'_n(t):=\sum_{v\in t} X_{\ell(v)}\otimes T'_v(t)$, 
where the sum is over all univalent vertices $v$ of $t$, with $T'_v(t)\in\sL'_{n+1}$ denoting the rooted tree gotten by replacing $v$ with a root,
and $\ell(v)$ the original label of $v$.
\end{defn}

The previously-mentioned Levine Conjecture \cite{L2} is the statement that $\eta'_n$ is an isomorphism, a surprisingly difficult fact which we prove in \cite{CST3}: 
\begin{thm}[\cite{CST3}]\label{thm:LC} 
The maps $\eta'_n\colon \cT_n\to \sD_n'$ are isomorphisms.
\end{thm} 

The surjectivity of the realization maps (Theorem~\ref{thm:R-onto-W}) together with the factorization $\eta'_n=\mu_n\circ R_n$ from \cite{CST2,ST2}
then gives us the following commutative diagram:
\[
\xymatrix{
\cT_{2k} \ar@{->>}[r]^{R_{2k}} \ar@{>->>}[rd]_{\eta'_{2k}} & \W_{2k} \ar@{->>}[d]^{\mu_{2k}}\\
& \sD'_{{2k}}
}\]
The groups $\sD'_{2k}$ are finite index subgroups of $\sD_{2k}$ and hence are free abelian of known rank for all $k$ \cite[Cor.2.3]{L3}, so by Theorem~\ref{thm:LC} this diagram gives the classification of the Whitney tower filtration in all \emph{even} orders, implying Theorem~\ref{thm:intro-even-order-R}:
\begin{thm}\label{thm:even-R-and-mu}
The maps $R_{2k}$ and $\mu_{2k}:\W_{2k}\to \sD'_{2k}$ are isomorphisms.
\end{thm}

\subsection{Higher-order Sato-Levine invariants}\label{subsec:higher-order-SL}
In order, to define the \emph{order $2k-1$ Sato-Levine invariant} $\SL_{2k-1}\colon \sK^\mu_{2k-1}\sra \z\otimes \sL_{k+1}$,  we first use a result of Levine to define homomorphisms $s\ell_{2k}$ algebraically:
\begin{defn}\label{def:sl-maps} The epimorphisms $s\ell_{2k}\colon {\sf D}_{2k}\sra \z\otimes{\sL}_{k+1}$ are defined by the snake lemma applied to the diagram:
$$\xymatrix{
&&\mathbb Z_2\otimes\sL_{k+1}\ar@{>->}[d]^{sq}\\
\sD'_{2k}\ar@{>->}[d] \ar@{>->}[r]&\sL_1\otimes\sL'_{2k+1}\ar@{>->>}[d]^\cong\ar[r]&\sL'_{2k+2}\ar@{->>}[d]\\
\sD_{2k} \ar@{>->}[r]\ar@{-->>}[d]^{s\ell_{2k}}&\sL_1\otimes \sL_{2k+1}\ar@{->>}[r]&\sL_{2k+2}\\
\z\otimes \sL_{k+1}&&
}$$
The two horizontal sequences are exact by definition and the vertical sequence on the right is exact by Theorem 2.2 of \cite{L3}. The squaring map on the upper right is $sq(1\otimes X) := [X,X]$.
\end{defn}

Now suppose $L\in\bW_{2k-1}$ represents an element in 
\[
\sK^\mu_{2k-1}:=\Ker(\mu_{2k-1}\colon \W_{2k-1}\to\sD_{2k-1}).
\]
 Since $\mu_{2k-1}(L)=0$, we have that $\mu_{2k}(L)\in \sD_{2k}$ is defined. This gives the definition of the higher-order Sato-Levine invariants:

\begin{defn}\label{def:SL-invariants}
The Sato-Levine invariant $\SL_{2k-1}(L)$ is equal to $s\ell_{2k}\circ \mu_{2k}(L)$.
\end{defn}

Summarizing the previous discussion, the cokernel of the inclusion $\sD_{2k}'\to\sD_{2k}$ is isomorphic to $\z\otimes\sL_{k+1}$, and we get the following commutative diagram:
$$\xymatrix{
{\sf W}_{2k}\ar@{>->}[r]\ar[d]^{\mu_{2k}}_\cong& {\sf W}^\iinfty_{2k}\ar@{->>}[r]\ar@{->>}[d]^{\mu_{2k}}& {\sf K}^\mu_{2k-1}\ar@{-->>}[d]^{\SL_{2k-1}}\\
\sD'_{2k}\ar@{>->}[r]&\sD_{2k}\ar@{->>}[r]&\z\otimes \sL_{k+1}
}$$

We know from Theorem~\ref{thm:intro-twisted-3-quarter-classification} that $\mu_{2k}:{\sf W}^\iinfty_{2k}\to\sD_{2k}$ is an isomorphism when $k$ is even, implying the following proposition:
\begin{prop}\label{prop:SL4k-1-iso}
$\SL_{4k-1}$ gives an isomorphism  $\sK^\mu_{4k-1}\cong \z\otimes \sL_{2k+1}$.
\end{prop}
Then, using the fact that $\sL_{2k+1}=\sL'_{2k+1}$ we get the following commutative diagram from Theorem~\ref{thm:exact-sequence} and Theorem~\ref{thm:intro-twisted-3-quarter-classification}:
$$
\xymatrix{
\z\otimes \sL_{2k+1}\ar@{-->>}[d]\ar@{>->}[r]&\widetilde{\cT}_{4k-1}\ar@{->>}^{\widetilde{R}_{4k-1}}[d]\ar@{->>}[r]&\cT^\iinfty_{4k-1}\ar[d]^\cong\\
\sK^\mu_{4k-1}\ar@{>->}[r]&\W_{4k-1}\ar@{->>}[r]&\W^\iinfty_{4k-1}
}
$$
Because $\sK^\mu_{4k-1}\cong \z\otimes \sL_{2k+1}$, the left-hand epimorphism of finite dimensional $\z$-vector spaces must be an isomorphism. This implies that $\widetilde{R}_{4k-1}$ is an isomorphism, and in combination with Theorem~\ref{thm:even-R-and-mu} gives the classification of the framed filtration in three quarters of the cases:
\begin{thm}\label{thm:intro-framed-3-quarter-classification}
If 
$n\not\equiv 1\mod 4$, then $\widetilde{R}_n:\widetilde{\cT}_n\to \W_n$ is an isomorphism.
\end{thm}

In particular, this result (together with Proposition~\ref{prop:SL4k-1-iso}) proves Theorem~\ref{thm:intro-odd-R-isos} from the introduction.

\subsection{Higher-order Arf invariants in the framed filtration}\label{subsec:higher-order-Arf-in-framed-filtration}
Finally we consider the remaining cases $\W_{4k-3}$ of the framed filtration, which is where the higher-order Arf invariants reappear. Let $\sK^{\SL}_{4k-3}$ be the kernel of the order $4k-3$ Sato-Levine invariant $\SL_{4k-3}\colon \sK^\mu_{4k-3}\to \z\otimes \sL_{2k}$. 
Recall from Definition~\ref{def:Arf-k} that in the twisted setting the higher-order Arf invariants $\Arf_{k}:\sK^\iinfty_{4k-2}\to(\mathbb Z_2\otimes {\sf L}_{k})/\Ker \alpha^\iinfty_{k}$ are defined by inverting a surjection $\alpha^\iinfty_k: \Z_2 \otimes \sL_k
\sra \sK^\iinfty_{4k-2}$ onto the 
kernel $\sK^\iinfty_{4k-2}$ of the order $4k-2$ invariant $\mu_{4k-2}:\sW^\iinfty_{4k-2}\to\sD_{4k-2}$.
\begin{lem}\label{lem:canonical-arf-iso}
$\sK^{\SL}_{4k-3}$ is canonically isomorphic to $\sK^\iinfty_{4k-2}$.
\end{lem}
\begin{proof}
This follows from the commutative diagram: 
$$\xymatrix{
&{\sf K}^\iinfty_{4k-2}\ar@{>-->>}[r]\ar@{>->}[d]&{\sf K}^{\SL}_{4k-3}\ar@{>->}[d]\\
{\sf W}_{4k-2}\ar@{>->}[r]\ar[d]_\cong& {\sf W}^\iinfty_{4k-2}\ar@{->>}[r]\ar@{->>}[d]& {\sf K}^\mu_{4k-3}\ar@{->>}[d]\\
\sD'_{4k-2}\ar@{>->}[r]&\sD_{4k-2}\ar@{->>}[r]&\z\otimes \sL_{2k}
}$$
(The left-hand isomorphism comes from Theorem~\ref{thm:even-R-and-mu}.) 
\end{proof}

Thus the higher-order Arf invariants induce maps $\Arf_k\colon \sK^{\SL}_{4k-3}\to \z\otimes\sL_k/\Ker(\alpha_k)$, and we get a complete classification for the framed filtration:

\begin{cor}\label{cor:mu-sl-arf-classify}
The groups $\W_n$ are classified by Milnor invariants $\mu_n$ and in addition, Sato-Levine invariants $\SL_n$ if $n$ is odd, and finally, Arf invariants $\Arf_k$ for $n=4k-3$. 
\end{cor}
In particular, a link bounds an order $n$ Whitney tower if and only it has all vanishing Milnor, Sato-Levine and Arf invariants up to order $n$ (Compare Theorem~\ref{thm:intro-classification}).

\subsection{The Arf invariant conjecture in terms of reduced realization maps}\label{subsec:higher-order-arf-reduced-realization-conj}
To complete the translation from the twisted setting, as promised in section~\ref{subsec:intro-untwisting-the-twisted-filtration}, we observe that
Conjecture~\ref{conj:Arf-k} is equivalent to the statement that $\widetilde{R}_{4k-3}$ is an isomorphism
by considering the following commutative diagram:
$$
\xymatrix{
\z\otimes \sL'_{2k}\ar@{-->>}[d]\ar@{>->}[r]&\widetilde{\cT}_{4k-3}\ar@{->>}^{\widetilde{R}_{4k-3}}[d]\ar@{->>}[r]&\cT^\iinfty_{4k-3}\ar[d]^\cong\\
\sK^\mu_{4k-3}\ar@{>->}[r]&\W_{4k-3}\ar@{->>}[r]&\W^\iinfty_{4k-3}
}
$$
Conjecture~\ref{conj:Arf-k} is equivalent to $\sK^{\SL}_{4k-2}\cong \z\otimes\sL_k$, which in turn is equivalent to $\sK^\mu_{4k-3}\cong \sL'_{2k}\otimes \z$ (not necessarily canonically), by Levine's exact sequence~\eqref{Z2-L'-L ses} (see just after Definition~\ref{def:L'} above). 

This happens if and only if the left-hand vertical epimorphism is an isomorphism, which happens if and only if $\widetilde{R}_{4k-3}$ is an isomorphism.

Thus, Theorem~\ref{thm:exact-sequence} well illustrates both the relationships between the various $\cT$- and $\W$-groups, and the implications of the higher-order Arf invariant Conjecture~\ref{conj:Arf-k}. In Section~\ref{sec:master-diagrams-and-algebra}
this commutative diagram is extended by the relevant $\eta$- and $\mu$-maps to include exact sequences of $\sD$-groups, giving our {\em Master Diagrams}.

\subsection{Exactness of the tree sequence}\label{subsec:proof-thm-exact-sequences}
Here we complete the proof of Theorem~\ref{thm:exact-sequence} by showing the exactness of the top sequence of $\cT$-groups.
The exactness of the bottom sequence of $\sW$-groups was checked just after the statement, and the realization maps are well-defined surjections by Section~\ref{sec:realization-maps}.
Commutativity of the diagram follows from the fact that elements of the $\sW$-groups are determined by the 
value of the corresponding $\tau$-invariant (Corollary~\ref{cor:tau=w-concordance}).

So the proof of Theorem~\ref{thm:exact-sequence} will be completed by showing that for any $k\in\N$, there are short exact sequences
\[
\xymatrix{ 
0\ar[r] & \cT_{2k} \ar[r] &  \cT^\iinfty_{2k} \ar[r] & \Z_2 \otimes \sL'_{k+1} \ar[r] & 0
 } \]
and
\[
\xymatrix{ 
0\ar[r] &  \Z_2 \otimes \sL'_{k+1} \ar[r] & \widetilde\cT_{2k-1}  \ar[r] &  \cT^\iinfty_{2k-1} \ar[r] & 0.
 } \]

The injectivity of the map $\cT_{2k}\hookrightarrow \cT^\iinfty_{2k}$ is proven in \cite{UQF} (see also Remark~40 and Corollary~44 of \cite{CST4}).  The cokernel is then spanned by $\iinfty$-trees, with relations coming from the defining relations of $\cT^\iinfty_{2k}$, with non-$\iinfty$ trees set to $0$: $J^\iinfty=(-J)^\iinfty, I^\iinfty=H^\iinfty+X^\iinfty$, and $2J^\iinfty=0$. Thus the cokernel is isomorphic to $\z\otimes \sL'_{k+1}$.

The odd order sequence is shown to be exact as follows:

Recall that the framing relations in $\widetilde{\cT}_{2k-1}$ are the image of $\Delta_{2k-1}\colon \cT_{k-1}\to \cT_{2k-1}$, as described in Section~\ref{sec:reduced-groups-obstruction-theory}.
The image of $\Delta_{2k-1}$ is $2$-torsion by AS relations and hence it factors through $\Z_2\otimes\cT_{k-1}$.
Thus we get an exact sequence as in the top of the diagram below, the middle exact sequence is Corollary 2.3 of \cite{L3}.
$$
\xymatrix{
(\z\otimes\cT_{k-1})/\Ker\Delta\ar@{>->}[r]^(.65){\Delta}\ar@{>->}[d]&\cT_{2k-1}\ar@{->>}[r]\ar[d]_\cong^{\eta'}&\widetilde{\cT}_{2k-1}\ar@{->>}[d]\\
\z^m\otimes\sL_k \ar@{=}[d] \ar@{>->}[r]^{sq} &{\sD}'_{2k-1} \ar@{>->}[d]\ar@{->>}[r]&{\sD}_{2k-1} \ar@{>->}[d]\\
\z\otimes\sL_1 \otimes\sL_k\ar@{>->}[r]^{sq} &\sL_1 \otimes\sL'_{2k}\ar@{->>}[r]&\sL_1 \otimes\sL_{2k}
}
$$
The map on the right is defined via the factorization $\widetilde{\cT}_{2k-1}\twoheadrightarrow {\cT}^\iinfty_{2k-1}\overset{\eta}{\longrightarrow} \sD_{2k-1}.$ So by definition the right-hand square commutes, and induces the left-hand vertical map. In fact, we claim that the induced map $\z\otimes\cT_{k-1}\to\z^m\otimes\sL_k$ factors as 
\[
\z\otimes\cT_{k-1}\overset{1\otimes\eta'}{\longrightarrow}\z\otimes{\sD}'_{k-1}\to\z^m\otimes\sL_k,
\]
 with the right hand map induced by $\sD'_{k-1}\to\Z^m\otimes\sL'_k\twoheadrightarrow\Z^m\otimes \sL_k$. To see this, let $t\in\cT_{k-1}$ and compute
\begin{eqnarray*}
\eta'\left(\Delta (1\otimes t)\right) &= &\eta'\left(\sum_{v\in t} \langle \ell(v),(T_v(t),T_v(t))\rangle\right)\\
&=&\sum_{v\in t} X_{\ell(v)}\otimes (T_v(t),T_v(t))\\
&=&\sum_{v\in t} X_\ell(v)\otimes sq(T_v(t))\\
&=&sq(1\otimes\eta')(1\otimes t)
\end{eqnarray*}
Now we claim that for all orders $k$, there is an exact sequence $$\z\otimes\sD'_{k-1}\to\z^m\otimes\sL_k\to\z\otimes \sL'_{k+1}\to 0.$$
This is clear if $k$ is odd, by tensoring the defining exact sequence for $\sD'_{k-1}$ with $\z$. If $k$ is even, then this follows since there is a surjection $\sD_{k-1}'\twoheadrightarrow \sD_{k-1}$ and $\sL'_{k+1} \cong\sL_{k+1}$. Therefore, the commutative diagram above supports a vertical short exact sequence on the left:
$$
\xymatrix{
&&\z\otimes\sL'_{k+1}\ar@{>-->}[d]\\
(\z\otimes\cT_{k-1})/\Ker\Delta\ar@{>->}[r]^(.65){\Delta}\ar@{>->}[d]&\cT_{2k-1}\ar@{->>}[r]\ar@{>->>}[d]^{\eta'}&\widetilde{\cT}_{2k-1}\ar@{->>}[d]\\
\z^m\otimes\sL_k\ar@{->>}[d]\ar@{>->}[r]&{\sD}'_{2k-1}\ar@{->>}[r]&{\sD}_{2k-1}\\
\z\otimes \sL'_{k+1}&&
}
$$
which gives us the indicated map on the right. Furthermore, $\eta\colon\cT^\iinfty_{2k-1}\to\sD_{2k-1}$ is an isomorphism by 
Theorem~\ref{thm:isomorphisms} below, so the right-hand exact sequence is precisely the exact sequence we're interested in.

\section{Summary of computations for the Whitney filtrations}\label{sec:master-diagrams-and-algebra}

In this section the commutative diagram of $\cT$- and $\W$-groups from Theorem~\ref{thm:exact-sequence} is extended by the relevant $\eta$- and $\mu$-maps to include exact sequences of $\sD$-groups (section~\ref{subsec:master-diagrams}). Then section~\ref{subsec:eta-isos} establishes some implications of the Levine Conjecture that were used earlier, including a proof of Proposition~\ref{prop:kerEta4k-2} from the introduction.

\subsection{The easy Master Diagram}\label{subsec:master-diagrams}
We have already extensively used the commutative diagram of Theorem~\ref{thm:exact-sequence} connecting the 4-term exact sequences for the various $\cT$- and $\W$-groups. Here and in the subsequent subsection we introduce an exact sequence of $\sD$-groups to complement these. For this we need to define two additional groups, $\widetilde{\sD}_{2k-1}$ and $\sD^\iinfty_{4k-2}$. The group $\widetilde{\sD}_{2k-1}$ is defined to be the quotient of $\sD'_{2k-1}$ by the image of the framing relations under $\eta'_{2k-1}$. The group $\sD^\iinfty_{4k-2}$, explained in Definition~\ref{def:D-infty-and-tilde}, is (non-canonically) isomorphic to $\sD_{4k-2}\oplus (\z\otimes \sL_k).$ 

Then the entire classification picture can be organized into {\em Master Diagrams} relating the various aspects of the story. There will be two such diagrams, each covering half of the cases,  
We present first the ``easier'' diagram, which combines three commutative triangles we have already seen, together with a new one involving $\widetilde{\sD}_{4k-1}$. This new triangle follows more-or-less by definition: since $\widetilde{R}_{4k-1}$ is an isomorphism, we can let $\widetilde{\mu}_{4k-1}:=\widetilde{\eta}_{4k-1}\circ \widetilde{R}_{4k-1}^{-1}$. The maps $\sD'_{4k}\to \sD_{4k}$ and $\widetilde{\sD}_{4k-1}\to \sD_{4k-1}$ are clear, and commutativity follows by definition of the various $\eta$ maps. The map $\sD_{4k}\to\widetilde{\sD}_{4k-1}$ can be defined so that it is induced by the map $\cT^\iinfty_{4k}\to\widetilde{\cT}_{4k-1}$.
\begin{thm}\label{thm:diagram 1}
 The following is a commutative diagram connecting three $4$-term exact sequences with triangles of isomorphisms.
\[
\xymatrix{ 
 \cT_{4k}\ar@{>->>}[dd]^(.6){\eta_{4k}'} \ar@{>->>}[dr]^{\widetilde{R}_{4k}} \ar@{>->}[rr] & &  \cT^\iinfty_{4k} \ar@{>->>}[dd]_(.6){\eta_{4k}}  \ar@{>->>}[dr]^{R^\iinfty_{4k}} \ar[rr] &&\widetilde{\cT}_{4k-1} \ar@{>->>}[dd]_(.6){\widetilde{\eta}_{4k-1}}  \ar@{>->>}[dr]^{\widetilde{R}_{4k-1}} \ar@{->>}[rr] &&\cT^\iinfty_{4k-1}\ar@{>->>}[dd]_(.6){\eta_{4k-1}}\ar@{>->>}[dr]^{R^\iinfty_{4k-1}}&
 \\
  & \W_{4k} \ar@{>->>}[dl]^{\mu_{4k}} \ar@{>->}[rr]|(.50)\ &&  \W^\iinfty_{4k}\ar @{>->>}[dl]^{\mu_{4k}}   \ar[rr]|(.485)\ && \W_{4k-1}\ar @{>->>}[dl]^{\widetilde{\mu}_{4k-1}}   \ar@{->>}[rr]|(.485)\ && \W^\iinfty_{4k-1}\ar@{>->>}[dl]^{\mu_{4k-1}} 
  \\
\sD'_{4k} \ar@{>->}[rr] &&  \sD_{4k} \ar[rr]&&  \widetilde{\sD}_{4k-1} \ar@{->>}[rr]&&\sD_{4k-1}&
}
\]


Moreover, the three horizontal sequences can each be split into two short exact sequences, with the term $\sK^\mu_{4k-1}\cong \z\otimes\sL_{2k+1}$ appearing in the middle as the cokernel of the left-hand maps and the kernel of the right-hand maps. 
\end{thm}

Theorem~\ref{thm:diagram 1} will be made precise and proven in the next subsection.

\subsection{The hard Master Diagram}\label{subsec:hard master-diagrams}

The other half of cases are covered by the following diagram. Here we notice that $\z\otimes \sL'_{2k}$ breaks the $\cT$-sequence into two exact sequences, but that the exact sequence at the bottom $\sD'_{4k-2}\to\sD_{4k-2}\to\z\otimes\sL_{2k}$ involves $\sL_{2k}$ and not $\sL'_{2k}$. Hence to make the corresponding 4-term $\sD$ sequence, we create $\sD^\iinfty_{4k-2}$ by a pullback diagram as pictured by the ``p.b."-labeled parallelogram. Then one can lift $\eta_{4k-2}$ to $\eta^{\iinfty}_{4k-2}$ which becomes an isomorphism (Theorem~\ref{thm:isomorphisms}).
We have seen that $\mu_{4k-2}$ can be defined with image in $\sD_{4k-2}$. Lifting it to $\mu^\iinfty_{4k-2}$ is equivalent to Conjecture~\ref{conj:Arf-k}, since that conjecture is equivalent to $R^\iinfty_{4k-2}$ being an isomorphism. If the lifted maps $\mu^\iinfty_{4k-2}$ exist, they would automatically be isomorphisms.
\begin{thm}\label{thm:diagram 2}
The following diagram commutes.
\[
\xymatrix{ 
 \cT_{4k-2}\ar@{>->>}[dd]^(.6){\eta_{4k-2}}  \ar@{>->>}[dr]^{\widetilde{R}_{4k-2}} \ar@{>->}[rr] &&  \cT^\iinfty_{4k-2} \ar@{>->>}[dd]_(.6){\eta^\iinfty_{4k-2}}  \ar@{->>}[dr]^{R^\iinfty_{4k-2}} \ar[rr] && \widetilde\cT_{4k-3} \ar@{>->>}[dd]_(.6){\widetilde\eta_{4k-3}} \ar@{->>}[dr]^{\widetilde{R}_{4k-3}} \ar@{->>}[rr] &&  \cT^\iinfty_{4k-3}\ar@{>->>}[dd]_(.6){\eta_{4k-3}} \ar@{>->>}[dr]^{R^\iinfty_{4k-3}}\\
& \W_{4k-2} \ar@{>->>}[dl]^{\mu_{4k-2}} \ar@{>->}[rr]|(.48)\ &&  \W^\iinfty_{4k-2}\ar @{->>}[ddl]^(.3){\mu_{4k-2}}|(.5){\phantom{X}}|(.66){\phantom{X}}  \ar@{-->}[dl]|{\mu^\iinfty_{4k-2} ?} \ar[rr]|(.52)\ && \W_{4k-3}  \ar@{-->}[dl]|{\widetilde \mu_{4k-3} ?}  \ar@{->>}[rr]|\  &&  \W^\iinfty_{4k-3} \ar@{>->>}[dl]^{\mu_{4k-3}}\\
\sD'_{4k-2} \ar@{=}[d] \ar@{>->}[rr] &&  \sD^\iinfty_{4k-2} \ar@{->>}[d]\ar@{->>}[dr]\ar[rr] && \widetilde {\sD}_{4k-3}  \ar@{->>}[rr] &&  \sD_{4k-3}\\
\sD'_{4k-2} \ar@{>->}[rr]\ &&  \sD_{4k-2} \ar@{}[r]|(.42)*[F-,]{\, p.b.\,} \ar@{->>}[dr] & \sL'_{2k}\otimes\z \ar@{->>}[d] \ar@{>->}[ur]&&& \\
&&&  \sL_{2k} \otimes\z&&&
 } \]

As in the first diagram, the three 4-term horizontal sequences are exact. The horizontal $\cT$- and 
$\sD$-sequences actually break into two short exact sequences with the groups $\z\otimes \sL'_{2k}$ in the middle. For the horizontal $\W$-sequence, the group $\sK^\mu_{4k-3}$ sits in the middle, which we conjecture to be isomorphic to $\z\otimes\sL'_{2k}$. At the bottom, there are two short exact sequences, both starting with $\sD'_{4k-2}$ because the diagonal square is a pullback (p.b.).
\end{thm}

In the rest of this subsection we will collect the remaining algebraic definitions and proofs needed to set up these master diagrams.

First we need the definitions of $\widetilde{\sD}_{2k-1}$ and $\sD^\iinfty_{4k-2}$:

\begin{defn}\label{def:D-infty-and-tilde}\hfill
 The groups $\sD^\iinfty_{4k-2}$ (as well as the maps $s\ell_{4k-2}',p_{2k},sq^\iinfty$) are defined by the pullback diagram
\begin{equation}\tag{$\sD^\iinfty$}
\xymatrix{
\z\otimes \sL_{k} \ar@{>->}[r]^{sq^\iinfty} \ar@{=}[d] & \sD^\iinfty_{4k-2}\ar@{->>}[d]^{s\ell_{4k-2}'}\ar@{->>}[r]^{p_{4k-2}}&\sD_{4k-2}\ar@{->>}[d]^{s\ell_{4k-2}}\\
\z\otimes \sL_{k} \ar@{>->}[r]^{sq} & \Z_2\otimes \sL'_{2k}\ar@{->>}[r]^p &\Z_2\otimes \sL_{2k}
}
\end{equation}

To define the group $\widetilde{\sf D}_{2k-1}$, we note that the homomorphism $\id\otimes sq\colon \z^m \otimes\sL_{k}\to \z^m\otimes\sL_{2k}$ restricts to a homomorphism $\z\otimes\sD'_{k-1}\to \sD'_{2k-1}$, because elements of the form $X\otimes[Y,Y]$ are in the kernel of the bracketing map.
 Now $\widetilde{\sf D}_{2k-1}$ is defined to be the quotient  of ${\sf D}'_{2k-1}$ by the image of this homomorphism. It is not hard to show that the image of this homomorphism is equal to
 the image  $\eta'_{2k-1}\circ\Delta(\z\otimes \cT_{k-1})$.  Thus there is an induced map $\widetilde{\eta}_{2k-1}\colon \widetilde{\cT}_{2k-1}\to \widetilde{\sD}_{2k-1}$.
\end{defn}

\begin{lem}\label{lem:lifted eta}
There is a canonical lift $\eta_{4k-2}^\iinfty$ of $\eta_{4k-2}$ 
\[
\xymatrix{
& \sD^\iinfty_{4k-2}\ar@{->>}[d]^{p_{4k-2}}\\
\cT^\iinfty_{4k-2} \ar@{->>}[r]^{\eta_{4k-2}} \ar@{->>}[ru]^{\eta^\iinfty_{4k-2}} & \sD_{4k-2} 
}\]
such that $\eta_{4k-2}^\iinfty((J,J)^\iinfty) = sq^\iinfty(1\otimes J)$ for all rooted trees $J\in \sL_k$.
\end{lem}
\begin{proof}
To construct $\eta^\iinfty_{4k-2}$ it suffices to observe that we have a commutative diagram
$$
\xymatrix{\cT^\iinfty_{4k-2}\ar[r]^{\eta_{4k-2}}\ar[d]&{\sD}_{4k-2}\ar[d]\\
\z\otimes {\sL}'_{2k}\ar[r]^p&\z\otimes{\sL}_{2k}
}
$$
which gives rise to a map to the pullback $\sD_{4k-2}^\iinfty$ from  diagram ${\sf D}^\iinfty$ in Definition~\ref{def:D-infty-and-tilde} above. To calculate $\eta^\iinfty_{4k-2}((J,J)^\iinfty)$ notice that $\eta_{4k-2}((J,J)^\iinfty)=0$ and $s\ell^\prime_{4k-2}((J,J)^\iinfty)=1\otimes (J,J)=sq(1\otimes J)$. So $\eta^\iinfty_{4k-2}((J,J)^\iinfty)=sq^\iinfty(1\otimes J)$. 
\end{proof}

\begin{rem}
The superscripts in our $\eta$-maps reflect those of their {\em target} groups. 
\end{rem}

\subsection{Variations on the Levine conjecture}\label{subsec:eta-isos}
The next theorem establishes that various versions of the $\eta$ map are isomorphisms as a result of the Levine Conjecture (Theorem~\ref{thm:LC} above), and gives as a corollary the characterization of the
kernel of $\eta_{4k-2}$ stated in Proposition~\ref{prop:kerEta4k-2} of the introduction.

\begin{thm}\label{thm:isomorphisms} The following maps  are all isomorphisms: 
\begin{enumerate}
\item\label{item:eta-tilde} $\widetilde{\eta}_{2k-1}\colon\widetilde{\cT}_{2k-1}\to \widetilde{\sD}_{2k-1}$
\item\label{item:eta-odd} $\eta_{2k-1}\colon\cT^\iinfty_{2k-1}\to\sD_{2k-1}$
\item\label{item:eta-4k} $\eta_{4k}\colon\cT^\iinfty_{4k}\to\sD_{4k}$
\item\label{item:eta-infty-4k-2} $\eta^\iinfty_{4k-2}\colon \cT^\iinfty_{4k-2}\to\sD^\iinfty_{4k-2}$
\end{enumerate}

\end{thm}

\begin{proof}
(\ref{item:eta-tilde}) follows from Theorem~\ref{thm:LC} and the definition of $\widetilde{\eta}_{2k-1}$. 

To show~(\ref{item:eta-odd}), consider the following diagram (commutative by definition of $\eta$):
$$\xymatrix{
0\ar[r]&\operatorname{span}\{\langle i,(J,J)\rangle\}\ar[r]\ar[d]^{\eta'}&\mathcal T_{2k-1}\ar[r]\ar[d]_\cong^{\eta'}&\mathcal T_{2k-1}^\iinfty\ar[r]\ar[d]^{\eta}&0\\
0\ar[r]&\mathbb Z_2^m\otimes {\sf L}_k\ar[r]^{sq}&\sD'_{2k-1}\ar[r]&\sD_{2k-1}\ar[r]&0\\
}
$$
The bottom row is exact by Corollary 2.3 of \cite{L3}. The top row is exact by definition and the middle map $\eta'$ is an isomorphism by Theorem~\ref{thm:LC}. This implies the left-hand restriction $\eta'$ is one to one, and it is onto since $\eta'(\langle i,(J,J)\rangle)=X_i\otimes J$. Therefore, the right-hand map $\eta$ is an isomorphism by the 5-lemma. 

For (\ref{item:eta-4k}), consider the following diagram (commutative by Lemma~\ref{lem:cd})
$$\xymatrix{
0\ar[r]&\cT_{4k}\ar[r]\ar[d]_\cong^{\eta'_{4k}}&\cT_{4k}^\iinfty\ar[r]\ar[d]^{\eta_{4k}}&\z\otimes{\sL}'_{2k+1}\ar[r]\ar[d]^p_{\cong}&0\\
0\ar[r]&{\sD}'_{4k}\ar[r]&{\sD}_{4k}\ar[r]^{s\ell_{4k}\,\,\,\,\,\,\,\,\,\,\,\,}&\z\otimes{\sL}_{2k+1}\ar[r]&0
}
$$
The bottom horizontal sequence is exact by Definition~\ref{def:sl-maps}. The top one is part of Theorem~\ref{thm:exact-sequence}, proven in Section~\ref{subsec:proof-thm-exact-sequences}.
Since ${\sL}_{2k+1}\cong {\sL}'_{2k+1}$ it follows that $\eta_{4k}$ is an isomorphism.

For (\ref{item:eta-infty-4k-2}), we note that
Diagram $\sD^\iinfty$ and Lemma~\ref{lem:lifted eta} imply a commutative diagram:
$$
\xymatrix{
0\ar[r]&\langle(J,J)^\iinfty\rangle\ar[r]\ar@{->>}[d]&\mathcal T^{\iinfty}_{4k-2}\ar[r]^{\eta_{4k-2}}\ar[d]_{\eta^\iinfty_{4k-2}}&{\sf D}_{4k-2}\ar[r]\ar@{=}[d]&0\\
0\ar[r]&\z\otimes{\sf L}_{k}\ar[r]^(.6){sq^\iinfty_{4k-2}}&{\sD}^\iinfty_{4k-2}\ar[r]^{p_{4k-2}}&{\sD}_{4k-2}\ar[r]&0
}
$$
where the vertical left hand map sends $(J,J)^\iinfty$ to $1\otimes J$.
The right-hand square commutes by the main commutative triangle of Lemma~\ref{lem:lifted eta}, whereas the left square commutes by the calculation $\eta^\iinfty_{4k-2}((J,J)^\iinfty)=sq^\iinfty_{4k-2}(1\otimes J)$, as in Lemma~\ref{lem:cd}. 
Thus $\eta^\iinfty_{4k-2}$ is an isomorphism by the 5-lemma.
\end{proof}

As a corollary to the above proof of (\ref{item:eta-infty-4k-2}) we get the following characterization of $\Ker(\eta_{4k-2}\colon \cT^\iinfty_{4k-2}\to \sD_{4k-2})$ which is equivalent to 
Proposition~\ref{prop:kerEta4k-2} in the introduction:
\begin{cor}\label{cor:kernel-eta 4k-2}
The kernel of the homomorphism $\eta_{4k-2}\colon \cT^\iinfty_{4k-2}\to \sD_{4k-2}$ is isomorphic to $\z\otimes{\sf L}_k$, with $(J,J)^\iinfty\mapsto 1\otimes J$.
\end{cor}
\begin{proof}
That there is an isomorphism follows since $\eta^\iinfty_{4k-2}$ is an isomorphism which is a lift of $\eta_{4k-2}$ and from the exact sequence $0\to\z\otimes\sL_{k}\to\sD^\iinfty_{4k-2}\to \sD_{4k-2}\to 0.$ The fact that $(J,J)^\iinfty\mapsto 1\otimes J$ follows from Lemma~\ref{lem:lifted eta}
\end{proof}

The following lemma was used in the proof of Theorem~\ref{thm:isomorphisms}.
\begin{lem}\label{lem:cd}
Sending $J^\iinfty$ to $1\otimes J$ gives a commutative diagram
$$\xymatrix{
\cT_{2k}^\iinfty\ar[r]\ar[d]^{\eta_{2k}}&\z\otimes{\sL}'_{k+1}\ar[d]^p\\
{\sD}_{2k}\ar[r]^{s\ell_{2k}\,\,\,\,\,\,\,\,\,\,\,\,}&\z\otimes{\sL}_{k+1}
}
$$
\end{lem}
\begin{proof}
First, we need a better handle on the map $sl_{2k}\colon\sD_{2k}\to \mathbb Z_2\otimes \sL_{k+1}$. Let $Z\in\sD_{2k}$ and pick a lift $Z'\in \sL_1\otimes\sL'_{2k+1}$. Tracing through the snake lemma diagram in Definition~\ref{def:sl-maps}, one sees that the bracket of $Z'$ is a sum of commutators $[J_i,J_i]$, and that $sl_{2k}(Z)=\sum 1\otimes J_i$.

Consider a tree $t\in\mathcal T^\iinfty_{2k}$ which maps to zero in $\mathbb Z_2\otimes {\sf L}'_{k+1}$ by definition.
Mapping $t$ down by $\eta_{2k}$, we end up in $\sD_{2k}'$ and hence in the kernel of $s\ell_{2k}$.

Now consider $J^\iinfty\in\mathcal T^\iinfty_{2k}$. Then $\eta_{2k}(J^\iinfty)$ doubles $J$ to $\langle J,J\rangle$ and sums over putting a root at all of the leaves of one copy of $J$. The result represents an element in $\sL_1\otimes \sL_{k+1}$.
Calculating the bracket has the effect of summing over putting a root near all of the leaves on one copy of $J$ in $\langle J,J\rangle$, which modulo IHX is equal to $(J,J)$. To see this requires pushing the central root of $(J,J)$ to one side using IHX relations.

Thus $sl_{2k}(\eta_{2k}(J^\iinfty))=1\otimes J$ which is equal to mapping $J^\iinfty$ right and then down.
\end{proof}

\end{document}